\documentclass[10pt,reqno]{amsart}
\usepackage[10pt,nona4]{optional}

\usepackage{comment}
\includecomment{nonpictures}

\usepackage{dsfont,amsmath,amsfonts,amscd,amssymb,graphicx,mathrsfs,eufrak,upgreek,enumitem}
\usepackage[all,arc,curve,color,frame]{xy}
\usepackage[usenames]{color}
\usepackage{setspace}
\usepackage{array,multirow,booktabs,longtable}

\setlist[enumerate]{format=\normalfont}

\newcommand{\marginparstretch}{0.6}
\let\oldmarginpar\marginpar
\renewcommand\marginpar[1]{\-\oldmarginpar[\framebox{\setstretch{\marginparstretch}\begin{minipage}{\marginparwidth}{\raggedleft\tiny #1}\end{minipage}}]{\framebox{\setstretch{\marginparstretch}\begin{minipage}{\marginparwidth}{\raggedright\tiny #1}\end{minipage}}}}

\usepackage[colorlinks]{hyperref}
\usepackage{tikz,mathrsfs}
\usetikzlibrary{arrows,decorations.pathmorphing,decorations.pathreplacing,positioning,shapes.geometric,shapes.misc,decorations.markings,decorations.fractals,calc,patterns,intersections}

\tikzset{
        cvertex/.style={circle,draw=black,inner sep=1pt,outer sep=3pt},
        vertex/.style={circle,fill=black,inner sep=1pt,outer sep=3pt},
        DBs/.style={circle,draw=black,circle,fill=black,inner sep=0pt, minimum size=3pt},
        DB/.style={circle,draw=black,circle,fill=black,inner sep=0pt, minimum size=4pt},
         DWs/.style={circle,draw=black,circle,fill=white,inner sep=0pt, minimum size=3pt},
         DWds/.style={circle,draw=black,densely dotted,circle,fill=white,inner sep=0pt, minimum size=3pt},
        DW/.style={circle,draw=black,inner sep=0pt, minimum size=4pt},
        tvertex/.style={inner sep=1pt,font=\scriptsize},
        gap/.style={inner sep=0.5pt,fill=white},
        Ggap/.style={inner sep=0.5pt,fill=green!40!black!20}}

\newcommand{\rotateRPY}[3]
{   \pgfmathsetmacro{\rollangle}{#1}
    \pgfmathsetmacro{\pitchangle}{#2}
    \pgfmathsetmacro{\yawangle}{#3}

    \pgfmathsetmacro{\newxx}{cos(\yawangle)*cos(\pitchangle)}
    \pgfmathsetmacro{\newxy}{sin(\yawangle)*cos(\pitchangle)}
    \pgfmathsetmacro{\newxz}{-sin(\pitchangle)}
    \path (\newxx,\newxy,\newxz);
    \pgfgetlastxy{\nxx}{\nxy};

    \pgfmathsetmacro{\newyx}{cos(\yawangle)*sin(\pitchangle)*sin(\rollangle)-sin(\yawangle)*cos(\rollangle)}
    \pgfmathsetmacro{\newyy}{sin(\yawangle)*sin(\pitchangle)*sin(\rollangle)+ cos(\yawangle)*cos(\rollangle)}
    \pgfmathsetmacro{\newyz}{cos(\pitchangle)*sin(\rollangle)}
    \path (\newyx,\newyy,\newyz);
    \pgfgetlastxy{\nyx}{\nyy};

    \pgfmathsetmacro{\newzx}{cos(\yawangle)*sin(\pitchangle)*cos(\rollangle)+ sin(\yawangle)*sin(\rollangle)}
    \pgfmathsetmacro{\newzy}{sin(\yawangle)*sin(\pitchangle)*cos(\rollangle)-cos(\yawangle)*sin(\rollangle)}
    \pgfmathsetmacro{\newzz}{cos(\pitchangle)*cos(\rollangle)}
    \path (\newzx,\newzy,\newzz);
    \pgfgetlastxy{\nzx}{\nzy};
}

\tikzset{RPY/.style={x={(\nxx,\nxy)},y={(\nyx,\nyy)},z={(\nzx,\nzy)}}}

\tikzstyle{mybox} = [draw=black, fill=blue!10, very thick,
    rectangle, rounded corners, inner sep=10pt, inner ysep=20pt]
\tikzstyle{boxtitle} =[fill=blue!50, text=white,rectangle,rounded corners]

\newtheorem{thm}{Theorem}[section]
\newtheorem{prop}[thm]{Proposition}
\newtheorem{lemma}[thm]{Lemma}
\newtheorem{defin}[thm]{Definition}
\newtheorem{cor}[thm]{Corollary}

\theoremstyle{definition} 

\newtheorem{example}[thm]{Example}
\newtheorem{setup}[thm]{Setup}

\newtheorem{remark}[thm]{Remark}
\newtheorem{conj}[thm]{Conjecture}

\newtheorem{notation}[thm]{Notation}

\numberwithin{equation}{section}

\newcounter{tempenum}

\newcounter{enumeratenoindentcounter}

\newcommand{\m}{\mathfrak{m}}
\newcommand{\n}{\mathfrak{n}}
\newcommand{\p}{\mathfrak{p}}

\renewcommand{\t}[1]{\textnormal{#1}}

\def\op{\mathop{\rm op}\nolimits}
\def\rest{\mathop{\rm res}\nolimits}

\def\CM{\mathop{\rm CM}\nolimits}
\def\uCM{\mathop{\underline{\rm CM}}\nolimits}

\def\depth{\mathop{\rm depth}\nolimits}

\def\mod{\mathop{\rm mod}\nolimits}

\def\coh{\mathop{\rm coh}\nolimits}
\def\Qcoh{\mathop{\rm Qcoh}\nolimits}
\def\Mod{\mathop{\rm Mod}\nolimits}
\def\Rad{\mathop{\rm Rad}\nolimits}
\def\refl{\mathop{\rm ref}\nolimits}

\def\pd{\mathop{\rm pd}\nolimits}
\def\id{\mathop{\rm inj.dim}\nolimits}

\def\Hom{\mathop{\rm Hom}\nolimits}

\def\RHom{\mathop{\rm {\bf R}Hom}\nolimits}

\def\End{\mathop{\rm End}\nolimits}
\def\Ext{\mathop{\rm Ext}\nolimits}

\def\add{\mathop{\rm add}\nolimits}

\def\Cone{\mathop{\rm Cone}\nolimits}
\def\Ker{\mathop{\rm Ker}\nolimits}

\def\rank{\mathop{\rm rank}\nolimits}

\def\Sing{\mathop{\rm Sing}\nolimits}
\def\Supp{\mathop{\rm Supp}\nolimits}

\def\Aut{\mathop{\rm Aut}\nolimits}
\def\Spec{\mathop{\rm Spec}\nolimits}

\def\Max{\mathop{\rm Max}\nolimits}

\def\Perf{\mathop{\mathrm{per}}\nolimits}

\def\D{\mathop{\rm{D}^{}}\nolimits}
\def\Dsg{\mathop{\rm{D}_{\sf sg}}\nolimits}
\def\Dm{\mathop{\rm{D}^-}\nolimits}
\def\Db{\mathop{\rm{D}^b}\nolimits}

\def\flop{{\sf{F}}}

\def\Id{\mathop{\rm{Id}}\nolimits}

\newcommand{\K}{\mathop{{}_{}\mathbb{C}}\nolimits}
\newcommand{\rk}{\mathop{\mathsf{rk}}\nolimits}

\newcommand{\con}{\mathrm{con}}
\newcommand{\fib}{0}
\newcommand{\CA}{\mathrm{A}_{\con}}
\newcommand{\CAR}{\mathrm{A}_{\fib}}

\newcommand{\AB}{\mathrm{A}}
\newcommand{\ABplus}{\mathrm{A}\hspace{-.1em}^{+}}
\newcommand{\BB}{\mathrm{B}}

\def\RA{\mathop{\rm RA}\nolimits}
\def\LA{\mathop{\rm LA}\nolimits}

\def\redu{\mathop{\rm red}\nolimits}

\def\noniso{\mathop{\rm Ram}\nolimits}

\def\Rf{{\rm\bf R}f}

\def\Ri{{\rm\bf R}i}

\def\RHom{{\rm{\bf R}Hom}}

\def\RsHom{{\bf R}\mathcal{H}om}
\newcommand\RDerived[1]{{\rm\bf R}{#1}}
\newcommand\RDerivedi[2]{{\rm\bf R}^{#1}{#2}}
\newcommand\Rfi[1]{{\rm\bf R}^{#1}f}

\newcommand\LDerived[1]{{\rm\bf L}{#1}}

\newcommand\FM[1]{\operatorname{\sf{FM}}({#1})}

\newcommand\nuJ{\upnu_{\hspace{-1pt} J}\hspace{0.5pt}}
\newcommand\fundgp{\uppi_{\hspace{0.5pt}1}\hspace{-0.5pt}}

\newcommand{\cA}{\mathcal{A}}
\newcommand{\cB}{\mathcal{B}}
\newcommand{\cC}{\mathcal{C}}
\newcommand{\cD}{\mathcal{D}}
\newcommand{\cE}{\mathcal{E}}
\newcommand{\cF}{\mathcal{F}}
\newcommand{\cG}{\mathcal{G}}
\newcommand{\cH}{\mathcal{H}}

\newcommand{\cL}{\mathcal{L}}
\newcommand{\cM}{\mathcal{M}}
\newcommand{\cN}{\mathcal{N}}
\newcommand{\cO}{\mathcal{O}}
\newcommand{\cP}{\mathcal{P}}
\newcommand{\cQ}{\mathcal{Q}}
\newcommand{\cR}{\mathcal{R}}
\newcommand{\cS}{\mathcal{S}}

\newcommand{\cV}{\mathcal{V}}
\newcommand{\cW}{\mathcal{W}}
\newcommand{\cX}{\mathcal{X}}

\newcommand{\DF}{\mathds{D}_\mathsf{Flop}}

\newcommand{\Per}{{}^{0}\mathrm{Per}}
\newcommand{\PerOne}{{}^{-1}\mathrm{Per}}

\newcommand{\TwistVsFlop}{\Uppsi}

\newcommand{\boldb}{\mathbf{b_J}}

\newcommand\dualizing{\mathbb{D}}
\newcommand\Tcon{{J\sf Twist}} 
\newcommand\Tcondual{{\dualizing\, \Tcon\, \dualizing^{-1}}} 
\newcommand\Tconconj{{J\sf Twist}_\conjlb}
\newcommand\Tfib{{F\sf Twist}}
\newcommand\Tfibdual{{\dualizing\, \Tfib\, \dualizing^{-1}}}

\newcommand\Tconp{{J\sf Twist}_{p}} 
\newcommand\Tfibp{{F\sf Twist}_p}

\newcommand\Tconpiinv[1]{{J_{{#1}}\!{\sf Twist}{}_{{p_#1}}^*}} 
\newcommand\Tconpip[1]{{J_{{#1}}\!{\sf Twist}{}_{{p}}}} 

\newcommand\Tfibpiinv[1]{{F{\sf Twist}{}_{{p_#1}}^*}} 
\newcommand\indx{j}
\newcommand\indxset{J}
\newcommand\otherindx{k}
\newcommand\functcong{\cong}

\newcommand\clocCon{\upvarphi} 
\newcommand\conjlb{\cF} 

\newcommand\diag{\Updelta}
\newlength\tempWidth

\newcommand\FMproja{\uppi}
\newcommand\FMprojb{\uppi}
\newcommand\projenv{{\bar{X}}}
\newcommand\projenvtwist{{\bar{F}}}

\newcommand{\defColor}{gray!50}

  \def\size{5}
  \def\bulge{1/3}
\def\gridShift{-4}
\def\bend{0.1}
\def\len{0.8}
\def\dotsize{0.1}

\newcommand{\defCurveCoords}[3]{
  (#1+\bend,-#3,#2) .. controls (#1-\bend,-#3/3,#2) and (#1-\bend, #3/3,#2) .. (#1+\bend,#3,#2)
}

\newcommand{\defDrawDot}[2]{
    \draw[fill] (#1+\dotsize,\gridShift,#2+\dotsize) -- (#1+\dotsize,\gridShift,#2-\dotsize) -- (#1-\dotsize,\gridShift,#2-\dotsize) -- (#1-\dotsize,\gridShift,#2+\dotsize) -- cycle
    }

\newcommand{\defCurve}[3]{
  \draw[#3] \defCurveCoords{#1}{#2}{\len}
}

\newcommand{\defThreeFold}[3]{
  \draw[#3] (#1+0.5+\size,#2-0.3) .. controls (#1+0.5+2*\size*\bulge,#2-0.3+\size/2) and (#1+0.5-\size*\bulge,#2-0.3+\size/2) .. (#1+0.5-\size,#2-0.3) .. controls (#1+0.5-5*\size*\bulge,#2-0.3-\size/2) and (#1+0.5+4*\size*\bulge,#2-0.3-\size/2) .. (#1+0.5+\size,#2-0.3)
}

\newcommand{\defUniversalCurves}[4]{
  \def\x{#1} \def\y{#2}
  \def\numx{#3} \def\numy{#4}
  \def\zero{0}
  \foreach \i in {0,...,#4} {
    \ifx\i\zero \defCurve{#1+\i}{#2}{black} \else \defCurve{#1+\i}{#2}{\defColor} \fi;
    \defDrawDot{#1+\i}{#2};
  };
  
  \ifx\numx\zero {} \else {
  \foreach \j in {1,...,#3} {
    \defCurve{#1}{#2+\j}{\defColor};
    \defDrawDot{#1}{#2+\j};
    }
    } \fi
}

\begin{document}
\title{\textsc{Twists and Braids for General 3-fold Flops}}
\author{Will Donovan}
\address{Will Donovan, Kavli Institute for the Physics and Mathematics of the Universe (WPI), The University of Tokyo Institutes for Advanced Study, Kashiwa, Chiba 277-8583, Japan.}
\email{will.donovan@ipmu.jp}
\author{Michael Wemyss}
\address{Michael Wemyss, The Maxwell Institute, School of Mathematics, James Clerk Maxwell Building, The King's Buildings, Peter Guthrie Tait Road, Edinburgh, EH9 3FD, UK.}
\email{wemyss.m@googlemail.com}
\begin{abstract}
Given a quasi-projective $3$-fold $X$ with only Gorenstein terminal singularities, we prove that the flop functors beginning at $X$ satisfy higher degree braid relations, with the combinatorics controlled by a real hyperplane arrangement $\cH$. This leads to a general theory, incorporating known special cases with degree $3$ braid relations, in which we show that higher degree relations can occur even for two smooth rational curves meeting at a point. This theory yields an action of the fundamental group of the complexified complement $\fundgp(\mathbb{C}^n\backslash \cH_\mathbb{C})$ on the derived category of $X$, for any such $3$-fold that admits individually floppable curves. We also construct such an action in the more general case where individual curves may flop analytically, but not algebraically, and furthermore we lift the action to a form of affine pure braid group under the additional assumption that $X$ is $\mathds{Q}$-factorial.

Along the way, we produce two new types of derived autoequivalences. One uses commutative deformations of the scheme-theoretic fibre of a flopping contraction, and the other uses noncommutative deformations of the fibre with reduced scheme structure, generalising constructions of Toda and the authors \cite{Toda,DW1} which considered only the case when the flopping locus is irreducible. For type $A$ flops of irreducible curves, we show that the two autoequivalences are related, but that in other cases they are very different, with the noncommutative twist being linked to birational geometry via the Bridgeland--Chen \cite{Bridgeland, Chen} flop--flop functor.
\end{abstract}

\subjclass[2010]{Primary 14J30; Secondary 14E30, 14F05, 16S38, 18E30, 20F36, 52C35}

\thanks{The first author was supported by World Premier International Research Center Initiative (WPI Initiative), MEXT, Japan; by EPSRC grant~EP/G007632/1; and by the Hausdorff Research Institute for Mathematics, Bonn. The second author was supported by EPSRC grant~EP/K021400/1.}
\maketitle
\parindent 20pt
\parskip 0pt

\tableofcontents

\section{Introduction}

The derived category, through its autoequivalence group and associated Bridgeland stability manifold, is one of the fundamental tools that enriches and furthers our understanding of birational geometry, especially in low dimension.  Amongst other things it is widely expected, and proved in many cases, that the derived category allows us to run the minimal model program \cite{BayerMacri,BO,Bridgeland,Chen,TodaExtremal, HomMMP} and track minimal models \cite{CI,TodaResPub,HomMMP}, to illuminate new and known symmetries \cite{Donovan,DonSeg1,HLS,ST,Toda}, and to understand wall crossing phenomena \cite{CI,Calabrese,Nagao}, in both a conceptually and computationally easier manner. 

Consider a general algebraic $3$-fold flopping contraction $f\colon X\to X_{\con}$, where $X$ has only Gorenstein terminal singularities. It is known that there are derived equivalences between $X$ and its flops \cite{Bridgeland,Chen}, and in this paper:
\begin{enumerate}
\item We prove that the flop functors satisfy higher degree braid relations, with combinatorics controlled by the real hyperplane arrangement $\cH$ appearing in \cite[\S5--\S7]{HomMMP}, which need not be Coxeter. We then show that the derived category of $X$ carries an action of the fundamental group of the complexified complement of $\cH$, which should be thought of as a natural generalisation of a pure braid group.  We obtain this new group action via topological methods, without knowledge of a group presentation.  
\item We give a geometric description of generators of this action using noncommutative deformations of multiple curves.  Subsets $J$ of the curves in the fibres of $f$ need not flop algebraically, however for every subset we produce a twist autoequivalence, which we call the $J$-twist.  We prove that the $J$-twist is inverse to a known flop--flop functor in the case when the latter exists, namely when the union of curves in $J$ flop algebraically.  However the $J$-twist exists in full generality,  and the ability to vary $J$ regardless allows us to produce many more derived symmetries.
\item We also produce another new autoequivalence, using (commutative) deformations of the scheme-theoretic exceptional fibre, which we call the fibre twist.  We explain why both the $J$-twist and the fibre twist must be understood in order to construct affine braid group actions on the derived category of $X$.
\end{enumerate}

We remark that in special cases, braiding of flop functors is already known.  In certain toric cases studied by Segal and the first author \cite{DonSeg2}, and in other situations studied by Szendr\H{o}i \cite{Szendroi}, degree $3$ braid relations $\flop_1\circ\flop_2\circ\flop_1\functcong\flop_2\circ\flop_1\circ\flop_2$ exist between flops, which correspond to analogous degree $3$ relations between Seidel--Thomas twists on surfaces \cite{ST}.

\subsection{Motivation}\label{mot ex}
Before describing our results in detail, we sketch why the general theory requires higher degree relations. Consider a contraction $f\colon U\to\Spec R$, where $R$ is an isolated $3$-fold cDV singularity, $U$ is a minimal model, and $f$ contracts precisely two intersecting irreducible curves.  There exist examples where a generic hyperplane section $g\in R$ cuts to an ADE surface singularity $R/g$ of Type $E_6$, and $U$ cuts to the partial crepant resolution depicted below. Using notation from \cite{Katz}, this is described by the given marked Dynkin diagram, by blowing down the curves in the minimal resolution of $R/g$ corresponding to the black vertices.
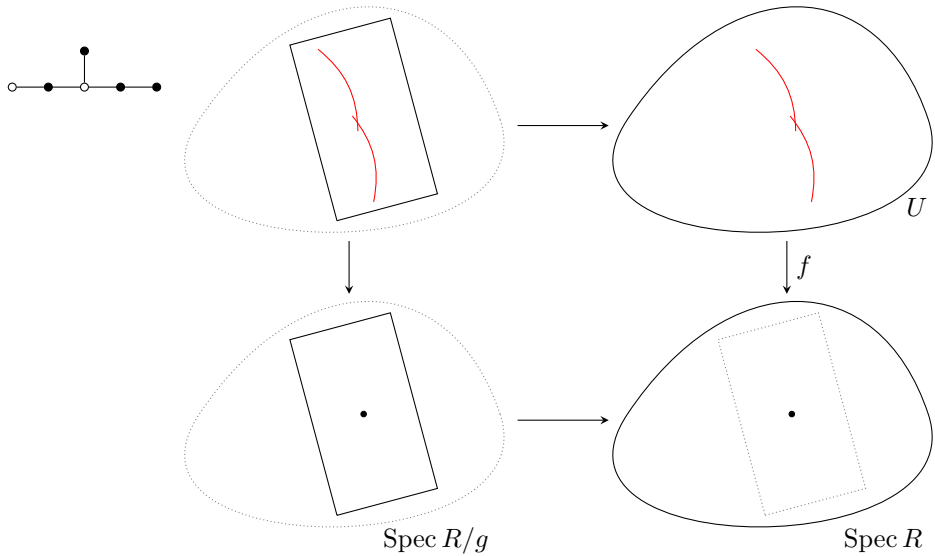
\def\overallscaling{0.8}
\def\internalscaling{1.6}
\def\basescaling{2}
\def\layoutstretch{2}
\def\curvetilt{-70}
\def\cuttilt{15}
\def\cutshift{0.1}
\def\cutscaling{0.9}
\def\toplineshift{0.9}
\def\bottomlineshift{0}
\def\leftshift{1.75}
\def\Dynkinscale{1}
\[
\begin{tikzpicture}[>=stealth,scale=\overallscaling,xscale=1/2] 
\defThreeFold{\layoutstretch*\internalscaling*-0.5}{\internalscaling*2.5+\toplineshift}{black};
\node (a1) at (\layoutstretch*\internalscaling*-0.25,\internalscaling*2.25+\toplineshift) 
{\begin{tikzpicture} [xscale=\cutscaling*\internalscaling*0.75,yscale=\cutscaling*\internalscaling*0.75,bend angle=25, looseness=1,transform shape, rotate=\curvetilt]
\draw[red] (0,0,0) to [bend left=25] (1,0,-0.3);
\draw[red] (0.8,0,-0.3) to [bend left=25] (2,0,0);
\end{tikzpicture}
};
\node (a2) at (\layoutstretch*\internalscaling*-4.25-\leftshift,\internalscaling*2.25+\toplineshift) 
{\begin{tikzpicture} [xscale=\cutscaling*\internalscaling*0.75,yscale=\cutscaling*\internalscaling*0.75,bend angle=25, looseness=1,transform shape, rotate=\curvetilt]
\draw[red] (0,0,0) to [bend left=25] (1,0,-0.3);
\draw[red] (0.8,0,-0.3) to [bend left=25] (2,0,0);
\end{tikzpicture}
};
    \defThreeFold{\layoutstretch*\internalscaling*-4.4-\leftshift}{\internalscaling*2.5+\toplineshift}{densely dotted, gray};
    \node (a2) at (\layoutstretch*\internalscaling*-4.1-\leftshift,\internalscaling*2.25+\toplineshift+\cutshift) 
{\begin{tikzpicture}[xscale=\cutscaling*\internalscaling*1.2,yscale=\cutscaling*\internalscaling*1.2,transform shape, rotate=\cuttilt]
\draw[black] (0,0)--(0.8,0)--(0.8,1.4)--(0,1.4)--(0,0);
\end{tikzpicture}
};
\node (a3) at (\layoutstretch*\internalscaling*-7-\leftshift,\internalscaling*3.4)
{\begin{tikzpicture}[xscale=0.4*\internalscaling*\Dynkinscale,yscale=0.4*\internalscaling*\Dynkinscale]
\node (-1) at (-0.75,0) [DWs] {};
\node (0) at (0,0) [DBs] {};
\node (1) at (0.75,0) [DWs] {};
\node (1b) at (0.75,0.75) [DBs] {};
\node (2) at (1.5,0) [DBs] {};
\node (3) at (2.25,0) [DBs] {};
\draw [-] (-1) -- (0);
\draw [-] (0) -- (1);
\draw [-] (1) -- (2);
\draw [-] (2) -- (3);
\draw [-] (1) -- (1b);
\end{tikzpicture}};
\defThreeFold{\layoutstretch*\internalscaling*-0.5+6*\bottomlineshift}{0}{black};
\node (a2) at (\layoutstretch*\internalscaling*-0.2+6*\bottomlineshift,-\internalscaling*0.25+\cutshift) 
{\begin{tikzpicture}[xscale=\cutscaling*\internalscaling*1.2,yscale=\cutscaling*\internalscaling*1.2,transform shape, rotate=\cuttilt]
\draw[densely dotted,  gray] (0,0)--(0.8,0)--(0.8,1.4)--(0,1.4)--(0,0);
\filldraw [black]  (intersection of 0,0--0.8,1.4 and 0.8,0--0,1.4) circle (0.6pt);
\end{tikzpicture}};
    \defThreeFold{\layoutstretch*\internalscaling*-4.4+6*\bottomlineshift-\leftshift}{0}{densely dotted, gray};
\node (a2) at (\layoutstretch*\internalscaling*-4.1+6*\bottomlineshift-\leftshift,-\internalscaling*0.25+\cutshift) 
{\begin{tikzpicture}[xscale=\cutscaling*\internalscaling*1.2,yscale=\cutscaling*\internalscaling*1.2,transform shape, rotate=\cuttilt]
\draw[black] (0,0)--(0.8,0)--(0.8,1.4)--(0,1.4)--(0,0);
\filldraw [black]  (intersection of 0,0--0.8,1.4 and 0.8,0--0,1.4) circle (0.6pt);
\end{tikzpicture}};

\draw[->] (\layoutstretch*\internalscaling*-0.25+2*\bottomlineshift,\internalscaling*1.05+\toplineshift)  -- node[right] {$\small f$} (\layoutstretch*\internalscaling*-0.25+4*\bottomlineshift,\internalscaling*0.5+\toplineshift);
\draw[->] (\layoutstretch*\internalscaling*-4.25+2*\bottomlineshift-\leftshift,\internalscaling*1.05+\toplineshift)  -- node[left] {$\scriptstyle $} (\layoutstretch*\internalscaling*-4.25+4*\bottomlineshift-\leftshift,\internalscaling*0.5+\toplineshift);
\draw[->] (\layoutstretch*\internalscaling*-2.5+6*\bottomlineshift-\leftshift,\internalscaling*-0.25)  --  (\layoutstretch*\internalscaling*-2.1+6*\bottomlineshift,\internalscaling*-0.25);
\draw[->] (\layoutstretch*\internalscaling*-2.5-\leftshift,\internalscaling*2.25+\toplineshift)  --  (\layoutstretch*\internalscaling*-2.1,\internalscaling*2.25+\toplineshift);
\node at (\layoutstretch*\internalscaling*+0.75+6*\bottomlineshift,\internalscaling*-1.5) {$\small\Spec R$};
\node at (\layoutstretch*\internalscaling*-3.35+6*\bottomlineshift-\leftshift,\internalscaling*-1.5) {$\small\Spec R/g$};
\node at (\layoutstretch*\internalscaling*+1.1,\internalscaling*1.4+\toplineshift) {$\small U$};
\end{tikzpicture}
\]

Since all minimal models of $\Spec R$ are connected by repeatedly flopping the two curves (see e.g.\ \cite{KollarFlops}), a degree $3$ braid relation would imply that $\Spec R$ has only $6$ minimal models. However, in the above example, by \cite{Pinkham} or \cite[7.2]{HomMMP} the minimal models of $\Spec R$ correspond to the $10$ chambers in the following hyperplane arrangement $\cH$ in~$\mathbb{R}^2$.
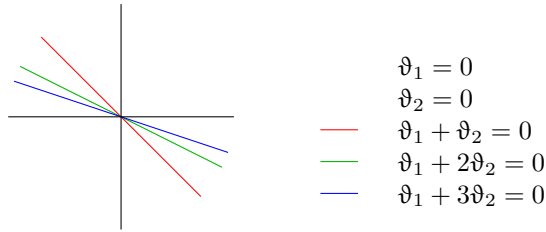
\[
\begin{array}{c}
\begin{tikzpicture}
\node at (6.5,0) {$\begin{tikzpicture}[scale=0.75,>=stealth]
\coordinate (A1) at (135:2cm);
\coordinate (A2) at (-45:2cm);
\coordinate (B1) at (153.435:2cm);
\coordinate (B2) at (-26.565:2cm);
\coordinate (C1) at (161.565:2cm);
\coordinate (C2) at (-18.435:2cm);
\draw[red] (A1) -- (A2);
\draw[green!70!black] (B1) -- (B2);
\draw[blue] (C1) -- (C2);
\draw (-2,0)--(2,0);
\draw (0,-2)--(0,2);
\end{tikzpicture}$};
\node at (10.5,0) {
$
\begin{array}{cl}
\\
&\upvartheta_1=0\\
&\upvartheta_2=0\\
\begin{array}{c}\begin{tikzpicture}\node at (0,0){}; \draw[red] (0,0)--(0.5,0);\end{tikzpicture}\end{array}&\upvartheta_1+\upvartheta_2=0\\
\begin{array}{c}\begin{tikzpicture}\node at (0,0){}; \draw[green!70!black] (0,0)--(0.5,0);\end{tikzpicture}\end{array}&\upvartheta_1+2\upvartheta_2=0\\
\begin{array}{c}\begin{tikzpicture}\node at (0,0){}; \draw[blue] (0,0)--(0.5,0);\end{tikzpicture}\end{array}&\upvartheta_1+3\upvartheta_2=0
\end{array}
$};
\end{tikzpicture}
\end{array}
\]
This shows that a degree $3$ braid relation cannot hold in this example, but does suggest a higher length braid relation given by
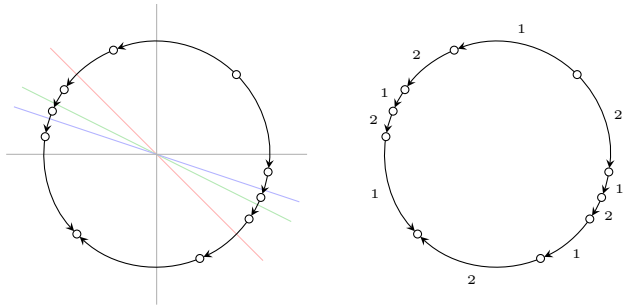
\[
\begin{array}{c}
\begin{array}{cc}
\begin{array}{c}
\begin{tikzpicture}[scale=1,>=stealth]
\coordinate (A1) at (135:2cm);
\coordinate (A2) at (-45:2cm);
\coordinate (B1) at (153.435:2cm);
\coordinate (B2) at (-26.565:2cm);
\coordinate (C1) at (161.565:2cm);
\coordinate (C2) at (-18.435:2cm);
\draw[red!30] (A1) -- (A2);
\draw[green!70!black!30] (B1) -- (B2);
\draw[blue!30] (C1) -- (C2);
\draw[black!30] (-2,0)--(2,0);
\draw[black!30] (0,-2)--(0,2);
\draw[->] (45:1.5cm) arc (45:110.5:1.5cm);
\draw[->] (112.5:1.5cm) arc (112.5:143:1.5cm);
\draw[->] (145:1.5cm) arc (145:155.5:1.5cm);
\draw[->] (157.5:1.5cm) arc (157.5:169:1.5cm);
\draw[->] (171:1.5cm) arc (171:223:1.5cm);
\draw[->] (45:1.5cm) arc (45:-7:1.5cm);
\draw[->] (-9:1.5cm) arc (-9:-20.5:1.5cm);
\draw[->] (-22.5:1.5cm) arc (-22.5:-33:1.5cm);
\draw[->] (-35:1.5cm) arc (-35:-65.5:1.5cm);
\draw[->] (-67.5:1.5cm) arc (-67.5:-133:1.5cm);
\node at (45:1.5cm) [DWs] {};
\node at (112.5:1.5cm) [DWs] {};
\node at (145:1.5cm) [DWs] {};
\node at (157.5:1.5cm) [DWs] {};
\node at (171:1.5cm) [DWs] {};
\node  at (225:1.5cm) [DWs] {};
\node at (-67.5:1.5cm) [DWs] {};
\node at (-35:1.5cm) [DWs] {};
\node at (-22.5:1.5cm) [DWs] {};
\node at (-9:1.5cm) [DWs] {};
\end{tikzpicture}
\end{array}
&
\begin{array}{c}
\begin{tikzpicture}[scale=1,>=stealth]
\draw[->] (45:1.5cm) arc (45:110.5:1.5cm);
\draw[->] (112.5:1.5cm) arc (112.5:143:1.5cm);
\draw[->] (145:1.5cm) arc (145:155.5:1.5cm);
\draw[->] (157.5:1.5cm) arc (157.5:169:1.5cm);
\draw[->] (171:1.5cm) arc (171:223:1.5cm);
\draw[->] (45:1.5cm) arc (45:-7:1.5cm);
\draw[->] (-9:1.5cm) arc (-9:-20.5:1.5cm);
\draw[->] (-22.5:1.5cm) arc (-22.5:-33:1.5cm);
\draw[->] (-35:1.5cm) arc (-35:-65.5:1.5cm);
\draw[->] (-67.5:1.5cm) arc (-67.5:-133:1.5cm);
\node at (45:1.5cm) [DWs] {};
\node at (112.5:1.5cm) [DWs] {};
\node at (145:1.5cm) [DWs] {};
\node at (157.5:1.5cm) [DWs] {};
\node at (171:1.5cm) [DWs] {};
\node  at (225:1.5cm) [DWs] {};
\node at (-67.5:1.5cm) [DWs] {};
\node at (-35:1.5cm) [DWs] {};
\node at (-22.5:1.5cm) [DWs] {};
\node at (-9:1.5cm) [DWs] {};
\node at (78.75:1.7cm) {$\scriptstyle {}_1$};
\node at (128.75:1.7cm) {$\scriptstyle {}_2$};
\node at (151.25:1.7cm) {$\scriptstyle {}_1$};
\node at (164.25:1.7cm) {$\scriptstyle {}_2$};
\node at (198:1.7cm) {$\scriptstyle {}_1$};
\node at (18:1.7cm) {$\scriptstyle {}_2$};
\node at (-15.75:1.7cm) {$\scriptstyle {}_1$};
\node at (-28.75:1.7cm) {$\scriptstyle {}_2$};
\node at (-51.25:1.7cm) {$\scriptstyle {}_1$};
\node at (-101.25:1.7cm) {$\scriptstyle {}_2$};
\end{tikzpicture}
\end{array}
\end{array}\\
\flop_1\circ\flop_2\circ\flop_1\circ\flop_2\circ\flop_1\functcong\flop_2\circ\flop_1\circ\flop_2\circ\flop_1\circ\flop_2,
\end{array}
\]
and indeed this follows from the general theory developed below. This example demonstrates that the braiding of $3$-fold flops is not controlled by the dual graph, in contrast to the braiding of spherical twists associated to $(-2)$-curves on surfaces \cite{ST}.

\subsection{Group Actions from Flops}\label{Intro new results}
We first describe the case where the flopping contraction $f\colon X\to X_{\con}$ contracts precisely two irreducible curves, where each curve is individually floppable.  The general case will be described later.  

Since there are two curves, $f$ has an associated hyperplane arrangement $\cH$ in $\mathbb{R}^2$, using \cite{HomMMP} (see~\S\ref{moduli tracking section}).
As in the example above, we write $\flop_1$ and $\flop_2$ for the Bridgeland--Chen flop functors associated to the two flopping curves in $X$. By abuse of notation, we also use $\flop_1$ and $\flop_2$ to denote flop functors at the varieties obtained after successive flops.  The functor associated to the flop of the two curves together is denoted by $\flop$.

\begin{thm}[=\ref{braid 2 curve global}]\label{intro braid 2 curve}
Suppose that $X\to X_{\con}$ is a flopping contraction between quasi-projective $3$-folds, contracting precisely two independently floppable irreducible curves.  If $X$ has at worst Gorenstein terminal singularities, then
\[
\underbrace{\flop_1\circ\flop_2\circ\flop_1\circ\cdots}_{d}\functcong
\flop
\functcong
\underbrace{\flop_2\circ\flop_1\circ\flop_2\circ\cdots}_{d}
\]
where $d$ is the number of hyperplanes in $\cH$.  Furthermore:
\begin{enumerate}
\item\label{intro braid 2 curve 1} If the curves intersect, then $d\geq 3$.
\item\label{intro braid 2 curve 2} If the curves are disjoint, then $d=2$.
\end{enumerate}
\end{thm}

In fact $d\leq 8$, though we do not show this here.
The proof of \ref{intro braid 2 curve} uses moduli tracking, developed as part of the Homological MMP \cite{HomMMP}.  The strategy is to track skyscrapers $\cO_x$, with the challenge being to show that the composition 
\[
(\flop_1^{-1}\circ\flop_2^{-1}\circ\flop_1^{-1}\circ\cdots)\circ(\cdots\circ\flop_2\circ\flop_1\circ\flop_2)
\]
applied to $\cO_x$ is also a skyscraper.

When there are $n$ irreducible curves in the exceptional locus of the flopping contraction, braiding is controlled by a hyperplane arrangement $\cH$ in $\mathbb{R}^n$, which we prove is simplicial in \ref{issimplicial}.  As a motivating example, consider the following hyperplane arrangement $\cH$ in $\mathbb{R}^3$, which arises from certain $D_4$ flops with three curves above the origin.  It has $7$ hyperplanes, and $32$ chambers:

\smallskip
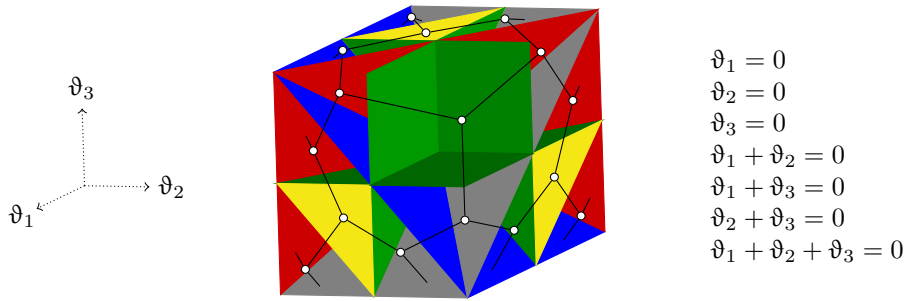
\[
\begin{array}{ccccccc}
\begin{array}{c}
\begin{tikzpicture}[scale=1]
\rotateRPY{-5}{-15}{5}
\begin{scope}[RPY]
\draw [->,densely dotted] (0,0) -- (1,0,0) node [right] {$\upvartheta_2$}; 
\draw [->,densely dotted] (0,0) -- (0,1,0) node [above] {$\upvartheta_3$}; 
\draw [->,densely dotted] (0,0) -- (0,0,1);
\node at (0,0,1.3) {$\upvartheta_1$}; 
\end{scope}
\end{tikzpicture}
\end{array}
&&
\begin{array}{c}
\begin{tikzpicture}[scale=1.2]
\rotateRPY{-5}{-15}{5}
\begin{scope}[RPY]
\filldraw[gray] (-1,-1,1) -- (-1,1,-1) -- (1,-1,1) -- (-1,-1,1);
\filldraw[red!80!black] (-1,1,1) -- (-1,-1,1) -- (0,0,0) -- cycle;
\filldraw[green!40!black] (-1,0,1) -- (0,0,1) -- (0,0,0)-- cycle;
\filldraw[green!60!black] (0,-1,1) -- (0,0,1) -- (0,0,0)-- cycle;
\filldraw[blue] (-1,1,1) -- (-1,1,-1) -- (1,-1,-1) -- (1,-1,1);
\filldraw[red!80!black] (1,1,-1) -- (1,-1,-1) -- (0,0,0) -- cycle;
\filldraw[green!40!black] (1,0,-1) -- (1,0,0) -- (0,0,0)-- cycle;
\filldraw[green!50!black] (1,-1,0) -- (1,0,0) -- (0,0,0)-- cycle;
\filldraw[yellow!95!black] (1,-1,0) -- (1,0,-1) -- (1,0,0)-- cycle;
\filldraw[yellow!95!black] (-1,0,1) -- (0,-1,1) -- (0,0,1)-- cycle;
\filldraw[gray] (-1,1,-1) -- (1,1,-1) -- (1,-1,1)--(-1,1,-1);
\filldraw[yellow!95!black] (-1,1,0) -- (0,1,-1) -- (0,0,0)--(-1,1,0);
\filldraw[green!50!black] (-1,1,0) -- (0,0,0) -- (0,1,0) -- cycle;
\filldraw[green!60!black] (0,1,-1)--(0,0,0) -- (0,1,0) -- cycle;
\filldraw[red!80!black] (-1,1,1) -- (1,1,-1) -- (0,0,0)-- cycle;
\filldraw[green!60!black] (0,1,1) -- (0,0,1) -- (0,0,0)-- (0,1,0)-- cycle;
\filldraw[green!50!black] (1,1,0) -- (1,0,0) -- (0,0,0)-- (0,1,0)-- cycle;
\filldraw[green!40!black] (1,0,1) -- (1,0,0) -- (0,0,0)-- (0,0,1)-- cycle;
\node (A) at (1,0.6,1) [DWs] {};
\node (B1) at (-0.6,0.7,0.6) [DWs] {};
\node (B2) at (-1,0.9,0) [DWs] {};
\node (C1) at (0.6,0.7,-0.7) [DWs] {};
\node (C2) at (0,0.9,-1) [DWs] {};
\node (D) at (-1,0.7,-1.2) [DWs] {};
\node (E) at (-1,0.9,-1) [DWs] {};
\node (F) at (1,-0.3,1) [DWs] {};
\node (F1) at (0.3,-0.6,1) [DWs] {};
\node (F2) at (-0.3,-0.3,1) [DWs] {};
\node (J) at (-0.8,-0.8,0.9) [DWs] {};
\node (F3) at (-0.6,0.3,1) [DWs] {};
\node (G1) at (1,-0.6,0.3) [DWs] {};
\node (G2) at (1,-0.3,-0.3) [DWs] {};
\node (H) at (0.9,-0.8,-0.8) [DWs] {};
\node (G3) at (1,0.3,-0.6) [DWs] {};
\draw (A)--(B1)--(B2)--(D)--(E);
\draw (A)--(C1)--(C2)--(D);
\draw (E) -- (-1,0.82,-1.15); 
\draw (E) -- (-1.1,0.83,-1); 
\draw (B2) -- (-1.1,0.83,0);
\draw (C2) -- (0.1,0.85,-1.15);
\draw (A)--(F)--(F1)--(F2)--(F3)--(B1);
\draw (F)--(G1)--(G2)--(G3)--(C1);
\draw (G2)--(H);
\draw (F2)--(J);
\draw (G3)--(1,0.4,-0.7);
\draw (H)--(0.85,-0.7,-1);
\draw (H)--(1,-0.9,-0.4);
\draw (G1)--(0.9,-0.9,0.5);
\draw (F1)--(0.5,-0.9,0.9);
\draw (F3)--(-0.75,0.4,0.9);
\draw (J)--(-0.9,-0.65,0.9);
\draw (J)--(-0.65,-0.95,0.9);
\end{scope}
\end{tikzpicture}
\end{array}
&&&
\begin{array}{l}
\upvartheta_1=0\\
\upvartheta_2=0\\
\upvartheta_3=0\\
\upvartheta_1+\upvartheta_2=0\\
\upvartheta_1+\upvartheta_3=0\\
\upvartheta_2+\upvartheta_3=0\\
\upvartheta_1+\upvartheta_2+\upvartheta_3=0
\end{array}
\end{array}
\]

\smallskip
The Deligne groupoid associated to $\cH$ has by definition a vertex for every chamber, arrows between adjacent chambers, and relations corresponding to codimension-two walls (see \ref{Paris defin} for full details).  In the above example, four of these relations are illustrated in the picture below, corresponding to the four codimension-two walls marked blue. 
\smallskip
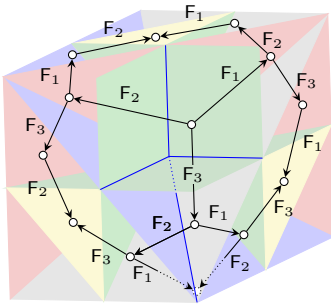
\[
\begin{array}{c}
\begin{tikzpicture}[scale=1.2,>=stealth]
\rotateRPY{-5}{-15}{5}
\begin{scope}[RPY]
\filldraw[gray!20] (-1,-1,1) -- (-1,1,-1) -- (1,-1,1) -- (-1,-1,1);
\filldraw[red!80!black!20] (-1,1,1) -- (-1,-1,1) -- (0,0,0) -- cycle;
\filldraw[green!40!black!20] (-1,0,1) -- (0,0,1) -- (0,0,0)-- cycle;
\filldraw[green!60!black!20] (0,-1,1) -- (0,0,1) -- (0,0,0)-- cycle;
\filldraw[blue!20] (-1,1,1) -- (-1,1,-1) -- (1,-1,-1) -- (1,-1,1);
\filldraw[red!80!black!20] (1,1,-1) -- (1,-1,-1) -- (0,0,0) -- cycle;
\filldraw[green!40!black!20] (1,0,-1) -- (1,0,0) -- (0,0,0)-- cycle;
\filldraw[green!50!black!20] (1,-1,0) -- (1,0,0) -- (0,0,0)-- cycle;
\filldraw[yellow!95!black!20] (1,-1,0) -- (1,0,-1) -- (1,0,0)-- cycle;
\filldraw[yellow!95!black!20] (-1,0,1) -- (0,-1,1) -- (0,0,1)-- cycle;
\filldraw[gray!20] (-1,1,-1) -- (1,1,-1) -- (1,-1,1)--(-1,1,-1);
\filldraw[yellow!95!black!20] (-1,1,0) -- (0,1,-1) -- (0,0,0)--(-1,1,0);
\filldraw[green!50!black!20] (-1,1,0) -- (0,0,0) -- (0,1,0) -- cycle;
\filldraw[green!60!black!20] (0,1,-1)--(0,0,0) -- (0,1,0) -- cycle;
\filldraw[red!80!black!20] (-1,1,1) -- (1,1,-1) -- (0,0,0)-- cycle;
\filldraw[green!60!black!20] (0,1,1) -- (0,0,1) -- (0,0,0)-- (0,1,0)-- cycle;
\filldraw[green!50!black!20] (1,1,0) -- (1,0,0) -- (0,0,0)-- (0,1,0)-- cycle;
\filldraw[green!40!black!20] (1,0,1) -- (1,0,0) -- (0,0,0)-- (0,0,1)-- cycle;
\node (A) at (1,0.6,1) [DWs] {};
\node (B1) at (-0.6,0.7,0.6) [DWs] {};
\node (B2) at (-1,0.9,0) [DWs] {};
\node (C1) at (0.6,0.7,-0.7) [DWs] {};
\node (C2) at (0,0.9,-1) [DWs] {};
\node (D) at (-1,0.7,-1.2) [DWs] {};
\node (F) at (1,-0.3,1) [DWs] {};
\node (F1) at (0.3,-0.6,1) [DWs] {};
\node (F2) at (-0.3,-0.3,1) [DWs] {};
\node (F3) at (-0.6,0.3,1) [DWs] {};
\node (G1) at (1,-0.6,0.3) [DWs] {};
\node (G2) at (1,-0.3,-0.3) [DWs] {};
\node (G3) at (1,0.3,-0.6) [DWs] {};
\node (N) at (1,-0.9,1) [DWds] {};
\draw[-,blue] (0,0,0)--(0,1,0);
\draw[-,blue] (0,0,0)--(0,0,1);
\draw[-,blue] (0,0,0)--(1,0,0);
\draw[-,blue,densely dotted] (0,0,0)--(.25,-.25,.25);
\draw[-,blue] (.25,-.25,.25)--(1,-1,1);
\draw[->] (A)-- node[above] {$\scriptstyle\flop_2$} (B1);
\draw[->] (B1)--node[left] {$\scriptstyle\flop_1$}(B2);
\draw[->] (B2)--node[above] {$\scriptstyle\flop_2$}(D);
\draw[->] (A)-- node[above] {$\scriptstyle\flop_1$} (C1);
\draw[->] (C1)--node[right] {$\scriptstyle\flop_2$}(C2);
\draw[->] (C2)--node[above] {$\scriptstyle\flop_1$}(D);
\draw[->] (A)-- node[Ggap] {$\scriptstyle\flop_3$} (F);
\draw[->] (F)--node[above] {$\scriptstyle\flop_2$}(F1);
\draw[->] (F1)--node[below] {$\scriptstyle\flop_3$}(F2);
\draw[->] (F)--node[above] {$\scriptstyle\flop_2$}(F1);
\draw[->] (B1)--node[left] {$\scriptstyle\flop_3$}(F3);
\draw[->] (F3)--node[left] {$\scriptstyle\flop_2$}(F2);
\draw[->] (C1)--node[right] {$\scriptstyle\flop_3$}(G3);
\draw[->] (G3)--node[right] {$\scriptstyle\flop_1$}(G2);
\draw[->] (F)--node[above] {$\scriptstyle\flop_1$}(G1);
\draw[->] (G1)--node[right] {$\scriptstyle\flop_3$}(G2);
\draw[-] (F1)--($(F1)!.4!(N)$);
\draw[->,densely dotted] ($(F1)!.4!(N)$)--(N);
\path (F1) -- node[left] {$\scriptstyle\flop_1$}(N);
\draw[-] (G1)--($(G1)!.4!(N)$);
\draw[->,densely dotted] ($(G1)!.4!(N)$)--(N);
\path (G1) -- node[right] {$\scriptstyle\flop_2$}(N);
\end{scope}
\end{tikzpicture}
\end{array}
\]

We may now state our first main theorem.
\begin{thm}[=\ref{flopsglobalmultiple}]\label{flops global intro}
Suppose that $X\to X_{\con}$ is a flopping contraction between quasi-projective $3$-folds, where $X$ has Gorenstein terminal singularities and each of the $n$ irreducible exceptional curves is individually floppable. 
\begin{enumerate}
\item\label{flops global intro 1} The flop functors $\flop_i$ form a representation of the Deligne groupoid associated to~$\mathcal{H}$.
\item\label{flops global intro 2} There is a group homomorphism 
\[
\fundgp(\mathbb{C}^n\backslash \cH_\mathbb{C})\to\Aut\Db(\coh X),
\]
where $\cH_{\mathbb{C}}$ denotes the complexification of $\cH$.
\end{enumerate}
\end{thm}

To prove \eqref{flops global intro 1}, we show that crashing through a codimension-two wall corresponds to  flopping two curves together, so that the relations in the Deligne groupoid representation can be verified using \ref{intro braid 2 curve}.
The proof does not require knowledge of a group presentation of $\fundgp(\mathbb{C}^n\backslash \cH_\mathbb{C})$.

\subsection{The $J$-Twists}\label{2 actors}In order to describe geometrically the action of some of the generators of $\fundgp(\mathbb{C}^n\backslash \cH_\mathbb{C})$ on the derived category of $X$, and also to describe the case when the curves are not individually floppable, we next associate intrinsic derived symmetries to a given flopping contraction $X\to X_{\con}$.  

In the case $n=1$, the authors previously described the action of a generator of $\fundgp(\mathbb{C}^1\backslash \cH_\mathbb{C})=\fundgp(\mathbb{C}\backslash \{0\})=\mathbb{Z}$ on $\Db(\coh X)$ by using noncommutative deformation theory of the reduced curve, and twisting around a universal family over the so-called contraction algebra \cite{DW1}.  However in the general case considered here, the twists from \cite{DW1} correspond only to monodromy around certain codimension-one walls, and in general this recovers only a small part of $\fundgp(\mathbb{C}^{n}\backslash \cH_\mathbb{C})$.  

To remedy this we describe the action of monodromies around higher codimension walls, by considering simultaneous deformations of multiple curves.   We thus choose a subset $J$ of the curves, and from this produce an autoequivalence, corresponding to monodromy around a codimension $|J|$ wall.

To produce such a twist autoequivalence, we first construct the deformation base.  Choosing a point $p\in X_{\con}$, it is now well-known that the formal fibre above $p$ is derived equivalent to a noncommutative ring $\AB$ \cite{VdB1d}.  If there are $n_p$ curves in the formal fibre, then this algebra has $n_p$ idempotents corresponding to these curves, plus one other.  Furthermore, as explained in \cite[2.15]{HomMMP}, the presentation of $\AB$ depends on the intersection theory of the curves. An example with two curves is drawn below.
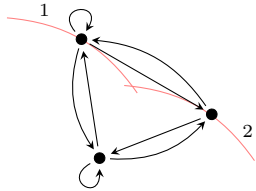
\[
\begin{tikzpicture} 
\begin{scope}[bend angle=25, looseness=1,transform shape, rotate=-30,>=stealth]
\draw[red!50] (0,0) to [bend left=25] (2,0,0);
\draw[red!50] (1.8,0) to [bend left=25] (3.8,0,0);
\filldraw [black] (1,0.25) circle (2pt);
\filldraw [black] (3,0.25) circle (2pt);
\filldraw [black] (2,-1) circle (2pt);
\node (1) at (1,0.25) {};
\node (2) at (3,0.25) {};
\node (0) at (2,-1) {};
\draw[->,black] (1) -- (2);
\draw[->,black]  (1) edge [in=90,out=160,loop,looseness=8] (1);
\draw[->,black,bend right] (2) to (1);
\draw[->,black] (0) -- (1);
\draw[->,black] (2) -- (0);
\draw[->,bend right,black] (1) to (0);
\draw[->, bend right,black] (0) to (2);
\draw[->,black]  (0) edge [in=-60,out=-120,loop,looseness=8] (0);
\end{scope}
\node at (0.5,0.1) {$\scriptstyle 1$};
\node at (3.2,-1.5) {$\scriptstyle 2$};
\end{tikzpicture}
\]
We generalise the contraction algebra of \cite{DW1} to this setting (see \ref{def basic algebra}), which in the example above can be visualised as follows.  For any subset of the curves $J\subseteq\{1,2\}$, we define $\AB_J$ by factoring out all other idempotents, so that different $\AB_J$ have presentations as drawn below:
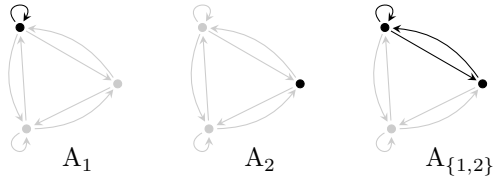
\[
\begin{array}{ccc}
\begin{array}{c}
\begin{tikzpicture} 
\begin{scope}[bend angle=25, looseness=1,transform shape, rotate=-30,>=stealth,scale=0.75]
\filldraw [black] (1,0.25) circle (2pt);
\filldraw [black!20] (3,0.25) circle (2pt);
\filldraw [black!20] (2,-1.25) circle (2pt);
\node (1) at (1,0.25) {};
\node (2) at (3,0.25) {};
\node (0) at (2,-1.25) {};
\draw[->,black!20] (1) -- (2);
\draw[->,black]  (1) edge [in=90,out=160,loop,looseness=8] (1);
\draw[->,black!20,bend right] (2) to (1);
\draw[->,black!20] (0) -- (1);
\draw[->,black!20] (2) -- (0);
\draw[->,bend right,black!20] (1) to (0);
\draw[->, bend right,black!20] (0) to (2);
\draw[->,black!20]  (0) edge [in=-60,out=-120,loop,looseness=8] (0);
\end{scope}
\node at (1.5,-2) {${}_{\phantom\}}\AB_{1\phantom\}}$};
\end{tikzpicture}
\end{array}
&
\begin{array}{c}
\begin{tikzpicture} 
\begin{scope}[bend angle=25, looseness=1,transform shape, rotate=-30,>=stealth,scale=0.75]
\filldraw [black!20] (1,0.25) circle (2pt);
\filldraw [black] (3,0.25) circle (2pt);
\filldraw [black!20] (2,-1.25) circle (2pt);
\node (1) at (1,0.25) {};
\node (2) at (3,0.25) {};
\node (0) at (2,-1.25) {};
\draw[->,black!20] (1) -- (2);
\draw[->,black!20]  (1) edge [in=90,out=160,loop,looseness=8] (1);
\draw[->,black!20,bend right] (2) to (1);
\draw[->,black!20] (0) -- (1);
\draw[->,black!20] (2) -- (0);
\draw[->,bend right,black!20] (1) to (0);
\draw[->, bend right,black!20] (0) to (2);
\draw[->,black!20]  (0) edge [in=-60,out=-120,loop,looseness=8] (0);
\end{scope}
\node at (1.5,-2) {${}_{\phantom\}}\AB_{2\phantom\}}$};
\end{tikzpicture}
\end{array}
&
\begin{array}{c}
\begin{tikzpicture} 
\begin{scope}[bend angle=25, looseness=1,transform shape, rotate=-30,>=stealth,scale=0.75]
\filldraw [black] (1,0.25) circle (2pt);
\filldraw [black] (3,0.25) circle (2pt);
\filldraw [black!20] (2,-1.25) circle (2pt);
\node (1) at (1,0.25) {};
\node (2) at (3,0.25) {};
\node (0) at (2,-1.25) {};
\draw[->,black] (1) -- (2);
\draw[->,black]  (1) edge [in=90,out=160,loop,looseness=8] (1);
\draw[->,black,bend right] (2) to (1);
\draw[->,black!20] (0) -- (1);
\draw[->,black!20] (2) -- (0);
\draw[->,bend right,black!20] (1) to (0);
\draw[->, bend right,black!20] (0) to (2);
\draw[->,black!20]  (0) edge [in=-60,out=-120,loop,looseness=8] (0);
\end{scope}
\node at (1.75,-2) {$\AB_{\{1,2\}}$};
\end{tikzpicture}
\end{array}
\end{array}
\]
It is shown in \cite{DW2} that $\AB_\indxset$ represents the functor of simultaneous noncommutative deformations of the reduced curves in $J$.  We then in \ref{twists definition} construct a functor
\[
\Tcon\colon \Db(\coh X)\to \Db(\coh X)
\]
called the $J$-twist, which twists around the universal sheaf $\cE_J$ from the noncommutative deformation theory.
\begin{thm}[=\ref{twists autoequivalence B}]\label{intro twists autoequivalence B} 
Suppose that $X\to X_{\con}$ is a flopping contraction between quasi-projective $3$-folds, where $X$ has Gorenstein terminal singularities.
\begin{enumerate}
\item\label{intro twists autoequivalence B 1} For any subset $J$ of the flopping curves, $\Tcon$ is an equivalence.
\item\label{intro twists autoequivalence B 2} If the curves $J$ are contracted to a  single point, then there is a functorial triangle
\[
\RHom_X(\cE_J,x)\otimes^{\bf L}_{\AB_J}\cE_J \to x\to \Tcon(x)\to.
\]
\end{enumerate}
\end{thm}

The construction of the functor $\Tcon$ is a three-step process. We first construct it on the formal fibre, then via an algebraic idempotent method we extend this to an open neighbourhood of the curves, before finally gluing to give a functor on $X$.  At various stages, the proof requires the existence of adjoints. The new adjoint technology of Rizzardo \cite{Alice} allows us to construct these, in the process dropping the projectivity assumption from previous works \cite{Toda,DW1}, and also relaxing the conditions on singularities in \ref{intro twists autoequivalence B}, and indeed throughout the whole paper.  It turns out that almost all of our theorems only require Gorenstein terminal singularities in a neighbourhood of the contracted curves; our general setup is explained in \ref{flopsglobal}.

When the union of curves $\bigcup_{j\in J}C_j$ flops algebraically, the flop functor $\flop_J$ exists, and so we are in a position to compare it to the $J$-twist.

\begin{thm}[{=\ref{twist vs flopflop}}] With assumptions as above, for a choice of subset $J$ of the flopping curves, suppose that $\bigcup_{j\in J}C_j$  flops algebraically. Then $\Tcon \circ (\flop_J \circ \flop_J) \functcong \Id.$\end{thm}

Thus the $J$-twist gives an intrinsic characterisation of the inverse of the flop--flop functor, but has the advantage of always existing regardless of whether $\bigcup_{j\in J}C_j$ flops algebraically.  This leads to our main result in the purely algebraic setting, where individual curves are not necessarily floppable.

\begin{thm}[=\ref{global main result action}]
Suppose that $X\to X_{\con}$ is a flopping contraction, where $X$ is a quasi-projective $3$-fold with only Gorenstein terminal singularities.  The subgroup $K$ of \,$\fundgp(\mathbb{C}^n\backslash \cH_\mathbb{C})$ generated by the $J$-twists, as $J$ ranges over all subsets of curves, acts on $\Db(\coh X)$.
\end{thm}

It can happen that $K=\fundgp(\mathbb{C}^n\backslash \cH_\mathbb{C})$ (see \ref{whole pure braid group}), but verifying this in any level of generality seems group-theoretically difficult; we discuss this briefly in \S\ref{Section6.1}.  We offer the following conjecture, which in particular would give new generating sets for fundamental groups in known cases where $\cH$ is associated to a semisimple Lie algebra.
\begin{conj}
$K=\fundgp(\mathbb{C}^{n}\backslash \cH_\mathbb{C})$ $\iff$ $\cH$ is a root system of a semisimple Lie algebra.
\end{conj}

\subsection{The Fibre Twist}
We then construct new derived autoequivalences not in the image of the group homomorphism in \ref{flops global intro}.  We expect these to coincide with the action of the affine element of some larger group, analogous to the affine braid group actions for chains of $(-2)$-curves on surfaces \cite{BridgelandKleinian}.  This is remarkable, since hyperplane arrangements $\cH$ such as the one in \S\ref{mot ex} do not have affine versions, so there is no obvious candidate for this expected larger group.

From the viewpoint of idempotents of $\AB$ in \S\ref{2 actors}, the construction of the new autoequivalence is clear, namely we define a base $\AB_\fib$ for the twist by factoring out the idempotents corresponding to all the curves. In the example from \S\ref{2 actors}, this can be visualised as follows.
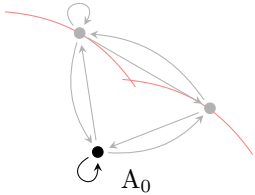
\[
\begin{tikzpicture} 
\begin{scope}[bend angle=25, looseness=1,transform shape, rotate=-30,>=stealth]
\draw[red!50] (0,0) to [bend left=25] (2,0,0);
\draw[red!50] (1.8,0) to [bend left=25] (3.8,0,0);
\filldraw [black!30] (1,0.25) circle (2pt);
\filldraw [black!30] (3,0.25) circle (2pt);
\filldraw [black] (2,-1) circle (2pt);
\node (1) at (1,0.25) {};
\node (2) at (3,0.25) {};
\node (0) at (2,-1) {};
\draw[->,black!30] (1) -- (2);
\draw[->,black!30]  (1) edge [in=90,out=160,loop,looseness=8] (1);
\draw[->,black!30,bend right] (2) to (1);
\draw[->,black!30] (0) -- (1);
\draw[->,black!30] (2) -- (0);
\draw[->,bend right,black!30] (1) to (0);
\draw[->, bend right,black!30] (0) to (2);
\draw[->,black]  (0) edge [in=-60,out=-120,loop,looseness=8] (0);
\end{scope}
\node at (1.75,-2.25) {$\AB_{\fib}$};
\end{tikzpicture}
\]
In contrast to the $J$-twist, whose base $\AB_J$ represents simultaneous noncommutative deformations of the reduced curves, by \cite{DW2} $\AB_\fib$ represents both the commutative and noncommutative deformations of the whole scheme-theoretic exceptional fibre, and so in particular $\AB_\fib$ is a finite dimensional commutative local $\mathbb{C}$-algebra.

We then in \ref{twists definition} construct a functor
\[
\Tfib\colon \Db(\coh X)\to \Db(\coh X)
\]
called the fibre twist, which twists around the universal sheaf $\cE_\fib$ from the commutative deformation theory.

\begin{thm}[=\ref{twists autoequivalence B}]\label{Ftwists autoequivalence intro} Suppose that $f\colon X\to X_{\con}$ is a flopping contraction of quasi-projective 3-folds, where $X$ is $\mathds{Q}$-factorial and has Gorenstein terminal  singularities.  
\begin{enumerate}
\item The fibre twist $\Tfib$ is an equivalence.
\item There is a functorial triangle
\[
 \RHom_X(\cE_{\fib},x)\otimes^{\bf L}_{\CAR}\cE_{\fib}\to x\to \Tfibp(x)\to
\]
\end{enumerate}
\end{thm}

As with the $J$-twist, the construction of the functor $\Tfib$ is a three-step process.  However step one, establishing the functor $\Tfib$ on the formal fibre, is now significantly harder and this results in the additional $\mathds{Q}$-factorial assumption in \ref{Ftwists autoequivalence intro}.  We explain the reasons for this briefly in the next subsection.

The fibre twist does not commute with the pushdown along the contraction $\RDerived f_*$, and for this reason does not belong to the image of the homomorphism in \ref{flops global intro}. It is natural to ask whether it nevertheless lies in the subgroup generated by this image and twists by line bundles, and in particular whether the fibre twist and $J$-twist can be conjugate by a line bundle. The following theorem asserts that for the simplest type~$A$ case this may indeed hold, but that in general the two twists are not conjugate in this way, and so both will be needed to understand the derived autoequivalence group. 
\begin{thm}[=\ref{conj NC and fibre}]
Under the assumptions of \ref{Ftwists autoequivalence intro}, suppose further that $X\to X_{\con}$ contracts a single irreducible rational curve to a point $p$. Then there exists a functorial isomorphism
\[
\Tfib(x \otimes \conjlb) \cong \Tcon(x) \otimes \conjlb
\]
for some line bundle $\conjlb$ on $X$, if and only if the following conditions hold.
\begin{enumerate}
\item The point $p$ is cDV of Type $A$.
\item There exists a line bundle $\conjlb$ on $X$ such that $\deg(\conjlb|_{f^{-1}(p)})=-1$.
\end{enumerate}
\end{thm}

\subsection{Mutation and MMAs}
The construction of the functors $\Tcon$ and $\Tfib$ on the formal fibre $\mathfrak{U}\to\Spec \mathfrak{R}$, and the proof that they are twist autoequivalences, both rely on the theory of mutation from \cite[\S6]{IW4}.  As in \S\ref{2 actors}, there is a noncommutative ring $\AB:=\End_{\mathfrak{R}}(N )$ derived equivalent to $\mathfrak{U}$. Furthermore, as explained in \ref{N notation}, $N$ has a Krull--Schmidt decomposition $N=\bigoplus_{i=0}^{n_p} N_i$, where $N_0=\mathfrak{R}$ and the remaining $N_i$ are in one-to-one correspondence with the $n_p$ exceptional irreducible curves. 

To obtain the $J$-twist on $\mathfrak{U}$, we simply choose a subset of the curves, equivalently a subset $J\subseteq\{1,\hdots,n_p\}$.  Then, by mutating the module $N$ at $N_J:=\bigoplus_{j\in J}N_j$, we obtain a new module denoted $\upnu_JN$, see \S\ref{mut prelim} for details.   The following is an extension of \cite[\S5]{DW1} to the case $|J|>1$.
\begin{prop}[{=\ref{mu=nu for hyper}, \ref{pd for N and ext}}]\label{J mut intro} 
Suppose that $\mathfrak{U}\to\Spec \mathfrak{R}$ is a complete local $3$-fold flopping contraction, where $\mathfrak{U}$ has only Gorenstein terminal singularities.  Then for any $\indxset\subseteq \{1,\hdots,n_p\}$,
\begin{enumerate}
\item\label{J mut intro 1}  $\upnu_J\upnu_JN\cong N$.
\item\label{J mut intro 2} $\pd_{\AB}\AB_J=3$, and further 
\[
\Ext_{\AB}^t(\AB_J,S_j) 
\cong\left\{ \begin{array}{cl} \mathbb{C}&\mbox{if }t=0,3\\
0&\mbox{else,}\\  \end{array} \right.
\]
for all simple $\AB_J$-modules $S_j$.
\end{enumerate}
\end{prop}
Part \eqref{J mut intro 1} is the key to establishing that the $J$-twist is an autoequivalence, and part \eqref{J mut intro 2} shows the link to spherical functors (see \cite{AL1,AL2}).  Both parts follow quite easily from the fact that $\mathfrak{R}$ is a hypersurface singularity, using matrix factorisations.

However, to construct the fibre twist is significantly more complicated, since if we mutate the summand $\mathfrak{R}$ twice and obtain $\upnu_0\upnu_0N$, there is no easy reason why $\upnu_0\upnu_0N\cong N$ should be true.  Thus establishing the analogue of \ref{J mut intro}  is much harder, and requires an additional assumption of $\mathds{Q}$-factoriality.  It is well known that being $\mathds{Q}$-factorial does not pass to the formal fibre, so this presents a significant technical challenge; we postpone the discussion here and refer the reader to \ref{pd thm for fibre} and \ref{mu=nu for R}.

After constructing the autoequivalences $\Tcon$ and $\Tfib$ on the formal fibre, we then lift them to the Zariski local model.  For $\Tcon$, as in \cite[\S6]{DW1}, this lifting is quite easy, however lifting $\Tfib$ turns out to be more difficult, and requires a delicate argument that involves passing to a localisation and using lifting numbers.  In the proof that $\Tfib$ lifts, we rely heavily on the theory of maximal modification algebras (=MMAs), as recalled in \ref{MMintro}. These are the noncommutative version of minimal models.  It is an open problem as to whether the property of being an MMA passes to localizations at maximal ideals (it is known that it does not pass to the completion), however in this paper we do establish and use the following, which may be of independent interest.

\begin{thm}[=\ref{MMAsLocalize}]
Suppose that $R$ is a three-dimensional Gorenstein normal domain over $\mathbb{C}$, and that $\Lambda$ is derived equivalent to a $\mathds{Q}$-factorial terminalization of $\Spec R$.  Then $\Lambda$ is an MMA, and further for all $\m\in\Max R$, the algebra $\Lambda_\m$ is an MMA of $R_\m$.
\end{thm}

\subsection{Acknowledgements} We thank Alice Rizzardo for many discussions regarding the existence of adjoints in \S\ref{ZLT and GT section}.

\subsection{Conventions}\label{conventions}  We work over $\K$. Unqualified uses of the word \emph{module} refer to right modules, and $\mod A$ denotes the category of finitely generated right $A$-modules.    If $M\in\mod A$, we let $\add M$ denote all possible summands of finite sums of $M$.   If $S,T\in\mod R$ where $S$ is a summand of $T$, then we define the ideal $[S]$ to be the two-sided ideal of $\End_R(T)$ consisting of all morphisms factoring through $\add S$. Further, we use the functional convention for composing arrows, so $f\cdot g$ means $g$ then $f$.  With this convention, $M$ is an $\End_R(M)^{\op}$-module, $\Hom_R(M,N)$ is an $\End_R(M)$-module and an $\End_R(N)^{\op}$-module, in fact a bimodule.  Note also that $\Hom_R({}_{S}M_R, {}_{T}M_R)$ is an $S$--$T$ bimodule and $\Hom_{R^{\op}}({}_{R}M_S, {}_{R}M_T)$ is a $T$--$S$ bimodule.

If $X$ is a scheme, $\cO_{X,x}$ will denote the localization of the structure sheaf at the closed point $x\in X$,  whereas $\cO_x$ will always denote the skyscraper sheaf at $x$.  We will write $\widehat{\cO}_{X,x}$ for the completion of $\cO_{X,x}$ at the unique maximal ideal.  Throughout, {\em locally} will always mean Zariski locally, and when we discuss the completion, we will speak of working {\em complete locally}.

\section{Flops Setting and Notation}

In this section we fix notation, and provide the necessary preliminary results. 

\subsection{Perverse Sheaves}\label{SetupSection}
Consider a projective birational morphism $f\colon X\to X_{\con}$ between noetherian integral normal $\mathbb{C}$-schemes with $\Rf_*\cO_X=\cO_{X_\con}$, such that the fibres are at most one-dimensional.  

\begin{defin}\label{Per0 defin}
For such a morphism $f\colon X\to X_{\con}$, recall \cite{Bridgeland,VdB1d} that $\Per (X,X_{\con})$, the category of perverse sheaves on $X$, is defined
\[
\Per (X,X_{\con})=\left\{ a\in\Db(\coh X)\left| \begin{array}{c}H^i(a)=0\mbox{ if }i\neq 0,-1\\
f_*H^{-1}(a)=0\mbox{, }\Rfi{1}_* H^0(a)=0\\ \Hom(c,H^{-1}(a))=0\mbox{ for all }c\in\cC^0 \end{array}\right. \right\},
\]
where
\[
\cC:=\{ c\in\Db(\coh X)\mid \Rf_*c=0\}
\]
and $\cC^0$ denotes the full subcategory of $\cC$ whose objects have cohomology only in degree $0$. 
\end{defin}

\subsection{Global and Local Flops Notation}\label{global flops setup}
In this subsection, and for the remainder of this paper, we will make use of the following geometric setup.
\begin{setup}\label{flopsglobal}
(Global flops) We let $f\colon X\to X_{\con}$ be a flopping contraction, where $X$ is a quasi-projective $3$-fold with only Gorenstein terminal singularities in a neighbourhood of the curves contracted.  We write $\noniso f$ for the (finite) set of points in $X_{\con}$ above which $f$ is not an isomorphism.
\end{setup}

This is more general than the setup of \cite{DW1}, since the new adjoint technology of \cite{Alice} allows us to drop the assumption that $X$ is projective, and also there is no global restriction on singularities.  We will not assume that $X$ is $\mathds{Q}$-factorial unless explicitly stated.  With the assumptions in \ref{flopsglobal}, around each point $p\in \noniso f$ we can find an affine open neighbourhood $\Spec R$ containing $p$ but none of the other points in $\noniso f$, as shown below.
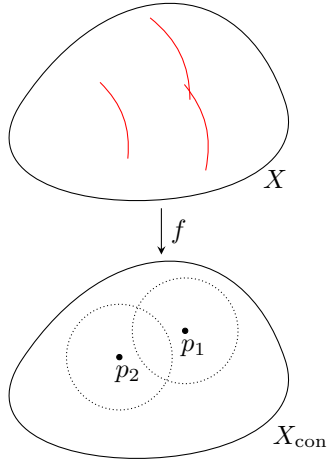
\def\singxshift{-2.5}
\def\singyshift{-0.5}
\def\totalxshift{0.75}
\def\totalyshift{0.25}
\def\overallscaling{0.7}
\[
\begin{tikzpicture}[>=stealth,scale=\overallscaling,xscale=1/2] 
\defThreeFold{\layoutstretch*\internalscaling*-0.5}{\internalscaling*2.5+\toplineshift}{black};
\node (a1) at (\layoutstretch*\internalscaling*-0.25+\totalxshift,\internalscaling*2.25+\toplineshift+\totalyshift) 
{\begin{tikzpicture} [xscale=\cutscaling*\internalscaling*0.75,yscale=\cutscaling*\internalscaling*0.75,bend angle=25, looseness=1,transform shape, rotate=\curvetilt]
\draw[red] (0,0,0) to [bend left=25] (1,0,-0.3);
\draw[red] (0.8,0,-0.3) to [bend left=25] (2,0,0);
\end{tikzpicture}};
\node (a1) at (\layoutstretch*\internalscaling*-0.25+\singxshift+\totalxshift,\internalscaling*2.25+\toplineshift+\singyshift+\totalyshift) 
{\begin{tikzpicture} [xscale=\cutscaling*\internalscaling*0.75,yscale=\cutscaling*\internalscaling*0.75,bend angle=25, looseness=1,transform shape, rotate=\curvetilt]
\draw[red] (0.5,0,0) to [bend left=25] (1.5,0,0);
\end{tikzpicture}};
\node at (\layoutstretch*\internalscaling*-0.1+6*\bottomlineshift+\totalxshift,\internalscaling*-0.2+\totalyshift) {$p_1$};
\node at (\layoutstretch*\internalscaling*-0.2+6*\bottomlineshift+\totalxshift,0+\totalyshift) 
{\begin{tikzpicture} [transform shape, rotate=-30]
\filldraw [black] (0.1,0.045,0) circle (1pt);
\draw [densely dotted] (0.1,0.045,0) circle (20pt);
\end{tikzpicture}};
\node at (\layoutstretch*\internalscaling*-0.1+6*\bottomlineshift+\singxshift+\totalxshift,\internalscaling*-0.2+\singyshift+\totalyshift) {$p_2$};
\node at (\layoutstretch*\internalscaling*-0.2+6*\bottomlineshift+\singxshift+\totalxshift,+\singyshift+\totalyshift) 
{\begin{tikzpicture} [transform shape, rotate=-30]
\filldraw [black] (0.1,0.045,0) circle (1pt);
\draw [densely dotted] (0.1,0.045,0) circle (20pt);
\end{tikzpicture}};
\defThreeFold{\layoutstretch*\internalscaling*-0.5+6*\bottomlineshift}{0}{black};

\draw[->] (\layoutstretch*\internalscaling*-0.25+2*\bottomlineshift,\internalscaling*1.05+\toplineshift)  -- node[right] {$\small f$} (\layoutstretch*\internalscaling*-0.25+4*\bottomlineshift,\internalscaling*0.5+\toplineshift);
\node at (\layoutstretch*\internalscaling*+1.4+6*\bottomlineshift,\internalscaling*-1.1) {$\small X_{\con}$};
\node at (\layoutstretch*\internalscaling*+1.1,\internalscaling*1.4+\toplineshift) {$\small X$};
\end{tikzpicture}
\]
We set $U:=f^{-1}(\Spec R)$ and thus consider the morphism of Gorenstein $3$-folds
\[
f|_{U}\colon U\to \Spec R.
\]
By construction, this is an isomorphism away from a single point, and above that point is a connected chain of rational curves.  Many of our global problems can be reduced to the following Zariski local setup.
\begin{setup}\label{flopslocal}
(Zariski local flops, single chain) Suppose that $f\colon U\to\Spec R$ is a crepant projective birational contraction of $3$-folds,  with fibres at most one-dimensional, which is an isomorphism away from precisely one point $\m\in\Max R$.  We assume that $U$ has only Gorenstein terminal singularities.  As notation, above $\m$ is a connected chain $C$ of $n_p$ curves with reduced scheme structure $C^{\redu}=\bigcup_{\indx=1}^{n_p} C_\indx$ such that each $C_\indx\cong\mathbb{P}^1$.  \end{setup}
Passing to the completion will bring technical advantages.
\begin{setup}\label{flopscompletelocal}
(Complete local flops) With notation and setup as in \ref{flopslocal}, we let $\mathfrak{R}$ denote the completion of $R$ at the maximal ideal $\m$.  We let $\clocCon\colon\mathfrak{U}\to \Spec \mathfrak{R}$ denote the formal fibre.  Above the unique closed point is a connected chain $C$ of $n_p$ curves with $C^{\redu}=\bigcup_{\indx=1}^{n_p} C_\indx$ such that each $C_\indx\cong\mathbb{P}^1$.  Because $R$ is Gorenstein terminal, necessarily $\mathfrak{R}$ is an isolated hypersurface singularity, by \cite[0.6(I)]{Pagoda}.
\end{setup}

We will often denote $n_p$ simply by $n$ where no confusion arises.

\subsection{The Contraction and Fibre Algebras} \label{localnotation}

In the Zariski local flops setup in \ref{flopslocal}, it is well-known \cite[3.2.8]{VdB1d} that there is a bundle $ \cV:=\cO_U\oplus\cN$  inducing a derived equivalence
\begin{eqnarray}
\begin{array}{c}
\begin{tikzpicture}
\node (a1) at (0,0) {$\Db(\coh U)$};
\node (a2) at (5,0) {$\Db(\mod \End_U(\cV))$};
\node (b1) at (0,-1) {$\Per(U,R)$};
\node (b2) at (5,-1) {$\mod\End_U(\cV) $};
\draw[->] (a1) -- node[above] {$\scriptstyle\RHom_U(\cV,-)$} node [below] {$\scriptstyle\sim$} (a2);
\draw[->] (b1) --  node [below] {$\scriptstyle\sim$} (b2);
\draw[right hook->] (b1) to (a1);
\draw[right hook->] (b2) to (a2);
\end{tikzpicture}
\end{array}\label{derived equivalence}
\end{eqnarray}
In this Zariski local setting there is choice in the construction of $\cV$, but in the setting of the formal fibre later in \S\ref{completelocalnotation}, there is a canonical choice.   

Throughout we set
\[
\Lambda:=\End_U(\cV) = \End_U(\cO_U\oplus\cN).
\]
Recall that if $\cF,\cG\in\coh U$ where $\cF$ is a summand of $\cG$, then we define the ideal $[\cF]$ to be the two-sided ideal of $\End_U(\cG)$ consisting of all morphisms factoring through $\add \cF$.

\begin{defin}\label{contraction algebra def 1}
With notation as above, we define the \emph{contraction algebra} associated to $\Lambda$ to be $\Lambda_{\con}:=\End_U(\cO_U\oplus\cN)/[\cO_U]$.
\end{defin}

We remark that $\Lambda_{\con}$ defined above depends on $\Lambda$ and thus the choice of derived equivalence \eqref{derived equivalence}.  In the complete local case there is a canonical choice for this; see \S\ref{completelocalnotation} below.  

\begin{lemma}\label{Lambdacon idempotents}
Under the Zariski local setup of \ref{flopslocal}, the basic algebra morita equivalent to $\Lambda_{\con}$ has precisely $n$ primitive idempotents.
\end{lemma}
\begin{proof}
As in \cite[2.17]{DW1}, $\Lambda_{\con}\cong\widehat{\Lambda}_{\con}$, and on the completion the assertion is clear, for example using \eqref{a_i decomp} below.
\end{proof}

Since $f$ is a flopping contraction, it follows from \cite[4.2.1]{VdB1d} that
\begin{eqnarray}
\Lambda=\End_{U}(\cV)\cong\End_{R}(R\oplus f_*\cN). \label{up=down VdB}
\end{eqnarray}
We set $L:=f_*\cN$ so that $\Lambda\cong\End_R(R\oplus L)$.  Translating \ref{contraction algebra def 1} through this isomorphism, the contraction algebra associated to $\Lambda$ becomes
\[
\Lambda_{\con}\cong \End_R(R\oplus L)/[R]\cong\End_R(L)/[R].
\]
There is a canonical ring homomorphism $\Lambda\to\Lambda_{\con}$, and we denote its kernel by $I_{\con}$.  Necessarily this is a two-sided ideal of $\Lambda$, so there is a short exact sequence
\begin{eqnarray}
0\to I_{\con}\to\Lambda\to\Lambda_{\con}\to 0\label{Icon ses 1}
\end{eqnarray}
of $\Lambda$-bimodules.  For global twist functors later, we need a more refined version of this.  Under the morita equivalence in \ref{Lambdacon idempotents}, $\Lambda_\con$ inherits $n$ primitive idempotents $e_1,\hdots,e_n$. We pick a subset $J\subseteq\{1,\hdots,n\}$ of these idempotents (equivalently, a subset of the exceptional curves in \ref{flopslocal}), and write
\[
\Lambda_J:=\frac{\Lambda_{\con}}{\Lambda_{\con}(1-\sum_{j\in J}e_j)\Lambda_{\con}}.
\]
The composition of the ring homomorphisms $\Lambda\to\Lambda_{\con}\to\Lambda_J$ is a surjective ring homomorphism.  We denote the kernel by $I_J$, which is a two-sided ideal of $\Lambda$, and thus for each subset $J\subseteq\{1,\hdots,n\}$, there is a short exact sequence 
\begin{eqnarray}
0\to I_J\to\Lambda\to\Lambda_J\to 0\label{Icon ses 1J}
\end{eqnarray}
of $\Lambda$-bimodules.  This exact sequence is needed later to extract a Zariski local twist functor from the formal fibre.

Naively, we want to repeat the above analysis for $\Lambda/[L]=\End_R(R\oplus L)/[L]$ to obtain a similar exact sequence, as this later will give another, different, derived autoequivalence.  However, by the failure of Krull--Schmidt it may happen that $R\in\add L$, in which case $\Lambda/[L]=0$.   We get round this problem by passing to the localization $\Lambda_\m$, where we can use lifting numbers.  Indeed, localizing $\Lambda$ with respect to the maximal ideal $\m$ in setup \ref{flopslocal} gives $\Lambda_\m\cong\End_{R_\m}(R_\m\oplus L_\m)$, so we can then use the following well-known lemma.

\begin{lemma}\label{take R out}
Suppose that $(S,\m)$ is a commutative noetherian local ring, and suppose that $T\in\mod S$.  Then we can write $T\cong S^{a}\oplus X$ for some $a\geq 0$ with $S\notin\add X$.
\end{lemma}
\begin{proof}
Since $\widehat{S}$ is complete local, we can write $\widehat{T}\cong \widehat{S}^{\oplus a}\oplus X_1^{\oplus a_1}\oplus\cdots\oplus X_n^{\oplus a_n}$ as its Krull--Schmidt decomposition into pairwise non-isomorphic indecomposables, where $a\geq 0$ and all $a_i\geq 1$. Since $\widehat{S}^{\oplus a}$ is a summand of $\widehat{T}$, it follows that $S^{\oplus a}$ is a summand of $T$ (see e.g.\ \cite[1.15]{LW}), so we can write $T\cong S^{\oplus a}\oplus X$ for some $X$.  Completing this decomposition, again by Krull--Schmidt $\widehat{X}\cong X_1^{\oplus a_1}\oplus\cdots\oplus X_n^{\oplus a_n}$ and so $\widehat{S}\notin\add\widehat{X}$.  It follows that $S\notin\add X$.
\end{proof}

By \ref{take R out}, we may pull out all free summands from $L_\m$, and write 
\[
\Lambda_\m\cong\End_{R_\m}(R_\m^{\oplus a}\oplus K)
\]
for some $K\in\mod R_\m$ with $R_\m\notin\add K$. 

\begin{defin}\label{Lambda fib def}
With the Zariski local flops setup in \ref{flopslocal}, we define the fibre algebra $\Lambda_{\fib}$ to be $\Lambda_\m/[K]$.
\end{defin}

Since by construction $R_\m\notin\add K$, the fibre algebra $\Lambda_{\fib}$ is non-zero.  Furthermore, the composition of ring homomorphisms
\[
\Lambda\xrightarrow{\uppsi_1}\Lambda_\m\xrightarrow{\uppsi_2}\Lambda_\m/[K]=\Lambda_{\fib}
\]
is a ring homomorphism, and we define $I_{\fib}:=\Ker(\uppsi_2\uppsi_1)$.
\begin{lemma}\label{fib ses bimodules}
With the Zariski local flops setup as in \ref{flopslocal}, as an $R$-module $\Lambda_{\fib}$ is supported only at $\m$, and further there is a short exact sequence
\begin{eqnarray}
0\to I_{\fib}\to\Lambda\to\Lambda_{\fib}\to 0\label{Ifib ses 1}
\end{eqnarray}
of $\Lambda$-bimodules.  
\end{lemma}
\begin{proof}
Since $R_\m$ is an isolated singularity, $\Lambda_{\fib}=\Lambda_\m/[K]$ has finite length as an $R_\m$-module by \cite[6.19(3)]{IW4}.  Thus, as an $R$-module, it is supported only at $\m$.  The only thing that remains to be proved is that $\uppsi_2\uppsi_1$ is surjective.  If we let $C$ denote the cokernel, then since $\Lambda_{\fib}$ is supported only at $\m$, $(\Lambda_{\fib})_\n=0$ for all $\n\in\Max R$ with $\n\neq\m$, and so  $C_\n=0$ for all $\n\neq\m$.  But on the other hand $(\uppsi_1)_\m$ is an isomorphism, and $(\uppsi_2)_\m$ is clearly surjective, hence $(\uppsi_2\uppsi_1)_\m$ is also surjective, so $C_\m=0$.  Hence $C=0$.
\end{proof}

\subsection{Complete Local Derived Category Notation}\label{completelocalnotation}
In this subsection we consider the complete local flops setup in \ref{flopscompletelocal}, and fix notation.  Completing the base in \ref{flopslocal} with respect to $\m$ gives $\Spec \mathfrak{R}$ and we consider the formal fibre $\clocCon\colon\mathfrak{U}\to\Spec\mathfrak{R}$.  The above derived equivalence \eqref{derived equivalence} induces an equivalence
\[
\begin{array}{c}
\begin{tikzpicture}
\node (a1) at (0,0) {$\Db(\coh\mathfrak{U})$};
\node (a2) at (4,0) {$ \Db(\mod\widehat{\Lambda}).$};
\draw[->] (a1) -- node[above] {$\scriptstyle\RHom_{\mathfrak{U}}(\widehat{\cV},-)$} node [below] {$\scriptstyle\sim$} (a2);
\end{tikzpicture}
\end{array}
\]
This can be described much more explicitly.  Let $C=\clocCon^{-1}(\m)$ where $\m$ is the unique closed point of $\Spec \mathfrak{R}$. Then the reduced scheme $C^{\redu}=\bigcup _{\indx=1}^n C_\indx$ with $C_\indx\cong\mathbb{P}^1$. Let $\cL_\indx$ denote a line bundle on $\mathfrak{U}$ such that $\cL_\indx\cdot C_i=\updelta_{\indx i}$.  If the multiplicity of $C_\indx$ is equal to one, set $\cM_\indx:=\cL_\indx$, else define $\cM_\indx$ to be given by the maximal extension
\[
0\to\cO_{\mathfrak{U}}^{\oplus(r-1)}\to\cM_\indx\to\cL_\indx\to 0
\]
associated to a minimal set of $r-1$ generators of $H^1(\mathfrak{U},\cL_\indx^{*})$ \cite[3.5.4]{VdB1d}. 

\begin{notation}\label{N notation}
In the complete local flops setup of \ref{flopscompletelocal}, 
\begin{enumerate}
\item Set $\cN_\indx:=\cM_\indx^*$, $\cN_0:=\cO_{\mathfrak{U}}$ and $\cV_{\mathfrak{U}}:=\bigoplus_{\indx=0}^n\cN_\indx$.
\item Set $N_\indx:=H^0(\cN_\indx)$ and $N:=H^0(\cV_\mathfrak{U})$.  Set $N_0:=\mathfrak{R}$, so that $N=\bigoplus_{\indx=0}^n N_\indx$.
\item Put $\AB:=\End_{\mathfrak{R}}(N)$.
\end{enumerate}
\end{notation}
By \cite[3.5.5]{VdB1d}, $\cV_{\mathfrak{U}}$ is a basic progenerator of $\Per(\mathfrak{U},\mathfrak{R})$, and furthermore is a tilting bundle on $\mathfrak{U}$.   The rank of $N_\indx$ as an $\mathfrak{R}$-module,   $\rank_{\mathfrak{R}}^{\phantom 1}N_\indx$, is equal to the scheme-theoretic multiplicity of the curve $C_\indx$ \cite[3.5.4]{VdB1d}. Further, by \cite[3.5.5]{VdB1d} we can write
\begin{eqnarray}\label{complete local tilting}
\widehat{\cV}\cong \cO_{\mathfrak{U}}^{\oplus a_0}\oplus \bigoplus_{\indx=1}^n \cN_\indx^{\oplus a_\indx}\label{a_i decomp}
\end{eqnarray}
for some $a_\indx\in\mathbb{N}$ and so consequently $\widehat{\Lambda}\cong\End_{\mathfrak{R}}(\bigoplus_{\indx=0}^n N_\indx^{\oplus a_\indx})$, hence $\AB$ is the basic algebra morita equivalent to $\widehat{\Lambda}$.

\begin{defin}\label{def basic algebra}
For any $\indxset\subseteq \{0,1,\hdots,n\}$, set
\begin{enumerate}
\item $\cN_\indxset:=\bigoplus_{\indx\in \indxset}\cN_\indx$ and $\cN_{\indxset^c}:=\bigoplus_{\otherindx\notin \indxset}\cN_\otherindx$, so that $\cV_\mathfrak{U}=\cN_\indxset\oplus \cN_{\indxset^c}$.
\item $N_\indxset:=\bigoplus_{\indx\in \indxset}N_\indx$ and $N_{\indxset^c}:=\bigoplus_{\otherindx\notin \indxset}N_\otherindx$, so that $N=N_\indxset\oplus N_{\indxset^c}$.
\setcounter{tempenum}{\theenumi}
\end{enumerate}
There are three important special cases, namely
\begin{enumerate}
\setcounter{enumi}{\thetempenum} 
\item When $\indxset\subseteq\{1,\hdots,n\}$, where we write
\[
\AB_{\indxset}:=\End_{\mathfrak{R}}(N)/[N_{\indxset^c}],
\]
and call $\AB_\indxset$ the {\em contraction algebra associated to $\bigcup_{j\in J}C_j$}.
\item When $\indxset= \{1,\hdots,n\}$, where we call $\AB_\indxset$ the {\em contraction algebra associated to $\clocCon$}, and denote it by 
\[
\CA\cong  \End_{\mathfrak{R}}(N)/[\mathfrak{R}]\cong  \End_{\mathfrak{R}}\left(\oplus_{\indx=1}^n N_\indx\right)/[\mathfrak{R}].
\]
\item When $\indxset=\{0\}$, where we write
\[
\AB_{\fib}:=  \End_{\mathfrak{R}}(N)/[\oplus_{\indx=1}^nN_\indx]\cong\End_{\mathfrak{R}}(\mathfrak{R})/[\oplus_{\indx=1}^nN_\indx]\cong \mathfrak{R}/[\oplus_{\indx=1}^nN_\indx],
\]
and call $\AB_{\fib}$ the {\em fibre algebra associated to $\clocCon$}.
\end{enumerate}
\end{defin}

It follows from the definition that $\AB_{\fib}$ is always commutative.

\begin{remark}
We conjectured in \cite[1.4]{DW1} that $\CA$ distinguishes the analytic type of irreducible contractible flopping curves.  We remark here that $\AB_{\fib}$ certainly does not, as for example inspecting the Laufer flop in \cite[1.3]{DW1} we see that $\AB_{\fib}=\mathbb{C}$, but it is also well-known that $\AB_{\fib}=\mathbb{C}$ for the Atiyah flop.  This already demonstrates that $\CA$ gives more information than $\AB_{\fib}$, although both play a role in the derived symmetry group.
\end{remark}

The following is shown in a similar way to \cite[2.16]{DW1}.

\begin{lemma}\label{Acon and Afib under FBC}
With the complete local setup as above,  for all $\p\in\Spec\mathfrak{R}$,
\begin{enumerate}
\item $(\AB_J)_\p\cong (\AB_\p)_{J}$ and $(\AB_J)_\p\otimes_{\mathfrak{R}_\p}\widehat{\mathfrak{R}_\p}\cong (\widehat{\AB_\p})_{J}$ for all $J\subseteq\{1,\hdots,n\}$.
\item $(\AB_{\fib})_\p\cong (\AB_\p)_{\fib}$ and $(\AB_{\fib})_\p\otimes_{\mathfrak{R}_\p}\widehat{\mathfrak{R}_\p}\cong (\widehat{\AB_\p})_{\fib}$.
\end{enumerate}
\end{lemma}

\begin{notation}\label{TandSnotation}
For the Zariski local flops setup \ref{flopslocal}, and its formal fibre \ref{flopscompletelocal}, as notation for the remainder of this paper,
\begin{enumerate}
\item Write $\mathbb{F}\colon\Mod\AB\to\Mod\widehat{\Lambda}$ for the morita equivalence induced from \eqref{complete local tilting}.
\item For each $i\in\{1,\hdots,n\}$ set $E_i:=\cO_{C_i}(-1)$, considered as a complex in degree zero.  Further, set $E_0:=\omega_C[1]$. 
\item For each $i\in\{0,1,\hdots,n\}$, define the simple $\Lambda$-modules $T_i$ to be the modules corresponding to the perverse sheaves $E_i$ across the derived equivalence \eqref{derived equivalence}.  Further, set $S_i:=\mathbb{F}^{-1}\widehat{T}_i$, which are the corresponding simple $\AB$-modules.
\item  For any $J\subseteq\{1,\hdots,n\}$, viewing $\mathbb{F}\AB_J$ and $\mathbb{F}\AB_{\fib}$ as $\Lambda$-modules via restriction of scalars, we put $\cE_J:=\mathbb{F}\AB_J\otimes_{\Lambda}^{\bf L}\cV$ and $\cE_{\fib}:=\mathbb{F}\AB_{\fib}\otimes_{\Lambda}^{\bf L}\cV$.  As in \cite[3.7]{DW1}, both are perverse sheaves, with $\cE_J$ concentrated in degree zero, and $\cE_\fib$ concentrated in degree $-1$.
\end{enumerate}

The above notation may be summarised as follows.
\begin{equation}\label{notation summary}
\begin{array}{c}
\begin{tikzpicture}[xscale=1]
\node (d1) at (2,0.8) {$\Db(\Qcoh U)$};
\node (e1) at (6,0.8) {$\Db(\Mod \Lambda)$};
\node (f1) at (9,0.8) {$\Db(\Mod \widehat{\Lambda})$};
\node (g1) at (12,0.8) {$\Db(\Mod \AB)$};
\draw[->] (d1.5) to  node[above] {$\scriptstyle\RHom_U(\cV,-)$} (e1.175);
\draw[<-] (d1.-5) to node [below]  {$\scriptstyle-\otimes_{\Lambda}^{\bf L}\cV$} (e1.-175);
\draw[->] (e1.5) to  node[above] {$\scriptstyle \otimes_{R_\m}\widehat{R}$} (f1.175);
\draw[<-] (e1.-5) to node [below]  {$\scriptstyle\rest$} (f1.-175);
\draw[<-]  (f1) -- node[above] {$\scriptstyle \mathbb{F}$} (g1);
\node (d2) at (2,0) {$E_i$};
\draw[<->] (d1.east |- 3.5,0) -- (e1.west |- 5.5,0);
\node (e2) at (6,0) {$T_i$};
\draw[|->] (e1.east |- 6.5,0) -- (f1.west |- 8.5,0);
\node (f2) at (9,0.08) {$\widehat{T}_i$};
\draw[<->] (f1.east |- 9.5,0) -- (g1.west |- 11.25,0);
\node (g2) at (12,0) {$S_i$};
\node (d3) at (2,-0.5) {$\cE_J$};
\draw[<->] (d1.east |- 3.5,-0.5) -- (e1.west |- 5.25,-0.5);
\node (e3) at (6.1,-0.5) {$\mathbb{F}\AB_J$};
\draw[<-|] (e1.east |- 6.85,-0.5) -- (f1.west |- 8.5,-0.5);
\node (f3) at (9.1,-0.5) {$\mathbb{F}\AB_J$};
\draw[<->] (f1.east |- 9.5,-0.5) -- (g1.west |- 11.25,-0.5);
\node (g3) at (12,-0.5) {$\AB_J$};

\node (d4) at (2,-1) {$\cE_\fib$};
\draw[<->] (d1.east |- 3.5,-1) -- (e1.west |- 5.25,-1);
\node (e4) at (6.1,-1) {$\mathbb{F}\AB_\fib$};
\draw[<-|] (e1.east |- 6.85,-1) -- (f1.west |- 8.5,-1);
\node (f4) at (9.1,-1) {$\mathbb{F}\AB_\fib$};
\draw[<->] (f1.east |- 9.5,-1) -- (g1.west |- 11.25,-1);
\node (g4) at (12,-1) {$\AB_\fib$};
\end{tikzpicture}
\end{array}
\end{equation}

\end{notation}

\subsection{Maximal Modification Algebras}\label{MMA prelim}
Later the fibre twist autoequivalence requires a restriction to $\mathds{Q}$-factorial singularities and some other technical results, which we review here.  For a commutative noetherian local ring $(R,\m)$ and $M\in\mod R$ recall that the \emph{depth} of $M$ is defined to be
\[
\depth_R M:=\inf \{ i\geq 0\mid \Ext^i_R(R/\m,M)\neq 0 \},
\]
which coincides with the maximal length of an $M$-regular sequence.  Keeping the assumption that $(R,\m)$ is local we say that $M\in\mod R$ is \emph{maximal Cohen--Macaulay} (=CM) if $\depth_R M=\dim R$.  

In the non-local setting, if $R$ is an arbitrary commutative noetherian ring we say that $M\in\mod R$ is \emph{CM} if $M_\p$ is CM for all prime ideals $\p$ in $R$, and we denote the category of CM $R$-modules by $\CM R$.  We say that $R$ is a \emph{CM ring} if $R\in\CM R$, and if further $\id_{R}R<\infty$, we say that $R$ is \emph{Gorenstein}. We write $\refl R$ for the category of reflexive $R$-modules.

Recall the following \cite{IW4}.
\begin{defin}\label{MMintro}
Suppose that $R$ is a $d$-dimensional CM ring.  We call $N\in\refl R$ a \emph{maximal modifying (=MM) module} if 
\[
\add N=\{ X\in\refl R\mid \End_{R}(N \oplus X)\in\CM R  \}.
\]
If $N$ is an MM module, we call $\End_R(N)$ a \emph{maximal modification algebra (=MMA).}  
\end{defin}

Throughout, we say a normal scheme $X$ is {\em $\mathds{Q}$-factorial} if for every Weil divisor $D$, there exists $n\in\mathbb{N}$ for which $nD$ is Cartier; this condition can be checked on the stalks $\cO_{X,x}$ of the closed points $x\in X$.  However this property is not complete local, so we say  $X$ is {\em complete locally $\mathds{Q}$-factorial} if the completion $\widehat{\cO}_{X,x}$ is $\mathds{Q}$-factorial for all closed points $x\in X$.

Recall that if $X$ and $Y$ are varieties over $\K$, then a projective birational morphism $f\colon Y\to X$ is called {\em crepant} if $f^*\omega_X=\omega_Y$.  A {\em $\mathds{Q}$-factorial terminalization} of $X$ is a crepant projective birational morphism $f\colon Y\to X$ such that $Y$ has only $\mathds{Q}$-factorial terminal singularities. When $Y$ is furthermore smooth, we call $f$ a {\em crepant resolution}. 

In our geometric setup later, we will require that localizations of MMAs are MMAs.  There is currently no known purely algebraic proof of this, but the following geometric proof suffices for our purposes.
\begin{thm}\label{MMAsLocalize}
Suppose that $R$ is a three-dimensional Gorenstein normal domain over $\mathbb{C}$, and that $\Lambda$ is derived equivalent to a $\mathds{Q}$-factorial terminalization of $\Spec R$.  Then for all $\m\in\Max R$, $\Lambda_\m$ is an MMA of $R_\m$.
\end{thm}
\begin{proof}
For $\m\in\Max R$, after base change
\[
\begin{array}{c}
\begin{tikzpicture}[yscale=1.25]
\node (Xp) at (-1,0) {$X^\prime$}; 
\node (X) at (1,0) {$X$};
\node (Rp) at (-1,-1) {$\Spec R_\m$}; 
\node (R) at (1,-1) {$\Spec R$};
\draw[->] (Xp) to node[above] {$\scriptstyle k$} (X);
\draw[->] (Rp) to node[above] {$\scriptstyle j$} (R);
\draw[->] (Xp) --  node[left] {$\scriptstyle \upvarphi$}  (Rp);
\draw[->] (X) --  node[right] {$\scriptstyle f$}  (R);
\end{tikzpicture}
\end{array}
\]
$\Lambda_\m$ is derived equivalent to $X^\prime$.  But the stalks of $\cO_{X^\prime}$ are isomorphic to stalks of $\cO_X$, and so in particular all stalks of  $\cO_{X^\prime}$ are isolated $\mathds{Q}$-factorial hypersurfaces.  By \cite[3.2(1)]{IW5}
\[
\Dsg(\Lambda_{\m})\hookrightarrow\bigoplus_{x\in\Sing X^\prime}\uCM \cO_{X^\prime,x}
\]  
and so $\Dsg(\Lambda_\m)$ is rigid-free since each $\uCM \cO_{X^\prime,x}$ is \cite[3.1(1)]{DaoNCCR}.  Since $\Lambda$ has isolated singularities \cite[4.2(2)]{IW5}, this implies that $\Lambda_\m$ is an MMA \cite[2.14]{IW5}.
\end{proof}

\subsection{Mutation Notation}\label{mut prelim}
Throughout this subsection we consider the complete local flops setting \ref{flopscompletelocal}, and use notation from \ref{N notation} and \ref{def basic algebra}, so in particular $\mathfrak{R}$ is an isolated complete local Gorenstein $3$-fold, and $\AB:=\End_{\mathfrak{R}}(N)$.   We set $(-)^*:=\Hom_\mathfrak{R}(-,\mathfrak{R})$.

Given our choice of summand $N_J$, we then mutate.  

\begin{setup}\label{setup2} As in \ref{def basic algebra}, write $N_J=\bigoplus_{j\in J}N_j$ as a direct sum of indecomposables.  For each $j\in J$, consider a \emph{minimal right $(\add N_{J^c})$-approximation}
\[
V_j\xrightarrow{a_j}N_j
\]
of $N_j$, which by definition means that
\begin{enumerate}
\item $V_j\in\add N_{J^c}$ and $(a_j\cdot)\colon\Hom_\mathfrak{R}(N_{J^c},V_j)\to\Hom_\mathfrak{R}(N_{J^c},N_j)$ is surjective.
\item If $g\in\End_\mathfrak{R}(V_j)$ satisfies $a_j=a_jg$, then $g$ is an automorphism.
\end{enumerate}
Since $\mathfrak{R}$ is complete, such an $a_j$ exists and is unique up to isomorphism. Thus there are exact sequences
\begin{gather}
0\to \Ker a_j\xrightarrow{c_j} V_j\xrightarrow{a_j}N_j\label{K0prime}\\
\nonumber \phantom{.}0\to \Hom_\mathfrak{R}(N_{J^c},\Ker a_j)\xrightarrow{c_j\cdot } \Hom_\mathfrak{R}(N_{J^c},V_j)\xrightarrow{a_j\cdot }\Hom_\mathfrak{R}(N_{J^c},N_j)\to 0.
\end{gather}
Summing the sequences \eqref{K0prime} over all $j\in J$ gives exact sequences
\begin{gather}
0\to \Ker a_J\xrightarrow{c_J} V_J\xrightarrow{a_J}N_J\label{K0}\\
\nonumber \phantom{.}0\to \Hom_\mathfrak{R}(N_{J^c},\Ker a_J)\xrightarrow{c_J\cdot } \Hom_\mathfrak{R}(N_{J^c},V_J)\xrightarrow{a_J\cdot }\Hom_\mathfrak{R}(N_{J^c},N_J)\to 0.
\end{gather}
Note that applying $\Hom_\mathfrak{R}(N,-)$ to \eqref{K0} yields an exact sequence
\begin{eqnarray}
0\to \Hom_\mathfrak{R}(N,\Ker a_J)\xrightarrow{c_J\cdot} \Hom_\mathfrak{R}(N,V_J)\xrightarrow{a_J\cdot}\Hom_\mathfrak{R}(N,N_J)\to\AB_J\to 0\label{begin lambdacon} 
\end{eqnarray}
of $\AB$-modules.

Dually, for each $j\in J$, consider a minimal right $(\add N_{J^c}^*)$-approximation
\[
U_j^*\xrightarrow{b_j}N_j^*
\]
of $N_j^*$, thus 
\begin{gather}
\nonumber 0\to \Ker b_j\xrightarrow{d_j} U_j^*\xrightarrow{b_j}N_j^*\\
\nonumber 0\to \Hom_\mathfrak{R}(N_{J^c}^*,\Ker b_j)\xrightarrow{d_j\cdot }
\Hom_\mathfrak{R}(N_{J^c}^*,U_j^*)\xrightarrow{b_j\cdot }
\Hom_\mathfrak{R}(N_{J^c}^*,N_j^*) 
\to 0
\end{gather}
are exact.  Summing over all $j\in J$ gives exact sequences
\begin{gather}
0\to \Ker b_J\xrightarrow{d_J} U_J^*\xrightarrow{b_J}N_J^*\label{K1}\\
\nonumber \phantom{.}0\to \Hom_\mathfrak{R}(N_{J^c}^*,\Ker b_J)\xrightarrow{d_J\cdot }
\Hom_\mathfrak{R}(N_{J^c}^*,U_J^*)\xrightarrow{b_J\cdot }
\Hom_\mathfrak{R}(N_{J^c}^*,N_J^*) 
\to 0.
\end{gather}
\end{setup}

\begin{defin}\label{mut defin main}
With notation as above, in particular $\AB:=\End_\mathfrak{R}(N)$, we define the \emph{left mutation} of $N$ at $N_J$ as
\[
\nuJ N:=N_{J^c}\oplus (\Ker b_J)^*,
\]
and set $\nuJ\AB:=\End_\mathfrak{R}(\nuJ N)$.
\end{defin}

One of the key properties of mutation is that it always gives rise to a derived equivalence. With the setup as above, the derived equivalence between $\AB$ and $\nuJ\AB$ is given by a tilting $\AB$-module $T_J$ constructed as follows.  There is an exact sequence
\begin{eqnarray}
0\to N_J\xrightarrow{b_J^*} U_J\xrightarrow{d_J^*} (\Ker b_J)^*\label{key track mut}
\end{eqnarray}
obtained by dualizing \eqref{K1}.  Applying $\Hom_\mathfrak{R}(N,-)$ induces $(b_J^*\cdot )\colon\Hom_\mathfrak{R}(N,N_J)\to\Hom_\mathfrak{R}(N,U_J)$, so denoting the cokernel by $C_J$ there is an exact sequence
\begin{eqnarray}
0\to \Hom_\mathfrak{R}(N,N_J)\xrightarrow{ b_J^*\cdot} \Hom_\mathfrak{R}(N,U_J)\to C_J\to 0.
\label{defin of CI}
\end{eqnarray}
The tilting $\AB$-module $T_J$ is defined to be $
T_J:=\Hom_\mathfrak{R}(N,N_{J^c})\oplus C_J$.  It turns out that $\End_{\AB}(T_J)\cong \nuJ\AB$ \cite[6.7, 6.8]{IW4}, and there is always an equivalence 
\[
\Upphi_J:=\RHom(T_J,-)\colon\Db(\mod\AB)\xrightarrow{\sim} \Db(\mod\nuJ\AB),
\]
which is called the {\em mutation functor} \cite[6.8]{IW4}.

\section{On the Braiding of Flops}

In this section we establish the braiding of flop functors in dimension three, and describe the combinatorial objects that allow us to read off the length of the braid relations which appear.  To account for the inconvenient fact that, algebraically, curves are often forced to flop together, we begin by establishing braiding in the complete local case, before addressing the Zariski local and global cases later.  Thus throughout this section, we will assume that the curves are all individually floppable.

\subsection{Moduli Tracking}\label{moduli tracking section}
Throughout this subsection we consider the complete local flops setup of \ref{flopscompletelocal}, and use the notation from \S\ref{completelocalnotation}.  Thus there is a flopping contraction $\clocCon\colon \mathfrak{U}\to \Spec \mathfrak{R}$ of $3$-folds, where $\mathfrak{U}$ has Gorenstein terminal singularities, and $\mathfrak{U}$ is derived equivalent to $\AB:=\End_{\mathfrak{R}}(N)$ from \ref{N notation}, where $N=\bigoplus_{i=0}^nN_i$.  The $N_i$ with $i>0$ correspond to the $n$ curves in the exceptional locus.

To prove the braiding of flops in this setting, we will heavily use the following moduli tracking result.  As notation, we present $\AB$ as a quiver with relations, and consider King stability; see for example \cite[\S5]{HomMMP} for a brief introduction in this setting. In particular, we define the dimension vector $\rk\AB:=(\rank_{\mathfrak{R}}N_i)_{i=0}^n$, and denote the space of stability parameters by $\Uptheta$.  

Given $\upvartheta\in\Uptheta$, denote by $\cS_{\upvartheta}(\AB)$  the full subcategory of finite length $\AB$-modules which has as objects the $\upvartheta$-semistable objects, and denote by $\cS_{\rk,\upvartheta}(\AB)$  the full subcategory of $\cS_{\upvartheta}(\AB)$ consisting of those objects with dimension vector $\rk\AB$.  We let $\cM_{\rk,\upvartheta}(\AB)$ denote the moduli space of $\upvartheta$-stable $\AB$-representations of dimension vector $\rk\AB$.  

Since by definition any stability condition satisfies $\upvartheta\cdot\rk\AB=0$, the fact that $N_0=\mathfrak{R}$ has rank one then implies that
\[
\upvartheta_{0}=-\sum_{i=1}^n(\rank_{\mathfrak{R}} N_i)\upvartheta_i
\]
and so  $\Uptheta=\mathds{Q}^{n}$, with co-ordinates $\upvartheta_i$ for each $i=1,\hdots,n$.

\begin{thm}[{Moduli Tracking}]\label{moduli main}
In the complete local flops setup of \ref{flopscompletelocal}, choose a subset of curves $J$, equivalently a subset $J\subseteq\{1,\hdots,n\}$.  Applying the mutation setup of \ref{setup2} to $N=\bigoplus_{i=0}^nN_i$ with summand $N_J:=\bigoplus_{j\in J}N_j$, consider the mutation exchange sequence~\eqref{key track mut}
\[
0\to N_j\to U_j\to (\Ker b_j)^*,
\]
for each $j\in J$. By Krull--Schmidt, $U_j$ decomposes into
\[
U_j\cong \bigoplus_{i\notin J}N_i^{\oplus b_{j,i}}
\]
for some collection of $b_{j,i}\geq 0$. Write $\boldb$ for the data $(b_{j,i})$ with $j\in J$, $i\notin J$. Then for any stability parameter $\upvartheta\in\Uptheta$ define the vector $\upnu_{\boldb}\upvartheta$ by
\[
(\upnu_{\boldb}\upvartheta)_i=\left\{\begin{array}{cl} \upvartheta_i+\sum_{j\in J} b_{j,i}\upvartheta_j&\mbox{if }i\notin J \vspace{0.1cm} \\ -\upvartheta_i&\mbox{if }i\in J. \end{array}\right.
\]
If $\upvartheta_j>0$ for all $j\in J$, then
\begin{enumerate}
\item\label{moduli main 1} The mutation functor $\Upphi_J$ restricts to a categorical equivalence
\[
\begin{tikzpicture}
\node (A) at (0,0) {$\mathcal{S}_{\rk,\upvartheta}(\AB)$};
\node (B) at (2.75,0) {$\mathcal{S}_{\rk,\upnu_{\boldb}\upvartheta}(\nuJ\AB)$};
\draw[->] (0.75,0) -- node [above] {$\scriptstyle \Upphi_J$} (1.5,0);
\end{tikzpicture}
\]
where the left-hand side has dimension vector $\rk\AB$, and the right-hand side dimension vector $\rk\nuJ\AB$.  This categorical equivalence preserves $S$-equivalence classes, and  $\upvartheta$-stable modules correspond to $\upnu_{\boldb}\upvartheta$-stable modules. Further, $\upvartheta$ is generic if and only if $\upnu_{\boldb}\upvartheta$ is generic.
\item\label{moduli main 2}  As schemes
\[
\cM_{\rk,\upvartheta}(\AB)\cong \cM_{\rk,\upnu_{\boldb}\upvartheta}(\nuJ\AB).
\]
\end{enumerate}
\end{thm}
\begin{proof}
(1) This is a special case of \cite[5.12]{HomMMP}, using \cite[2.25, 3.5]{HomMMP} to see that the assumption~(b) of \cite[5.12]{HomMMP} is satisfied.\\
(2) This is \cite[5.13]{HomMMP}.
\end{proof}

The stability parameter space $\Uptheta$ has a wall and chamber structure, and the combinatorics of this turns out to control the braiding.  The strictly semi-stable parameters cut out codimension-one walls, separating the generic stability conditions into chambers.  Within a given chamber, the set of semi-stable representations does not vary.  The following is known.
\begin{prop}\label{C+chamber}
In the complete local flops setup of \ref{flopscompletelocal}\begin{enumerate}
\item\label{C+chamber 1}  The region
\[
C_+:=\{\upvartheta\in\Uptheta\mid \upvartheta_i>0\mbox{ for all }i>0\}
\]
of $\Uptheta$ is a chamber.
\item\label{C+chamber 2}  $\Uptheta$ has a finite number of chambers, and the walls are given by a finite collection of hyperplanes containing the origin.  The co-ordinate hyperplanes $\upvartheta_i=0$ are included in this collection.
\item\label{C+chamber 3} Considering iterated mutations at indecomposable summands, tracking the chamber $C_+$ from $\upnu_{j_1}\hdots\upnu_{j_t}\AB$ back to $\Uptheta$ gives all the chambers of $\Uptheta$.
\end{enumerate}
\end{prop}
\begin{proof}
This follows immediately from \cite[5.16, 5.23]{HomMMP}.
\end{proof}

\begin{remark}
Moduli tracking works in both directions, and this is important for our application.  First, \ref{C+chamber}\eqref{C+chamber 3} allows us to fix the input $\AB$, and track moduli from some iterated mutation  $\upnu_{j_1}\hdots\upnu_{j_t}\AB$ back to $\AB$.  This procedure computes the chamber structure for the fixed $\AB$, which gives the combinatorial data needed to state theorems on braiding later on, in \S\ref{2 curve section}.  Second, \ref{moduli main} also allows us to track moduli starting from $\AB$ to some iterated mutation  $\upnu_{j_1}\hdots\upnu_{j_t}\AB$.  This second direction is needed to prove the theorems, in particular to establish the braiding in \ref{2braidcompletelocal}.
\end{remark}

\subsection{Chamber Structures} We keep the notation and setting from above.  In this subsection we give an example of a chamber structure arising in $3$-fold flops.  Although not strictly needed for the proof of the main theorem, this illustrates some new phenomena and subtleties in the combinatorics.

With input the flopping contraction $\mathfrak{U}\to\Spec\mathfrak{R}$ of \ref{flopscompletelocal}, by Reid's general elephant principle \cite[1.1, 1.14]{Pagoda}, cutting by a generic hyperplane section $g\in\mathfrak{R}$ gives
\[
\begin{array}{c}
\begin{tikzpicture}[yscale=1.25]
\node (Xp) at (-1,0) {$\mathfrak{U}_g$}; 
\node (X) at (1,0) {$\mathfrak{U}$};
\node (Rp) at (-1,-1) {$\Spec (\mathfrak{R}/g)$}; 
\node (R) at (1,-1) {$\Spec \mathfrak{R}$};
\draw[->] (Xp) to node[above] {$\scriptstyle $} (X);
\draw[->] (Rp) to node[above] {$\scriptstyle $} (R);
\draw[->] (Xp) --  node[left] {$\scriptstyle \upphi$}  (Rp);
\draw[->] (X) --  node[right] {$\scriptstyle \clocCon$}  (R);
\end{tikzpicture}
\end{array}
\]
where $\mathfrak{R}/g$ is an ADE singularity and $\upphi$ is a partial crepant resolution.  By general theory, $\End_{\mathfrak{R}/g}(N/gN)$ is derived equivalent to $\mathfrak{U}_g$, and so the module $N_i$ cuts to $N_i/gN_i$, which is precisely one of the CM modules corresponding to a vertex in an ADE Dynkin diagram via the McKay correspondence. 

Following the notation from \cite{Katz}, we encode $\mathfrak{U}_g$ pictorially by describing which curves are blown down from the minimal resolution.  The diagrams 

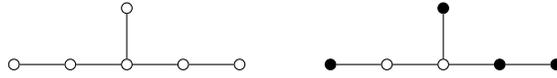
\[
\begin{array}{cccc}
\begin{array}{c}
\begin{tikzpicture}
\node (-1) at (-0.75,0) [DW] {};
\node (0) at (0,0) [DW] {};
\node (1) at (0.75,0) [DW] {};
\node (1b) at (0.75,0.75) [DW] {};
\node (2) at (1.5,0) [DW] {};
 \node (3) at (2.25,0) [DW] {};
\draw [-] (-1) -- (0);
\draw [-] (0) -- (1);
\draw [-] (1) -- (2);
\draw [-] (2) -- (3);
\draw [-] (1) -- (1b);
\end{tikzpicture}
\end{array}
&&
\begin{array}{c}
\begin{tikzpicture}
\node (-1) at (-0.75,0) [DB] {};
\node (0) at (0,0) [DW] {};
\node (1) at (0.75,0) [DW] {};
\node (1b) at (0.75,0.75) [DB] {};
\node (2) at (1.5,0) [DB] {};
 \node (3) at (2.25,0) [DB] {};
 \draw [-] (-1) -- (0);
\draw [-] (0) -- (1);
\draw [-] (1) -- (2);
\draw [-] (2) -- (3);
\draw [-] (1) -- (1b);
\end{tikzpicture}
\end{array}
\end{array}
\]
represent, respectively, the minimal resolution of the $E_6$ surface singularity, and the partial resolution obtained from it by contracting the curves corresponding to the black vertices.

\begin{example}\label{Katz example}
There is an example \cite[2.3]{Katz} of a $cD_4$ singularity $\mathfrak{R}$ with crepant resolution $X\to\Spec \mathfrak{R}$ with two curves above the origin, that cuts under generic hyperplane section to give the configuration 
\begin{eqnarray}
\begin{array}{c}
\begin{tikzpicture}
\node (0) at (0,0) [DW] {};
\node (1) at (0.75,0) [DW] {};
\node (1b) at (0.75,0.75) [DB] {};
\node (2) at (1.5,0) [DB] {};
\draw [-] (0) -- (1);
\draw [-] (1) -- (2);
\draw [-] (1) -- (1b);
\end{tikzpicture}
\end{array}\label{D4 config example}
\end{eqnarray}
\end{example}

By \ref{C+chamber}\eqref{C+chamber 3}, tracking moduli from iterated mutations back to $\AB$ computes the chamber structure of $\Uptheta_{\AB}$.  We illustrate this in the above example, referring the reader to \cite[7.2]{HomMMP} for more examples.

\begin{example}\label{knitting example}
As notation order the vertices

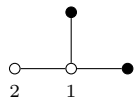
\[
\begin{array}{c}
\begin{tikzpicture}
\node (0) at (0,0) [DW] {};
\node (1) at (0.75,0) [DW] {};
\node (1b) at (0.75,0.75) [DB] {};
\node (2) at (1.5,0) [DB] {};
\draw [-] (0) -- (1);
\draw [-] (1) -- (2);
\draw [-] (1) -- (1b);
\node (0) at (0,-0.3) {$\scriptstyle 2$};
\node (1) at (0.75,-0.3) {$\scriptstyle 1$};
\end{tikzpicture}
\end{array}
\]
meaning that $N_1$ corresponds to the middle curve, and $N_2$ corresponds to the left-hand curve. The mutation exchange sequences are obtained by knitting (for details, see \cite[5.19, 5.24]{HomMMP}), so that in this example the $b$'s are determined by the data 
\begin{align}
U_1&\cong N_2\label{1}\\
U_2&\cong R^{\oplus 2}\oplus N_1^{\oplus 2}.\label{2}
\end{align}
First, we track the $C_+$ chamber from $\upnu_1\AB$ to $\AB$.  By \ref{moduli main}, 
\[
\begin{array}{c}
\upphi_1\\
\upphi_2
\end{array}
\stackrel{\scriptsize\mbox{\eqref{1}}}{\mapsto} 
\begin{array}{c}
-\upphi_1\\
\upphi_1+\upphi_2
\end{array}
\]
since from \eqref{1} the relevant $b$ is one, so we negate $\upphi_1$ and add $1\upphi_1$ to its neighbour.  Thus the $C_+$ chamber (namely $\upphi_1>0$, $\upphi_2>0$) from $\upnu_1\AB$ corresponds to  the region $\upvartheta_1<0$ and $\upvartheta_1+\upvartheta_2>0$ of $\Uptheta_{\AB}$, and thus this gives a chamber for $\Uptheta_{\AB}$.  

Next, we track the $C_+$ chamber from $\upnu_2\upnu_1\AB$ to $\upnu_1\AB$ to $\AB$.  By the same logic   
\[
\begin{array}{c}
\upphi_1\\
\upphi_2
\end{array}
\stackrel{\scriptsize\mbox{\eqref{2}}}{\mapsto} 
\begin{array}{c}
\upphi_1+2\upphi_2\\
-\upphi_2
\end{array}
\stackrel{\scriptsize\mbox{\eqref{1}}}{\mapsto} 
\begin{array}{c}
-(\upphi_1+2\upphi_2)\\
-\upphi_2+(\upphi_1+2\upphi_2)
\end{array}
=
\begin{array}{r}
-\upphi_1-2\upphi_2\\
\upphi_1+\upphi_2
\end{array}
\]
which gives the region $\upvartheta_1+2\upvartheta_2>0$
and $\upvartheta_1+\upvartheta_2<0$ of $\Uptheta_{\AB}$, and so this too is a chamber.  

Next, tracking $C_+$ from $\upnu_1\upnu_2\upnu_1\AB$ to $\upnu_2\upnu_1\AB$ to $\upnu_1\AB$ to $\AB$ gives 
\[
\begin{array}{c}
\upphi_1\\
\upphi_2
\end{array}
\stackrel{\scriptsize\mbox{\eqref{1}}}{\mapsto} 
\begin{array}{c}
-\upphi_1\\
\upphi_1+\upphi_2
\end{array}
\stackrel{\scriptsize\mbox{\eqref{2}}}{\mapsto} 
\begin{array}{c}
\upphi_1+2\upphi_2\\
-\upphi_1-\upphi_2
\end{array}
\stackrel{\scriptsize\mbox{\eqref{1}}}{\mapsto} 
\begin{array}{c}
-\upphi_1-2\upphi_2\\
\upphi_2
\end{array}
\]
which gives the region $\upvartheta_1+2\upvartheta_2<0$ and $\upvartheta_2>0$.  

Continuing in this fashion, $\Uptheta_{\AB}$ has the following eight chambers.
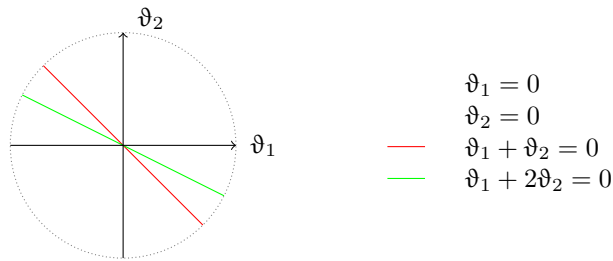
\[
\begin{array}{ccc}
\begin{array}{c}
\begin{tikzpicture}[scale=0.75]
\coordinate (A1) at (135:2cm);
\coordinate (A2) at (-45:2cm);
\coordinate (B1) at (-33.690:2cm);
\coordinate (B2) at (146.31:2cm);
\coordinate (C1) at (153.435:2cm);
\coordinate (C2) at (-26.565:2cm);
\coordinate (D1) at (161.565:2cm);
\coordinate (D2) at (-18.435:2cm);
\draw[red] (A1) -- (A2);
\draw[green] (C1) -- (C2);
\draw[->] (-2,0)--(2,0);
\node at (2.5,0) {$\upvartheta_1$};
\draw[->] (0,-2)--(0,2);
\node at (0.5,2.25) {$\upvartheta_2$};
\draw[densely dotted,gray] (0,0) circle (2cm);
\end{tikzpicture}
\end{array}
&&
\begin{array}{cl}
&\upvartheta_1=0\\
&\upvartheta_2=0\\
\begin{array}{c}\begin{tikzpicture}\node at (0,0){}; \draw[red] (0,0)--(0.5,0);\end{tikzpicture}\end{array}&\upvartheta_1+\upvartheta_2=0\\
\begin{array}{c}\begin{tikzpicture}\node at (0,0){}; \draw[green] (0,0)--(0.5,0);\end{tikzpicture}\end{array}&\upvartheta_1+2\upvartheta_2=0
\end{array}
\end{array}
\]
\end{example}

\subsection{Hyperplane Arrangements for 3-Fold Flops}\label{hyper local global section}
In this subsection, we explain how to obtain a hyperplane arrangement for any flopping contraction $f \colon X\to X_{\con}$ in the global quasi-projective setup of \ref{flopsglobal}, and we describe its basic properties.  This fixes notation for the remainder of the paper.

\begin{notation}\label{Hyperplane Hp}
For each $p\in\noniso f$, by considering the formal fibre above $p$, by \ref{C+chamber}\eqref{C+chamber 2} the space of stability conditions on this formal fibre has a chamber structure in which the codimension-one walls are all hyperplanes through the origin.  We denote by $\cH^p$ the set of codimension-one walls.
\end{notation}

Recall that for a hyperplane arrangement $\cA$ in $\mathbb{R}^{a}$, and a hyperplane arrangement $\cB$ in $\mathbb{R}^{b}$, their product is defined to be
\[
\cA\textstyle\prod \cB:=\{ H\oplus \mathbb{R}^b\mid H\in\cA\}\cup\{ \mathbb{R}^a\oplus H\mid H\in \cB\},
\] 
which is a hyperplane arrangement in $\mathbb{R}^{a+b}$.  
\begin{notation}
Under the global quasi-projective flops setup of \ref{flopsglobal}, we define
\[
\cH:=\prod_p\cH^{p},
\]
where $p$ ranges over $\noniso f$. This is a hyperplane arrangement in $\mathbb{R}^{n}$, where $n$ is the total number of irreducible curves contracted.
\end{notation}
For our purposes later, we require more precise information regarding the structure of $\cH$ and $\cH^p$.  Recall that a real hyperplane arrangement is called \emph{simplicial} if the intersection of all the hyperplanes is $\{ 0\}$, and furthermore every chamber is an open simplicial cone. The following result, which is an immediate consequence of the Homological MMP, will be used heavily.

\begin{lemma}\label{issimplicial}
With notation as above, $\cH^p$ is a simplicial hyperplane arrangement for all $p\in\noniso f$\!, and so consequently $\cH$ is a simplicial hyperplane arrangement.
\end{lemma}
\begin{proof}
By \ref{C+chamber}\eqref{C+chamber 2}, the codimension-one walls of $\Uptheta_\AB$ are given by a finite collection of hyperplanes, all of which contain the origin, and furthermore the co-ordinate hyperplanes $x_i=0$ are included in this collection.  It follows that the intersection of all the hyperplanes is the origin.  Further, since $C_+$ is clearly an open simplicial cone, and by \ref{C+chamber}\eqref{C+chamber 3} every chamber of $\Uptheta_\AB$ is given by tracking the chamber $C_+$ under moduli tracking, it follows that all chambers are open simplicial cones.  This proves $\cH^p$ is a simplicial hyperplane arrangement for all $p\in\noniso f$, and it is well known that the product of simplicial arrangements is simplicial.
\end{proof}

As one further piece of notation, for a real hyperplane arrangement $\cA$ in $\mathbb{R}^a$, as is standard we write
\[
\cA_{\mathbb{C}}:=\bigcup_{H\in\cA} H_{\mathbb{C}}
\]
where $H_{\mathbb{C}}$ denotes the complexification of $H$.  It is well known that the fundamental group of the complexified complement behaves well under products, so
\[
\fundgp(\mathbb{C}^{n}\backslash {\cH}_\mathbb{C})=
\bigoplus_p
\fundgp(\mathbb{C}^{n_p}\backslash \cH^{p}_\mathbb{C}),
\]
where $p$ ranges over $\noniso f$.

\subsection{Braiding: Two-Curve Case}\label{2 curve section}
The aim of this subsection is to prove the following.

\begin{thm}\label{braid 2 curve global}
With the global quasi-projective flops setup $f\colon X\to X_{\con}$ of \ref{flopsglobal}, suppose that $f$ contracts precisely two independently floppable irreducible curves.  Then
\[
\underbrace{\flop_1\circ\flop_2\circ\flop_1\circ\cdots}_{d}\functcong
\flop
\functcong
\underbrace{\flop_2\circ\flop_1\circ\flop_2\circ\cdots}_{d}
\]
where $d$ is the number of hyperplanes in $\cH$.  Furthermore:
\begin{enumerate}
\item\label{braid 2 curve global 1} If the curves intersect, then $d\geq 3$.
\item\label{braid 2 curve global 2} If the curves are disjoint, then $d=2$.
\end{enumerate}
\end{thm}

We split the proof: the statement and proof of \ref{braid 2 curve global}\eqref{braid 2 curve global 1} is contained in \ref{braid 2 curve global intersect}, and similarly \ref{braid 2 curve global}\eqref{braid 2 curve global 2} is contained in \ref{braid 2 curve global disjoint}.

We first prove \ref{braid 2 curve global}\eqref{braid 2 curve global 1}.  To do this, we give a complete local version of the result in \ref{2braidcompletelocal}.  We then establish a Zariski local version in \ref{braid 2 curve local}, before finally giving the result globally in \ref{braid 2 curve global intersect}.  Before beginning the proof, which is notationally complicated, we first illustrate the strategy in an example.

\begin{example}\label{braiding knitting example}
We continue the complete local example \ref{knitting example}.  Label the minimal models arising from the chambers by
\[
\begin{array}{c}
\begin{tikzpicture}[scale=0.75]
\coordinate (A1) at (135:2cm);
\coordinate (A2) at (-45:2cm);
\coordinate (B1) at (-33.690:2cm);
\coordinate (B2) at (146.31:2cm);
\coordinate (C1) at (153.435:2cm);
\coordinate (C2) at (-26.565:2cm);
\coordinate (D1) at (161.565:2cm);
\coordinate (D2) at (-18.435:2cm);
\draw[red] (A1) -- (A2);
\draw[green] (C1) -- (C2);
\draw[-] (-2,0)--(2,0);
\draw[-] (0,-2)--(0,2);
\node at (45:2cm)  {$\scriptstyle \mathfrak{U}$};
\node at (112.5:2cm) {$\scriptstyle \mathfrak{U}_1$};
\node at (145:2cm) {$\scriptstyle \mathfrak{U}_{21}$};
\node at (164.25:2cm) {$\scriptstyle \mathfrak{U}_{121}$};
\node  at (225:2cm) {$\scriptstyle \mathfrak{U}_{1212}$};
\node at (-67.5:2cm) {$\scriptstyle \mathfrak{U}_{212}$};
\node at (-35:2cm) {$\scriptstyle \mathfrak{U}_{12}$};
\node at (-15.75:2cm) {$\scriptstyle \mathfrak{U}_{2}$};
\end{tikzpicture}
\end{array}
\] 
By \cite[4.9]{HomMMP}, the above chamber structure implies that 
\[
\upnu_2\upnu_1\upnu_2\upnu_1 N\cong\upnu_1\upnu_2\upnu_1\upnu_2 N
\]
since the chamber $C_+$ for both $\End_{\mathfrak{R}}(\upnu_2\upnu_1\upnu_2\upnu_1 N)$ and $\End_{\mathfrak{R}}(\upnu_1\upnu_2\upnu_1\upnu_2 N)$ gives, under moduli tracking, the same chamber on $\Uptheta_{\AB}$.  Hence there is a diagram of mutation functors
\begin{eqnarray}
\begin{array}{cc}
\begin{tikzpicture}
\node (C1) at (9,0)  {$\Db(\upnu_1\upnu_2\upnu_1\upnu_2\AB)$};
\node (C2) at (7,1)  {$\Db(\upnu_2\upnu_1\upnu_2\AB)$};
\node (C3) at (4,1)  {$\Db(\upnu_1\upnu_2\AB)$};
\node (C4) at (1,1)  {$\Db(\upnu_2\AB)$};
\node (C5) at (-1,0)  {$\Db(\AB)$};
\node (C6) at (1,-1)  {$\Db(\upnu_1\AB)$};
\node (C7) at (4,-1)  {$\Db(\upnu_2\upnu_1\AB)$};
\node (C8) at (7,-1)  {$\Db(\upnu_1\upnu_2\upnu_1\AB)$};
\draw[->] (C5) -- node[above left] {$\scriptstyle \Upphi_2$} (C4);
\draw[->] (C4) -- node[above] {$\scriptstyle \Upphi_1$} (C3);
\draw[->] (C3) -- node[above] {$\scriptstyle \Upphi_2$} (C2);
\draw[->] (C2) -- node[above right] {$\scriptstyle \Upphi_1$} (C1);
\draw[->] (C5) -- node[below left] {$\scriptstyle \Upphi_1$} (C6);
\draw[->] (C6) -- node[below] {$\scriptstyle \Upphi_2$} (C7);
\draw[->] (C7) -- node[below] {$\scriptstyle \Upphi_1$} (C8);
\draw[->] (C8) -- node[below right] {$\scriptstyle \Upphi_2$} (C1);
\end{tikzpicture}
\end{array}\label{mut functors}
\end{eqnarray}
Further, by \cite[4.2]{HomMMP} the inverse of the flop functor is functorially isomorphic to mutation, so \eqref{mut functors} is functorially isomorphic to the diagram of functors
\begin{eqnarray}
\begin{array}{cc}
\begin{tikzpicture}
\node (C1) at (9,0)  {$\Db(\mathfrak{U}_{1212})$};
\node (C2) at (7,1)  {$\Db(\mathfrak{U}_{212})$};
\node (C3) at (4,1)  {$\Db(\mathfrak{U}_{12})$};
\node (C4) at (1,1)  {$\Db(\mathfrak{U}_2)$};
\node (C5) at (-1,0)  {$\Db(\mathfrak{U})$};
\node (C6) at (1,-1)  {$\Db(\mathfrak{U}_1)$};
\node (C7) at (4,-1)  {$\Db(\mathfrak{U}_{21})$};
\node (C8) at (7,-1)  {$\Db(\mathfrak{U}_{121})$};
\draw[->] (C5) -- node[above left] {$\scriptstyle \flop_2^{-1}$} (C4);
\draw[->] (C4) -- node[above] {$\scriptstyle \flop_1^{-1}$} (C3);
\draw[->] (C3) -- node[above] {$\scriptstyle \flop_2^{-1}$} (C2);
\draw[->] (C2) -- node[above right] {$\scriptstyle \flop_1^{-1}$} (C1);
\draw[->] (C5) -- node[below left] {$\scriptstyle \flop_1^{-1}$} (C6);
\draw[->] (C6) -- node[below] {$\scriptstyle \flop_2^{-1}$} (C7);
\draw[->] (C7) -- node[below] {$\scriptstyle \flop_1^{-1}$} (C8);
\draw[->] (C8) -- node[below right] {$\scriptstyle \flop_2^{-1}$} (C1);
\end{tikzpicture}
\end{array}\label{flop functors}
\end{eqnarray}
Now choose a skyscraper $\cO_x$ in $\Db(\mathfrak{U})$.  By \cite[\S5.2]{Joe}, this corresponds to some $\upvartheta$-stable module $M$ in $\Db(\AB)$ for $\upvartheta\in C_+$.  Remarkably, under the mutation functors in \eqref{mut functors}, this module $M$ is always sent to a module. This follows by using \ref{moduli main}\eqref{moduli main 1} repeatedly. Indeed the new module is stable for some stability parameter, which may be calculated using the formula given in \ref{moduli main}. Tracking this data we see that under mutation the module $M$ gets sent to modules stable for parameters as follows.
\begin{eqnarray}
\begin{array}{cc}
\begin{tikzpicture}
\node (C1) at (9,0)  {$\begin{array}{c}
-\upvartheta_1\\
-\upvartheta_2
\end{array}$};
\node (C1n) at (8.7,0.3) {};
\node (C1s) at (8.7,-0.3) {};
\node (C2) at (7,1)  {$\begin{array}{c}
\upvartheta_1\\
-\upvartheta_1-\upvartheta_2
\end{array}$};
\node (C3) at (4,1)  {$\begin{array}{c}
-\upvartheta_1-2\upvartheta_2\\
\upvartheta_1+\upvartheta_2
\end{array}$};
\node (C4) at (1,1)  {$\begin{array}{c}
\upvartheta_1+2\upvartheta_2\\
-\upvartheta_2
\end{array}$};
\node (C5) at (-1,0)  {$\begin{array}{c}
\upvartheta_1\\
\upvartheta_2
\end{array}$};
\node (C5n) at (-0.7,0.3) {};
\node (C5s) at (-0.7,-0.3) {};
\node (C6) at (1,-1)  {$\begin{array}{c}
-\upvartheta_1\\
\upvartheta_1+\upvartheta_2
\end{array}$};
\node (C7) at (4,-1)  {$\begin{array}{c}
\upvartheta_1+2\upvartheta_2\\
-\upvartheta_1-\upvartheta_2
\end{array}$};
\node (C8) at (7,-1)  {$\begin{array}{c}
-\upvartheta_1-2\upvartheta_2\\
\upvartheta_2
\end{array}$};
\draw[->] (C5n) -- node[above left] {$\scriptstyle \Upphi_2$} (C4);
\draw[->] (C4) -- node[above] {$\scriptstyle \Upphi_1$} (C3);
\draw[->] (C3) -- node[above] {$\scriptstyle \Upphi_2$} (C2);
\draw[->] (C2) -- node[above right] {$\scriptstyle \Upphi_1$} (C1n);
\draw[<-] (C5s) -- node[below left] {$\scriptstyle \Upphi_1^{-1}$} (C6);
\draw[<-] (C6) -- node[below] {$\scriptstyle \Upphi_2^{-1}$} (C7);
\draw[<-] (C7) -- node[below] {$\scriptstyle \Upphi_1^{-1}$} (C8);
\draw[<-] (C8) -- node[below right] {$\scriptstyle \Upphi_2^{-1}$} (C1s);
\end{tikzpicture}
\end{array}\label{mut functors stab}
\end{eqnarray}
It follows immediately that the composition of mutations
\[
\Upphi_1^{-1}\circ\Upphi_2^{-1}\circ\Upphi_1^{-1}\circ\Upphi_2^{-1}\circ\Upphi_1\circ\Upphi_2\circ\Upphi_1\circ\Upphi_2
\]
sends $M$, which is $\upvartheta$-stable for $\upvartheta=(\upvartheta_1,\upvartheta_2)\in C_+$, to a module which is also $\upvartheta$-stable.  Since \eqref{mut functors} is functorially isomorphic to \eqref{flop functors}, it follows that 
\[
\Uppsi:=\flop_1\circ\flop_2\circ\flop_1\circ\flop_2\circ\flop_1^{-1}\circ\flop_2^{-1}\circ\flop_1^{-1}\circ\flop_2^{-1}
\]
sends the skyscraper $\cO_x$ to some object in $\Db(\coh \mathfrak{U})$ corresponding to a $\upvartheta$-stable module.  But again by \cite[\S5.2]{Joe} these are precisely the skyscrapers.  Hence we obtain that $\Uppsi$ is a Fourier--Mukai equivalence that takes skyscrapers to skyscrapers, fixes $\cO_\mathfrak{U}$ and commutes with the pushdown $\RDerived f_*$.  It follows that $\Uppsi\functcong\Id$ and so
\[
\flop_1\circ\flop_2\circ\flop_1\circ\flop_2\functcong\flop_2\circ\flop_1\circ\flop_2\circ\flop_1.
\]
\end{example}

The following proposition is a simple extension of the above example. Recall that $\Upphi_{\{1,2\}}$ denotes the mutation functor for the summand $N_1 \oplus N_2$.

\begin{prop}\label{2braidcompletelocal}
Under the complete local flops setup $\mathfrak{U} \to \Spec\mathfrak{R}$ of \ref{flopscompletelocal}, suppose that precisely two irreducible curves are contracted. Then
\[
\underbrace{\Upphi_1\circ\Upphi_2\circ\Upphi_1\circ\cdots}_{d}\functcong
\Upphi_{\{1,2\}}\functcong
\underbrace{\Upphi_2\circ\Upphi_1\circ\Upphi_2\circ\cdots}_{d}
\]
and 
\[
\underbrace{\flop_1\circ\flop_2\circ\flop_1\circ\cdots}_{d}\functcong
\flop_{\{1,2\}}\functcong
\underbrace{\flop_2\circ\flop_1\circ\flop_2\circ\cdots}_{d}
\]
where $d$ is the number of hyperplanes in $\cH^p$.  Further, $d\geq 3$.
\end{prop}
\begin{proof}
Consider $\AB=\End_{\mathfrak{R}}(N)$, which is derived equivalent to $\mathfrak{U}$.  By \ref{C+chamber}, $\cH^p$ is a hyperplane arrangement in $\mathbb{R}^2$, with $C_+$ being a chamber.  This implies that the hyperplane arrangement is
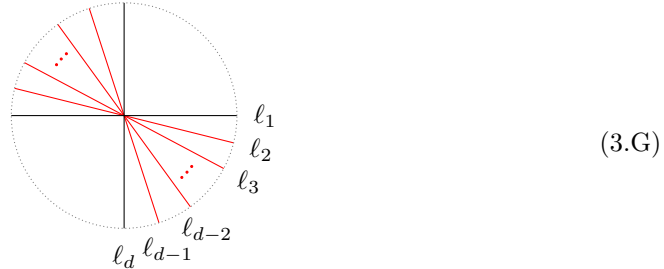
\begin{equation}\label{chamber structure}
\def\lowerpinch{14}
\def\upperpinch{18}
\def\dotpos{138}
\begin{array}{ccc}
\begin{array}{c}
\begin{tikzpicture}[scale=0.75]
\coordinate (A1) at (90+\upperpinch:2cm);
\coordinate (A2) at (-90+\upperpinch:2cm);
\coordinate (B1) at (90+2*\upperpinch:2cm);
\coordinate (B2) at (-90+2*\upperpinch:2cm);
\coordinate (C1) at (180-2*\lowerpinch:2cm);
\coordinate (C2) at (-2*\lowerpinch:2cm);
\coordinate (D1) at (180-\lowerpinch:2cm);
\coordinate (D2) at (-\lowerpinch:2cm);
\draw[red] (A1) -- (A2);
\draw[red] (B1) -- (B2);
\draw[red] (C1) -- (C2);
\draw[red] (D1) -- (D2);
\draw[-] (-2,0)--(2,0);
\filldraw[red] (\dotpos-5:1.5cm) circle (0.5pt);
\filldraw[red] (\dotpos:1.5cm) circle (0.5pt);
\filldraw[red] (\dotpos+5:1.5cm) circle (0.5pt);
\filldraw[red] (\dotpos-180-5:1.5cm) circle (0.5pt);
\filldraw[red] (\dotpos-180:1.5cm) circle (0.5pt);
\filldraw[red] (\dotpos-180+5:1.5cm) circle (0.5pt);
\node at (0:2.5cm) {$\ell_1$};
\node at (0-\lowerpinch:2.5cm) {$\ell_2$};
\node at (0-2*\lowerpinch:2.5cm) {$\ell_3$};
\node at (-90+2*\upperpinch:2.5cm) {$\ell_{d-2}$};
\node at (-90+\upperpinch:2.5cm) {$\ell_{d-1}$};
\node at (-90:2.5cm) {$\ell_{d}$};
\draw[-] (0,-2)--(0,2);
\draw[densely dotted,gray] (0,0) circle (2cm);
\end{tikzpicture}
\end{array}
\end{array}
\end{equation}
for some lines given by $\ell_1=0$, $\ell_2=0$, \dots, $\ell_d=0$, where
\[
\ell_1=\upvartheta_2
\quad\text{and}\quad
\ell_d=\upvartheta_1.
\] 
The fact that $d\geq 3$ follows immediately by knitting on the AR quivers of Kleinian singularities, as in \cite[5.23, 5.19]{HomMMP}.

We first claim that the above chamber structure \eqref{chamber structure} implies that 
\begin{eqnarray}
\underbrace{\hdots\upnu_2\upnu_1\upnu_2}_{d}N \cong
\upnu_{\{1,2\}}N \cong
\underbrace{\hdots\upnu_1\upnu_2\upnu_1}_{d}N. \label{N both ways}
\end{eqnarray}
For notation, as in \ref{braiding knitting example}, we let $\mathfrak{U}_{\hdots kji}$ denote the scheme obtained from $\mathfrak{U}$ by first flopping curve $i$, then curve $j$, then curve $k$ and so on (in that order).  By \cite[4.2]{HomMMP}, we have
\[
\upnu_1N\cong H^0(\cV_{\mathfrak{U}_1})
\quad\text{and}\quad
\upnu_2N\cong H^0(\cV_{\mathfrak{U}_2}).
\]
Iterating, again by \cite[4.2]{HomMMP} the left- and right-hand terms of \eqref{N both ways} are $H^0(\cV_{\mathfrak{U}_{\hdots 212}})$ and $H^0(\cV_{\mathfrak{U}_{\hdots 121}})$, respectively.  Further, by \cite[5.2.5]{Joe}, the scheme $\mathfrak{U}_{\hdots 212}$ can be obtained as the moduli in the chamber $C_+$ for $\hdots\upnu_{2}\upnu_1\upnu_2\AB$, and $\mathfrak{U}_{\hdots 121}$ can be obtained as the moduli in the chamber $C_+$ for $\hdots\upnu_{1}\upnu_2\upnu_1\AB$.   Tracking these chambers back to $\Uptheta_{\AB}$, which we can do by \ref{C+chamber}\eqref{C+chamber 3} (see also the proof of \cite[5.22]{HomMMP}), both give the region
\[
C_-=\{ \upvartheta\mid \upvartheta_i<0\mbox{ for }i=1,2 \}\subset\Uptheta_{\AB},
\]
and so $\mathfrak{U}_{\hdots 212}\cong\mathfrak{U}_{\hdots 121}$ as schemes over $\Spec \mathfrak{R}$.  In fact, again by moduli tracking, both $\mathfrak{U}_{\hdots 212}$ and $\mathfrak{U}_{\hdots 121}$  are isomorphic to the scheme obtained from $\mathfrak{U}$ by flopping $C_1\bigcup C_2$, which we denote by $\mathfrak{U}_{\{1,2\}}$. Using \cite[4.2]{HomMMP} once again, the middle term of \eqref{N both ways} is $H^0(\cV_{\mathfrak{U}_{\{1,2\}}}\!)$. Taking global sections of the progenerator of perverse sheaves on $\mathfrak{U}_{\{1,2\}}$ gives  \eqref{N both ways}, as claimed.

Because of \eqref{N both ways}, there exists a diagram of mutation functors
\begin{eqnarray}
\begin{array}{cc}
\begin{tikzpicture}
\node (C1) at (7.75,0)  {$\Db(\hdots\upnu_2\upnu_1\upnu_2\AB)$};
\node (C2) at (4.5,1)  {$\Db(\hdots\upnu_1\upnu_2\AB)$};
\node (C3) at (3,1)  {$\hdots$};
\node (C4) at (1,1)  {$\Db(\upnu_2\AB)$};
\node (C5) at (-1,0)  {$\Db(\AB)$};
\node (C6) at (1,-1)  {$\Db(\upnu_1\AB)$};
\node (C7) at (3,-1)  {$\hdots$};
\node (C8) at (4.5,-1)  {$\Db(\hdots\upnu_2\upnu_1\AB)$};
\draw[->] (C5.north east) -- node[above left] {$\scriptstyle \Upphi_2$} (C4.south west);
\draw[->] (C4) -- node[above] {$\scriptstyle \Upphi_1$} (C3);
\draw[->] (C2.south east) -- node[above right] {$\scriptstyle $} (C1.north west);
\draw[->] (C5.south east) -- node[below left] {$\scriptstyle \Upphi_1$} (C6.north west);
\draw[->] (C6) -- node[below] {$\scriptstyle \Upphi_2$} (C7);
\draw[->] (C8.north east) -- node[below right] {$\scriptstyle $} (C1.south west);
\draw[->] (C5) -- node[above] {$\scriptstyle \Upphi_{\{1,2\}}$} (C1);
\end{tikzpicture}
\end{array}\label{mut general braid}
\end{eqnarray}
where the functors on the right-hand side depend on whether $d$ is odd or even; respectively they are
\[
\begin{array}{c@{\:}c@{\:}c}
\begin{array}{c}
\begin{tikzpicture}
\node (C1) at (7.75,0)  {$\Db(\hdots\upnu_2\upnu_1\upnu_2\AB)$};
\node (C2) at (4.5,1)  {$\Db(\hdots\upnu_1\upnu_2\AB)$};
\node (C8) at (4.5,-1)  {$\Db(\hdots\upnu_2\upnu_1\AB)$};
\draw[->] (C2.south east) -- node[above right] {$\scriptstyle \Upphi_2$} (C1.north west);
\draw[->] (C8.north east) -- node[below right] {$\scriptstyle \Upphi_1$} (C1.south west);
\end{tikzpicture}
\end{array}
&\mbox{and}&
\begin{array}{c}
\begin{tikzpicture}
\node (C1) at (7.75,0)  {$\Db(\hdots\upnu_2\upnu_1\upnu_2\AB).$};
\node (C2) at (4.5,1)  {$\Db(\hdots\upnu_1\upnu_2\AB)$};
\node (C8) at (4.5,-1)  {$\Db(\hdots\upnu_2\upnu_1\AB)$};
\draw[->] (C2.south east) -- node[above right] {$\scriptstyle \Upphi_1$} (C1.north west);
\draw[->] (C8.north east) -- node[below right] {$\scriptstyle \Upphi_2$} (C1.south west);
\end{tikzpicture}
\end{array}
\end{array}
\]
We next chase moduli around \eqref{mut general braid}, repeatedly applying \ref{moduli main} using the characterisation of the chamber structure \eqref{chamber structure}. Consider a $\upvartheta$-stable $\AB$-module $M$, for $\upvartheta\in C_+$. Tracking $M$ around \eqref{mut general braid}, we find that it is sent to a module which is stable for the following parameters: when $d$ is even
\[
\begin{tikzpicture}[scale=1.1]
\node (C1) at (9.25,0)  {$\begin{array}{c}
-\ell_{n}\\
-\ell_1
\end{array}$
\!=\!
$\begin{array}{c}
-\upvartheta_1\\
-\upvartheta_2
\end{array}$};
\node (C2b) at (7,1)  {$\begin{array}{c}
\ell_{n}\\
-\ell_{n-1}
\end{array}$};
\node (C2c) at (5,1)  {$\begin{array}{c}
-\ell_{n-2}\\
\ell_{n-1}
\end{array}$};
\node (C2a) at (4,1)  {$\hdots$};
\node (C3) at (3,1)  {$\begin{array}{c}
-\ell_2\\
\ell_3
\end{array}$};
\node (C4) at (1,1)  {$\begin{array}{c}
\ell_2\\
-\ell_1
\end{array}$};
\node (C5) at (-1,0)  {$\begin{array}{c}
\upvartheta_1\\
\upvartheta_2
\end{array}$
=
$\begin{array}{c}
\ell_n\\
\ell_1
\end{array}$};
\node (C6) at (1,-1)  {$\begin{array}{c}
-\ell_n\\
\ell_{n-1}
\end{array}$};
\node (C7) at (3,-1)  {$\begin{array}{c}
\ell_{n-2}\\
-\ell_{n-1}
\end{array}$};
\node (C7a) at (4,-1)  {$\hdots$};
\node (C8a) at (5,-1)  {$\begin{array}{c}
\ell_2\\
-\ell_3
\end{array}$};
\node (C8b) at (7,-1)  {$\begin{array}{c}
-\ell_2\\
\ell_1
\end{array}$};
\draw[->] (C5) -- node[above left] {$\scriptstyle \Upphi_2$} (C4);
\draw[->] (C4) -- node[above] {$\scriptstyle \Upphi_1$} (C3);
\draw[->] (C2c) -- node[above] {$\scriptstyle \Upphi_2$} (C2b);
\draw[->] (C2b) -- node[above right] {$\scriptstyle \Upphi_1$} (C1);
\draw[->] (C5) -- node[below left] {$\scriptstyle \Upphi_1$} (C6);
\draw[->] (C6) -- node[below] {$\scriptstyle \Upphi_2$} (C7);
\draw[->] (C8a) -- node[below] {$\scriptstyle \Upphi_1$} (C8b);
\draw[->] (C8b) -- node[below right] {$\scriptstyle \Upphi_2$} (C1);
\end{tikzpicture}
\]
and when $d$ is odd
\[
\begin{tikzpicture}[scale=1.1]
\node (C1) at (7.25,0)  {$\begin{array}{c}
-\ell_{1}\\
-\ell_n
\end{array}
\!=\!
\begin{array}{c}
-\upvartheta_2\\
-\upvartheta_1
\end{array}$};
\node (C2) at (5,1)  {$\begin{array}{c}
-\ell_{n-1}\\
\ell_{n}
\end{array}$};
\node (C2a) at (4,1)  {$\hdots$};
\node (C3) at (3,1)  {$\begin{array}{c}
-\ell_2\\
\ell_3
\end{array}$};
\node (C4) at (1,1)  {$\begin{array}{c}
\ell_2\\
-\ell_1
\end{array}$};
\node (C5) at (-1,0)  {$\begin{array}{c}
\upvartheta_1\\
\upvartheta_2
\end{array}$
=
$\begin{array}{c}
\ell_n\\
\ell_1
\end{array}$};
\node (C6) at (1,-1)  {$\begin{array}{c}
-\ell_n\\
\ell_{n-1}
\end{array}$};
\node (C7) at (3,-1)  {$\begin{array}{c}
\ell_{n-2}\\
-\ell_{n-1}
\end{array}$};
\node (C7a) at (4,-1)  {$\hdots$};
\node (C8) at (5,-1)  {$\begin{array}{c}
\ell_1\\
-\ell_2
\end{array}$};
\draw[->] (C5) -- node[above left] {$\scriptstyle \Upphi_2$} (C4);
\draw[->] (C4) -- node[above] {$\scriptstyle \Upphi_1$} (C3);
\draw[->] (C2) -- node[above right] {$\scriptstyle \Upphi_2$} (C1);
\draw[->] (C5) -- node[below left] {$\scriptstyle \Upphi_1$} (C6);
\draw[->] (C6) -- node[below] {$\scriptstyle \Upphi_2$} (C7);
\draw[->] (C8) -- node[below right] {$\scriptstyle \Upphi_1$} (C1);
\end{tikzpicture}
\]
In either case,
\begin{eqnarray}
(\Upphi_1^{-1}\circ\Upphi_2^{-1}\circ\Upphi_1^{-1}\circ \cdots)\circ(\cdots\circ\Upphi_2\circ\Upphi_1\circ\Upphi_2)\label{all way round}
\end{eqnarray}
sends a $\upvartheta$-stable module to a $\upvartheta$-stable module. Similarly, since $\Upphi_{\{1,2\}}$ negates both parameters, 
\begin{eqnarray}
\Upphi_{\{1,2\}}^{-1}\circ(\cdots\circ\Upphi_2\circ\Upphi_1\circ\Upphi_2)\label{half way round then straight back}
\end{eqnarray}
sends a $\upvartheta$-stable module to a module stable for some parameter in $C_+$.

Now by \cite[4.2]{HomMMP} mutation is functorially isomorphic to the inverse of the flop functor, and by \cite[\S5.2]{Joe} under the derived equivalence skyscrapers correspond precisely to $\upvartheta$-stable modules, for $\upvartheta\in C_+$. Hence \eqref{all way round} and \eqref{half way round then straight back} are functorially isomorphic to the corresponding chain of inverse flop functors, and it follows that
\begin{eqnarray*}
(\flop_1\circ\flop_2\circ\flop_1\circ\cdots)\circ(\cdots\circ\flop_2^{-1}\circ\flop_1^{-1}\circ\flop_2^{-1})
\\
\text{and}\quad\flop_{\{1,2\}}\circ(\cdots\circ\flop_2^{-1}\circ\flop_1^{-1}\circ\flop_2^{-1}) 
\end{eqnarray*}
are equivalences that take skyscrapers to skyscrapers.  Since they also fix the structure sheaf $\cO_{\mathfrak{U}}$, and commute with the pushdown $\RDerived f_*$ as in \cite[7.16(1)]{DW1}, it follows that both are naturally isomorphic to the identity. Finally, we deduce that \eqref{all way round} and \eqref{half way round then straight back} are also naturally isomorphic to the identity, and the result follows.
\end{proof}

We next lift the above into the algebraic setting.  To do this,  we use the Zariski local tilting bundle $\cV$ in \S\ref{localnotation}.
\begin{lemma}\label{alg comm}
Under the Zariski local flops setup $U\to\Spec R$ of \ref{flopslocal}, suppose that there are precisely two irreducible curves contracted, and that they intersect. Denote by $U^+$ the flop of $U$ at one of the curves, and consider the tilting bundles $\cV$ and $\cV^+$ from \S\ref{localnotation}.  Set $\Lambda:=\End_U(\cV)$ and $\Lambda^+:=\End_{U^+}(\cV^+)$, then
\begin{enumerate}
\item\label{alg comm 1} $Z:=\Hom_R(H^0(\cV),H^0(\cV^+))$ is a tilting $\Lambda$-module with $\End_\Lambda(Z)\cong\Lambda^+$.
\item\label{alg comm 2}  The following is a commutative diagram of equivalences.
\begin{eqnarray}
\begin{array}{c}
\begin{tikzpicture}
\node (A1) at (0,0) {$\Db(\coh U)$};
\node (A2) at (0,-1.5) {$\Db(\coh U^+)$};
\node (B1) at (5,0) {$\Db(\mod\Lambda)$};
\node (B2) at (5,-1.5) {$\Db(\mod \Lambda^+)$};
\draw[->] (A1) -- node[above] {$\scriptstyle \RHom_U(\cV,-)$} (B1);
\draw[->] (A2) -- node[above] {$\scriptstyle \RHom_{U^+}(\cV^+\!,-)$} (B2);
\draw[->] (A1) -- node[left] {$\scriptstyle \flop^{-1}$} (A2);
\draw[->] (B1) -- node[right] {$\scriptstyle \RHom_\Lambda(Z,-)$} (B2);
\end{tikzpicture}
\end{array}\label{Zar local comm}
\end{eqnarray}
\end{enumerate}
\end{lemma}
\begin{proof}
(1)  For the statement on endomorphism rings, by repeated use of \eqref{up=down VdB} we see that
\[
\End_\Lambda(Z)\cong\End_{\End_R(H^0(\cV))}^{\phantom{\End}}\left(\Hom_R(H^0(\cV),H^0(\cV^+))\right)\cong \End_R(H^0(\cV^+))\cong\Lambda^+,
\]
where the second isomorphism is reflexive equivalence.
Now the property of being a tilting module can be checked complete locally, and certainly $\widehat{Z}_\n$ is a tilting $\widehat{\Lambda}_\n$-module for all $\n\in\Max R$ with $\n\neq \m$.  Further, for the point $\m$, by \eqref{a_i decomp} we know that $\add (H^0(\cV)\otimes_R\mathfrak{R})=\add N$ and similarly $\add(H^0(\cV^+)\otimes_R\mathfrak{R}) =\add N^+$.  But exactly as in \cite[5.8]{DW1} (or in \ref{mutation Morita diagram} below), there is a commutative diagram
\begin{eqnarray}
\begin{array}{c}
\begin{tikzpicture}
\node (A1) at (0,0) {$\Db(\mod\AB)$};
\node (A2) at (0,-1.5) {$\Db(\mod \ABplus)$};
\node (B1) at (4,0) {$\Db(\mod\widehat{\Lambda})$};
\node (B2) at (4,-1.5) {$\Db(\mod\widehat{\Lambda}^+)$.};
\draw[->] (A1) -- node[above] {$\scriptstyle morita$} (B1);
\draw[->] (A2) -- node[above] {$\scriptstyle morita$} (B2);
\draw[->] (A1) -- node[left] {$\scriptstyle \RHom_{\AB}(\Hom_\mathfrak{R}(N,N^+),-)$} (A2);
\draw[->] (B1) -- node[right] {$\scriptstyle  \RHom_{\widehat{\Lambda}}(\widehat{Z},-)$} (B2);
\end{tikzpicture}
\end{array}\label{the last morita}
\end{eqnarray}
Say curve $i$ has been flopped, then by \cite[4.17(1)]{HomMMP} $N^+=\upnu_i N$, and so the left-hand functor is the mutation functor (see e.g.\ \cite[2.22(1)]{HomMMP} or \ref{elementary mut for flops}\eqref{elementary mut for flops 2}), which is an equivalence.  Hence the right-hand functor is also an equivalence, and the statement follows.  \\
(2) For each $\n\in\Max R$ consider the formal fibre version of the diagram.  If $\n\neq \m$ then the diagram clearly commutes, since $f\colon U\to\Spec R$ is an isomorphism away from $\m$.  If $\n=\m$ then the diagram commutes by combining \eqref{the last morita} with \cite[4.2]{HomMMP}.

Now write $\Uppsi$ for a composition of the four equivalences in the square \eqref{Zar local comm} to give an autoequivalence of $\Db(\coh U)$. Consider a skyscraper $\cO_u\in\Db(\coh U)$. Since the formal fibre versions of the diagram \eqref{Zar local comm} commute, it follows that $\Uppsi$ fixes $\cO_u$. Furthermore $\Uppsi$ fixes $\cO_U$, and commutes with the pushdown $\RDerived f_*$ by the compatibility result of \cite[2.14(2)]{HomMMP}. We conclude that $\Uppsi$ is functorially isomorphic to the identity, and thence that \eqref{Zar local comm} is functorially commutative.
\end{proof}

Using the above, we can now extend \ref{2braidcompletelocal} into the Zariski local setting:
\begin{thm}\label{braid 2 curve local}
Under the Zariski local flops setup $f\colon U\to\Spec R$ of \ref{flopslocal},  suppose that $f$ contracts precisely two independently floppable irreducible curves to a point $p$. Then
\[
\underbrace{\flop_1\circ\flop_2\circ\flop_1\circ\cdots}_{d}\functcong
\flop_{\{1,2\}}\functcong
\underbrace{\flop_2\circ\flop_1\circ\flop_2\circ\cdots}_{d}
\]
where $d$ is the number of hyperplanes in $\cH^p$.  Further, $d\geq 3$.
\end{thm}
\begin{proof}
Since each algebraic flop, after passing to the formal fibre, is still a flop, 
iteratively flopping all possible combinations of all possible subsets of the two curves (which, since the two curves are individually floppable, we may do) gives the same number of schemes in both cases, and the combinatorics are the same for both.  In particular $d\geq 3$ by \ref{2braidcompletelocal}.

To show braiding for a chain of algebraic flops, again we track skyscrapers $\cO_u$ under the chain
\begin{eqnarray}
(\flop_1\circ\flop_2\circ\flop_1\circ\cdots)\circ(\cdots\circ\flop_2^{-1}\circ\flop_1^{-1}\circ\flop_2^{-1}).\label{flop way round alg}
\end{eqnarray}
Using \ref{alg comm}\eqref{alg comm 2}, we can reduce the problem to tracking $\upvartheta$-stable $\Lambda$-modules $M$ for $\upvartheta\in C_+$.  Since $M$ is supported only at a single point $\n$ as an $R$-module, 
\[
\RHom_\Lambda(Z,M)\otimes_R\widehat{R}_\n\cong\RHom_{\widehat{\Lambda}_\n}(\widehat{Z}_\n,M)
\]
and so it suffices to track $M$ complete locally.  If $\n\neq \m$ it is clear that $M$ tracks to itself. If $\n=\m$ then by \eqref{the last morita} $\RHom_{\widehat{\Lambda}}(\widehat{Z},-)$ is naturally isomorphic to the mutation functor.  Hence by \ref{2braidcompletelocal} all skyscrapers track to skyscrapers under \eqref{flop way round alg}, so again since the structure sheaf is fixed and the functors all commute with the pushdown, \eqref{flop way round alg} is functorially isomorphic to the identity. A similar argument shows that
\[
\flop_{\{1,2\}}\circ(\cdots\circ\flop_2^{-1}\circ\flop_1^{-1}\circ\flop_2^{-1}) 
\]
is functorially isomorphic to the identity, and the result follows.
\end{proof}

Next we focus on the global setting of \ref{braid 2 curve global}, for the case of two intersecting curves. Recall that we have a contraction $f\colon X\to X_{\con}$ mapping these curves to a point $p$. As in \S\ref{global flops setup}, put $C=f^{-1}(p)$, so that $C^{\redu}=C_1\cup C_2$ with $C_i\cong\mathbb{P}^1$. Then we choose an affine open neighbourhood $U_{\con} \cong \Spec R$ around $p$, so that setting $U:=f^{-1}(U_{\con})$, we have the commutative diagram
\begin{eqnarray*}
\begin{array}{c}
\begin{tikzpicture}
\node (C) at (-0.6,0) {$C^{}$}; 
\node (U) at (1,0) {$U$};
\node (X) at (3,0) {$X$};
\draw[right hook->] (C) to node[pos=0.6,above] {$\scriptstyle e$} (U);
\draw[right hook->] (U) to node[above] {$\scriptstyle i$} (X);

\node (x) at (-0.6,-1.5) {$\phantom{R}p\phantom{R}$}; 
\node (Uc) at (1,-1.5) {$U_{\con}$};
\node (Xc) at (3,-1.5) {$X_{\con}$};
\node at (0.1,-1.5) {$\in$}; 
\draw[right hook->] (Uc) to (Xc);

\draw[->] (X) --  node[right] {$\scriptstyle f$} (Xc);
\draw[->] (U) --  node[right] {$\scriptstyle f|_U$}  (Uc);
\draw[|->] (C) --  (x);
\end{tikzpicture}
\end{array}
\end{eqnarray*}
where $e$ is a closed embedding, and $i$ is an open embedding. Choose one of the curves $C_i$, and write $U^+$ and $X^+$ for the schemes obtained by flopping $C_i$ in $U$ and $X$ respectively. The following is similar to \cite[7.8]{DW1}, and is well-known to experts.

\begin{lemma}\label{diagram commutes U to X}  In the setting of \ref{braid 2 curve global}, and with notation as above, the following diagram is naturally commutative.
\[
\begin{array}{c}
\opt{10pt}{\begin{tikzpicture}}
\opt{12pt}{\begin{tikzpicture}[scale=1.2]}
\node (a1) at (0,0) {$\D(\Qcoh U)$};
\node (a2) at (3,0) {$\D(\Qcoh X)$};
\node (b1) at (0,-1.5) {$\D(\Qcoh U^+)$};
\node (b2) at (3,-1.5) {$\D(\Qcoh X^+)$};
\draw[->] (a1) to node[above] {$\scriptstyle \Ri_*$} (a2);
\draw[->] (b1) to node[above] {$\scriptstyle \Ri^+_*$} (b2);
\draw[->] (a1) to node[left] {$\scriptstyle \flop_i$} (b1);
\draw[->] (a2) to node[right] {$\scriptstyle \flop_i$} (b2);
\end{tikzpicture}
\end{array}
\]
\end{lemma}

\begin{proof}
Write $\Gamma_U$ (respectively $\Gamma_X$) for the fibre product of $U$ (respectively $X$) with its flop, over the contracted base $U_{\con}$ (respectively $X_{\con}$). Then there are maps 
\[
g_U\colon \Gamma_U \to U \times U^+\quad\text{and}\quad g_X\colon \Gamma_X \to X \times X^+,
\]
and a natural inclusion $\iota\colon \Gamma_U \to \Gamma_X$. Using this notation to translate the claim into the language of Fourier--Mukai kernels \cite[5.12]{HuybrechtsFM}, we require that
\[
\LDerived(i \times \Id)^* g_{X *} \cO \cong \RDerived(\Id \times i^+)_* g_{U *} \cO.
\]
This follows by considering the diagram
\[
\begin{array}{c}
\opt{10pt}{\begin{tikzpicture}}
\opt{12pt}{\begin{tikzpicture}[scale=1.2]}
\node (a1) at (0,0) {$\Gamma_U$};
\node (a2) at (3,0) {$U \times X^+$};
\node (b1) at (0,-1.5) {$\Gamma_X$};
\node (b2) at (3,-1.5) {$X \times X^+$};
\draw[->] (a1) to node[above] {$\scriptstyle h $} (a2);
\draw[->] (b1) to node[above] {$\scriptstyle g_X$} (b2);
\draw[->] (a1) to node[left] {$\scriptstyle \iota$} (b1);
\draw[->] (a2) to node[right] {$\scriptstyle i \times \Id$} (b2);
\end{tikzpicture}
\end{array}
\]
where $h$ is the natural map $(\Id \times i^+) \circ g_U $. Under the birational correspondence between $X$ and $X^+$,  points of $U$ correspond to points of $U^+$ and vice versa, so this square is cartesian. The right-hand map is flat, and so base change gives $\LDerived(i \times \Id)^* \RDerived g_{X *} \cong \RDerived h_*  \LDerived \iota^*.$ The result follows by applying this isomorphism to $\cO$ on $\Gamma_X$.
\end{proof}

Using \ref{diagram commutes U to X}, we can extend \ref{braid 2 curve local} to the global setting:

\begin{cor}\label{braid 2 curve global intersect}
With the global quasi-projective flops setup $f\colon X\to X_{\con}$ of \ref{flopsglobal}, suppose that $f$ contracts precisely two independently floppable irreducible curves to a point $p$. Then
\[
\underbrace{\flop_1\circ\flop_2\circ\flop_1\circ\cdots}_{d}\functcong
\flop_{\{1,2\}}
\functcong
\underbrace{\flop_2\circ\flop_1\circ\flop_2\circ\cdots}_{d}
\]
where $d$ is the number of hyperplanes in $\cH^p$.  Further, $d\geq 3$.
\end{cor}
\begin{proof}
Once again, we track skyscrapers $\cO_x$ under the chain
\begin{eqnarray}
\flop_{\{1,2\}}\circ(\cdots\circ\flop_2^{-1}\circ\flop_1^{-1}\circ\flop_2^{-1}).\label{flop way round alg 2}
\end{eqnarray}
If $x\notin U$, then certainly $x\notin C$, so all the flop functors take $\cO_x$ to $\cO_x$, and hence the composition does.  On the other hand, if $x\in U$ then we can combine \ref{braid 2 curve local} and \ref{diagram commutes U to X} to conclude that the chain of functors in \eqref{flop way round alg 2} takes $\cO_x$ to $\cO_x$.  Thus, either way, \eqref{flop way round alg 2} preserves skyscrapers, so again since the structure sheaf is fixed and the functors commute with the pushdown, \eqref{flop way round alg} is functorially isomorphic to the identity. This gives one of the isomorphisms in the statement, and the other follows by symmetry.
\end{proof}

In contrast, the disjoint curves case is easy, and is well-known.
 
\begin{lemma}\label{braid 2 curve global disjoint}
With the global quasi-projective flops setup $f\colon X\to X_{\con}$ of \ref{flopsglobal}, suppose that $f$ contracts precisely two independently floppable disjoint curves. If the first curve contracts to a point $p_1$, and the second curve contracts to $p_2$, then
\begin{enumerate}
\item $\cH=\cH^{p_1}\textstyle\prod\cH^{p_2}=\{ \{(x,y)\in\mathbb{R}^2\mid x=0\}, \{(x,y)\in\mathbb{R}^2\mid y=0\}\}$, and so $d$, the number of hyperplanes in $\cH$, equals two. 
\item There is a functorial isomorphism
\[
\underbrace{\flop_1\circ\flop_2}_{d}\functcong\flop_{\{1,2\}}\functcong\underbrace{\flop_2\circ\flop_1}_{d}.
\]
\end{enumerate}
\end{lemma}
\begin{proof}
(1) Since each $\cH^{p_i}$ is a simplicial hyperplane arrangement in $\mathbb{R}^1$, this is clear.\\
(2) Since the curves are disjoint, from the definition of flop it is immediate that the order of flops does not matter.  Further, since flop functors are local over the common singular base, chasing skyscrapers and using \ref{alg comm} the result is immediate, using the same argument as in \ref{braid 2 curve global intersect}.
\end{proof}

\subsection{Braiding:\ General Case}\label{braid global section}
The aim of this subsection is to use the braiding in the two-curve case above, together with properties of simplicial hyperplane arrangements, to extend the braiding results to more than two curves.  

For this, recall that simplicial hyperplane arrangements were studied in the seminal work of Deligne \cite{Deligne}, and the resulting \emph{Deligne groupoid} controls much of the combinatorics.  This has many equivalent definitions in the literature, but here for convenience we follow Paris \cite[\S2.A, \S2.B]{Paris}.

\begin{defin}\label{Paris defin}
Given the real simplicial hyperplane arrangement $\cH$ from \S\ref{hyper local global section}, we first associate the oriented graph $X_1$ which has vertices $v_i$ corresponding to the chambers of $\cH$, with an arrow $a\colon v_i\to v_j$ between pairs of vertices corresponding to adjacent chambers.  A~\emph{path of length~$n$} in $X_1$ is defined to be a formal symbol
\[
a_1^{\varepsilon_1}a_2^{\varepsilon_2}\hdots a_n^{\varepsilon_n}
\]
with each $\varepsilon_i\in\{\pm 1\}$, whenever there exists a sequence of vertices $v_0,\hdots,v_n$ of $X_1$ such that $a_i\colon v_{i-1}\to v_i$ if $\varepsilon_i=1$, and $a_i\colon v_{i-1}\leftarrow v_i$ if $\varepsilon_i=-1$.  A path is said to be \emph{positive} if each $\varepsilon_i=1$. Such a path is called \emph{minimal} if there is no positive path in $X_1$ of smaller length, and with the same endpoints.

Let $\sim$ be the smallest \emph{identification} (an equivalence relation satisfying appropriate properties, see \cite[p152]{Paris}) on the set of paths of $X_1$ such that if $f$ and $g$ are both positive minimal paths with the same endpoints then $f\sim g$.  Then the pair $(X_1,\sim)$ determines a groupoid, the \emph{Deligne groupoid} $\mathds{G}_{\cH}$, where the objects are the vertices, and the morphisms are the equivalence classes of paths.
\end{defin}

\begin{example}
The following illustrates part of the oriented graph $X_1$ for the hyperplane arrangement $\cH$ in $\mathbb{R}^3$ from \S\ref{Intro new results}.
\[
\begin{array}{c}
\begin{tikzpicture}[scale=1.2,bend angle=15, looseness=1,>=stealth]
\rotateRPY{-5}{-15}{5}
\begin{scope}[RPY]
\filldraw[gray!20] (-1,-1,1) -- (-1,1,-1) -- (1,-1,1) -- (-1,-1,1);
\filldraw[red!80!black!20] (-1,1,1) -- (-1,-1,1) -- (0,0,0) -- cycle;
\filldraw[green!40!black!20] (-1,0,1) -- (0,0,1) -- (0,0,0)-- cycle;
\filldraw[green!60!black!20] (0,-1,1) -- (0,0,1) -- (0,0,0)-- cycle;
\filldraw[blue!20] (-1,1,1) -- (-1,1,-1) -- (1,-1,-1) -- (1,-1,1);
\filldraw[red!80!black!20] (1,1,-1) -- (1,-1,-1) -- (0,0,0) -- cycle;
\filldraw[green!40!black!20] (1,0,-1) -- (1,0,0) -- (0,0,0)-- cycle;
\filldraw[green!50!black!20] (1,-1,0) -- (1,0,0) -- (0,0,0)-- cycle;
\filldraw[yellow!95!black!20] (1,-1,0) -- (1,0,-1) -- (1,0,0)-- cycle;
\filldraw[yellow!95!black!20] (-1,0,1) -- (0,-1,1) -- (0,0,1)-- cycle;
\filldraw[gray!20] (-1,1,-1) -- (1,1,-1) -- (1,-1,1)--(-1,1,-1);
\filldraw[yellow!95!black!20] (-1,1,0) -- (0,1,-1) -- (0,0,0)--(-1,1,0);
\filldraw[green!50!black!20] (-1,1,0) -- (0,0,0) -- (0,1,0) -- cycle;
\filldraw[green!60!black!20] (0,1,-1)--(0,0,0) -- (0,1,0) -- cycle;
\filldraw[red!80!black!20] (-1,1,1) -- (1,1,-1) -- (0,0,0)-- cycle;
\filldraw[green!60!black!20] (0,1,1) -- (0,0,1) -- (0,0,0)-- (0,1,0)-- cycle;
\filldraw[green!50!black!20] (1,1,0) -- (1,0,0) -- (0,0,0)-- (0,1,0)-- cycle;
\filldraw[green!40!black!20] (1,0,1) -- (1,0,0) -- (0,0,0)-- (0,0,1)-- cycle;
\node (A) at (1,0.6,1) [DWs] {};
\node (B1) at (-0.6,0.7,0.6) [DWs] {};
\node (B2) at (-1,0.9,0) [DWs] {};
\node (C1) at (0.6,0.7,-0.7) [DWs] {};
\node (C2) at (0,0.9,-1) [DWs] {};
\node (D) at (-1,0.7,-1.2) [DWs] {};
\node (E) at (-1,0.9,-1) [DWs] {};
\node (F) at (1,-0.3,1) [DWs] {};
\node (F1) at (0.3,-0.6,1) [DWs] {};
\node (F2) at (-0.3,-0.3,1) [DWs] {};
\node (J) at (-0.8,-0.8,0.9) [DWs] {};
\node (F3) at (-0.6,0.3,1) [DWs] {};
\node (G1) at (1,-0.6,0.3) [DWs] {};
\node (G2) at (1,-0.3,-0.3) [DWs] {};
\node (H) at (0.9,-0.8,-0.8) [DWs] {};
\node (G3) at (1,0.3,-0.6) [DWs] {};
\draw[->,bend right] (A) to (B1);
\draw[->,bend right] (B1) to (A);
\draw[->,bend right] (B1) to (B2);
\draw[->,bend right] (B2) to (B1);
\draw[->,bend right] (B2) to (D);
\draw[->,bend right] (D) to (B2);
\draw[->,bend right] (D) to (E);
\draw[->,bend right] (E) to (D);
\draw[->,bend right] (A) to (C1);
\draw[->,bend right] (C1) to (A);
\draw[->,bend right] (C1) to (C2);
\draw[->,bend right] (C2) to (C1);
\draw[->,bend right] (C2) to (D);
\draw[->,bend right] (D) to (C2);
\draw[-,bend left] (E) to (-1,0.88,-1.15); 
\draw[-,bend left] (-1,0.86,-1.15) to (E); 
\draw[-,bend left] (E) to (-1.1,0.83,-1); 
\draw[-,bend left]  (-1.11,0.85,-1) to (E); 
\draw[-,bend left] (B2) to (-1.1,0.83,0);
\draw[-,bend left] (-1.12,0.84,0) to (B2);
\draw[<-,bend left] (C2) to (0.075,0.9,-1.15);
\draw[-,bend left] (0.09,0.875,-1.15) to (C2);
\draw[->,bend right] (A) to (F);
\draw[->,bend right] (F) to (A);
\draw[->,bend right] (F) to (F1);
\draw[->,bend right] (F1) to (F);
\draw[->,bend right] (F1) to (F2);
\draw[->,bend right] (F2) to (F1);
\draw[->,bend right] (F2) to (F3);
\draw[->,bend right] (F3) to (F2);
\draw[->,bend right] (F3) to (B1);
\draw[->,bend right] (B1) to (F3);
\draw[->,bend right] (F) to (G1);
\draw[->,bend right] (G1) to (F);
\draw[->,bend right] (G1) to (G2);
\draw[->,bend right] (G2) to (G1);
\draw[->,bend right] (G2) to (G3);
\draw[->,bend right] (G3) to (G2);
\draw[->,bend right] (G3) to (C1);
\draw[->,bend right] (C1) to (G3);
\draw[->,bend right] (G2) to (H);
\draw[->,bend right] (H) to (G2);
\draw[->,bend right] (F2) to (J);
\draw[->,bend right] (J) to (F2);
\draw[bend right] (G3) to (1.025,0.42,-0.73);
\draw[->,bend right] (1,0.45,-0.7) to (G3);
\draw[-,bend right] (H) to (0.86,-0.65,-1);
\draw[->,bend right] (0.825,-0.65,-1) to (H);
\draw[-,bend right] (H) to (1.025,-0.88,-0.4);
\draw[->,bend right] (1.05,-0.9,-0.425) to (H);
\draw[-,bend right] (G1) to (0.9,-0.9,0.5);
\draw[->,bend right] (0.9,-0.925,0.45) to (G1);
\draw[-,bend right] (F1) to (0.5,-0.9,0.9);
\draw[->,bend right] (0.55,-0.875,0.9) to (F1);
\draw[-,bend right] (F3) to (-0.75,0.42,0.9);
\draw[->,bend right]  (-0.8,0.4,0.9) to (F3);
\draw[-,bend right] (J) to (-0.86,-0.625,0.9);
\draw[->,bend right] (-0.9,-0.65,0.9) to (J);
\draw[-,bend right] (J) to (-0.65,-0.95,0.9);
\draw[->,bend right] (-0.65,-0.935,0.85) to (J);
\end{scope}
\end{tikzpicture}
\end{array}
\]
\end{example}

\begin{remark}\label{fund gp remark}
By \cite{Paris, Salvetti} (see \cite[2.1]{Paris2}), any vertex group of the groupoid $\mathds{G}_{\cH}$ defined above is isomorphic to $\fundgp(\mathbb{C}^n\backslash \cH_\mathbb{C})$, the fundamental group of the complexified complement of $\cH$.  We thus let $\fundgp(\mathds{G}_{\cH})$ denote a vertex group of the Deligne groupoid. 
\end{remark}

By \cite[1.10, 1.12]{Deligne}, to produce a representation of the Deligne groupoid, it is sufficient to check certain codimension-two relations. 
As in the previous subsection, we first consider this problem in the formal fibre setting.
\begin{lemma}\label{crash=flopboth}
With notation as in the complete local setup of \S\ref{moduli tracking section}, suppose that $c$ is a chamber in $\Uptheta_\AB$ with a codimension-two wall $w$.  Then
\begin{enumerate}
\item\label{crash=flopboth 1} From $c$, crashing through $w$ corresponds to flopping a pair of curves $C_{i_1}\bigcup C_{i_2}$ in $\cM_{\rk,c}(\AB)$, where $\cM_{\rk,c}(\AB)$ is the scheme of $c$-stable $\AB$-modules of dimension vector $\rk$ corresponding to the chamber $c$ in $\Uptheta_\AB$. 
\item\label{crash=flopboth 2} Iterating $\cdots\circ \flop_{i_1}\circ \flop_{i_2}\circ \flop_{i_1}$ traverses one direction around the codimension-two wall $w$, whilst $\cdots\circ \flop_{i_2}\circ \flop_{i_1}\circ \flop_{i_2}$ traverses the other direction.
\end{enumerate}
\end{lemma}
\begin{proof}
Viewing $\cM_{\rk,c}(\AB)$ abstractly, we associate an algebra $\BB$ to it using the procedure in \S\ref{completelocalnotation}. In turn, this algebra has a chamber structure, which we denote $\Uptheta_{\BB}$. By \ref{C+chamber}\eqref{C+chamber 3}, the chamber $c$ in $\Uptheta_{\AB}$ is the tracking, under moduli tracking, of the chamber $C_+$ in $\Uptheta_{\BB}$.  Furthermore, since under moduli tracking walls get sent to walls,  the codimension-two wall $w$ corresponds to one of the codimension-two walls of $C_+$ in $\Uptheta_{\BB}$, which without loss of generality we can assume is  $x_1=x_2=0$.  Thus, since as schemes $\cM_{\rk,c}(\AB)=\cM_{\rk,C_+}\!(\BB)$, to prove \eqref{crash=flopboth 1} it suffices to prove that, from $C_+$ in $\Uptheta_\BB$, crashing through the codimension-two wall $x_1=x_2=0$  corresponds to flopping $C_1\bigcup C_2$.  But this is immediate by \cite[4.16]{HomMMP} and moduli tracking.

To prove \eqref{crash=flopboth 2}, since in all the chamber structures crossing codimension-one walls corresponds to a flop \cite[\S5]{HomMMP}, it suffices to show that iterating $\cdots\circ \flop_1\circ \flop_2\circ \flop_1$ traverses one direction around the codimension-two wall $x_1=x_2=0$ in $\Uptheta_\BB$, whilst $\cdots\circ \flop_2\circ \flop_1\circ \flop_2$ traverses the other direction.

Beginning in the chamber $C_+$ of $\Uptheta_\BB$, flopping $\flop_1$ corresponds to crashing through the single codimension-one wall $\upvartheta_1=0$ \cite[5.21]{HomMMP}, and produces a new chamber that by the moduli tracking formula has a codimension-two wall $x_1=x_2=0$.  This new chamber can then be viewed as $C_+$ on another algebra. Repeating the argument for that chamber then tracking back to $\Uptheta_\BB$, the iterate $\flop_2\circ \flop_1$ corresponds to crashing through two consecutive walls of $\Uptheta_\BB$.  Further, by the moduli tracking formula, both of the obtained chambers share the codimension-two wall $x_1=x_2=0$.  By induction, iterating $\cdots\circ \flop_1\circ \flop_2\circ \flop_1$ produces a series of chambers, at each stage crossing a single codimension-one wall, and each chamber has a codimension-two wall $x_1=x_2=0$.   It follows that iterating $\cdots\circ \flop_1\circ \flop_2\circ \flop_1$ traverses one direction around the codimension-two wall $x_1=x_2=0$.  By symmetry of the argument, necessarily $\cdots\circ \flop_2\circ \flop_1\circ \flop_2$ traverses the other direction. 
\end{proof}

\begin{defin}\label{flops data}
With the global quasi-projective flops setup $f\colon X\to X_{\con}$ of \ref{flopsglobal}, suppose that $f$ contracts precisely $n$ independently floppable irreducible curves. Given this data, the derived flops groupoid $\DF$ is defined by the following generating set.  It has vertices $\Db(\coh X)$, running over all varieties obtained from $X$ by iteratively flopping the $n$ curves, and as arrows we connect vertices by the Bridgeland--Chen flop functors, running through all possible combinations of single flopping curves.
\end{defin}

The following is the main result of this section.
\begin{thm}\label{flopsglobalmultiple}
With the global quasi-projective flops setup $f\colon X\to X_{\con}$ of \ref{flopsglobal}, suppose that $f$ contracts precisely $n$ independently floppable curves. Then:
\begin{enumerate}
\item\label{flopsglobalmultiple 1} There is a homomorphism of groupoids $\mathds{G}_\cH\to\DF$.
\item\label{flopsglobalmultiple 2} The group $\fundgp(\mathds{G}_{\cH})$ acts on $\Db(\coh X)$.
\end{enumerate}
\end{thm}
\begin{proof}
As in \ref{braid 2 curve local}, since each algebraic flop, after passing to the formal fibre, is still a flop, iteratively flopping all possible single curves gives the same number of schemes in both cases, and the combinatorics are the same for both. Hence the braiding of the algebraic flop functors is governed by the same simplicial hyperplane arrangement $\cH$.

By \cite[1.10, 1.12]{Deligne} (see also \cite{CM}), to prove \eqref{flopsglobalmultiple 1} we only need to check that the relations on the flop functors in \ref{flops data} arising from each codimension-two wall are satisfied by the algebraic flop functors.  But this follows immediately from \ref{braid 2 curve global} and \ref{crash=flopboth}.  Part \eqref{flopsglobalmultiple 2} follows directly from \eqref{flopsglobalmultiple 1}.
\end{proof}

\begin{remark}
The above proof does not require a presentation of $\fundgp(\mathds{G}_\cH)$, which is convenient since we do not know one in general.  It is known how to obtain a presentation given the explicit hyperplanes \cite{Arvola, Randell, Salvetti}, but all the possible simplicial hyperplane arrangements arising from flops have not yet been fully classified.
\end{remark}

\section{Mutation in the Flops Setting}\label{mut whole section}

We now work towards dropping the assumption that all the curves are individually floppable.  The aim of this section is to apply the mutation in \S\ref{mut prelim} to the setting of flops to obtain on the formal fibre various intrinsic derived autoequivalences.  These will then be made algebraic in \S\ref{ZLT and GT section}, and will give intrinsic algebraic autoequivalences regardless of whether the curves flop individually. 

The results related to the fibre twist in \S\ref{mu of refl section} and \S\ref{cltwists} will require an additional assumption on the singularities.

\subsection{Mutation for Flops}\label{mut for flops section}

We keep the complete local flops setup of \S\ref{completelocalnotation}, and in particular the notation of \ref{N notation} where $\AB:=\End_{\mathfrak{R}}(N)$.  As in \ref{def basic algebra}, for any $J\subseteq \{0,1,\hdots,n\}$ we set $N_J:=\bigoplus_{j\in J}N_j$ and $N_{J^c}:=\bigoplus_{i\notin J}N_i$, so that $N=N_J\oplus N_{J^c}$.

\begin{remark}\label{AXc is what we think it is}
In later sections we will be interested in two special cases.  The first is when $J\subseteq\{1,\hdots,n\}$, as this will give the $J$-twist corresponding to the noncommutative simultaneous deformations of the family $\{ E_j\mid j\in J\}$. The second will be when $J=\{0\}$, which will give the `fibre twist' corresponding to deformations of the scheme-theoretic fibre $\cO_C$. 
\end{remark}

The following is elementary.
\begin{lemma}\label{elementary mut for flops}
In the complete local flops setup of \ref{flopscompletelocal}, for any $J\subseteq \{0,1,\hdots,n\}$
\begin{enumerate}
\item\label{elementary mut for flops 1} $\AB_J$ is a finite dimensional algebra.
\item\label{elementary mut for flops 2} $T_J\cong\Hom_{\mathfrak{R}}(N,\nuJ N)$, and is a tilting $\AB$-module of projective dimension one.
\end{enumerate}
\end{lemma}
\begin{proof}
Since $\mathfrak{R}$ is an isolated singularity, \eqref{elementary mut for flops 1} is \cite[6.19(3)]{IW4}.   Part \eqref{elementary mut for flops 2} then follows by \cite[6.14]{IW4}.
\end{proof}
As in \S\ref{mut prelim}, the mutation functor gives an equivalence
\[
\Upphi_{J}:=\RHom_{\AB}(T_J,-)\colon\Db(\mod\AB)\xrightarrow{\sim} \Db(\mod\nuJ\AB).
\]
If further $\nuJ\nuJ N\cong N$ (see \ref{mu=nu for hyper} and \ref{mu=nu for R} later) then we can mutate $\End_{\mathfrak{R}}(\nuJ N)$ back to obtain $\End_{\mathfrak{R}}(N)\cong\AB$.  Applying \ref{elementary mut for flops} to $\nuJ N$, $
W_J:=\Hom_\mathfrak{R}(\nuJ N,N)$ is a tilting $\nuJ \AB$-module, giving rise to an equivalence which by abuse of notation we also denote
\[
\Upphi_J:=\RHom_{\nuJ\AB}(W_J,-)\colon\Db(\mod\nuJ\AB)\xrightarrow{\sim} \Db(\mod\AB).
\]
The following is an easy generalisation of \cite[5.9, 5.10, 5.11]{DW1}.
\begin{prop}\label{track Lambda J}
In the complete local flops setup of \ref{flopscompletelocal}, for any $J\subseteq \{0,1,\hdots,n\}$ such that $\nuJ\nuJ N\cong N$, the following statements hold.
\begin{enumerate}
\item\label{track Lambda J 1} $\Upphi_{J}\circ\Upphi_{J}\cong\RHom_{\AB}([N_{J^c}],-)$, where $[N_{J^c}]$ is the two-sided ideal defined in \S\ref{conventions}.
\item\label{track Lambda J 2} $\Upphi_{J}\circ\Upphi_{J}(\AB_J)\cong\AB_J[-2]$.
\item\label{track Lambda J 3} $\Upphi_{J}\circ\Upphi_{J}(S)\cong S[-2]$ for all simple $\AB_J$-modules $S$.
\end{enumerate}
\end{prop}
\begin{proof}
(1) By the assumption $\nuJ\nuJ N\cong N$ it follows that $(\Ker b_J)^*\cong \Ker a_J$.  From here, the proof is then identical to \cite[5.10]{DW1}.\\
(2) Since $(\Ker b_J)^*\cong \Ker a_J$, combining \eqref{K0} and \eqref{key track mut} gives us a complex of $R$-modules
\[
0\to (\Ker b_J)^*\xrightarrow{c_J} V_J\xrightarrow{b_J^{*}\cdot a_{\!J}^{\phantom{*}}} U_J\xrightarrow{d_J^{*}} (\Ker b_J)^* \to 0. 
\]
From here the proof of (2) is word-for-word identical to \cite[5.9(1--2)]{DW1}, since although the above complex need not be exact, whereas it was in \cite[5.9]{DW1}, this does not affect anything.\\
(3) Since by \eqref{track Lambda J 1}  $\Upphi_{J}\circ\Upphi_{J}(-)\cong\RHom_{\AB}([N_{J^c}],-)$, part (3) follows by tensoring both sides of \eqref{track Lambda J 2} by $\AB_{J}/\Rad(\AB_{J})$, just as in \cite[5.11]{DW1}, then applying idempotents.
\end{proof}

Since by \eqref{a_i decomp} $
\widehat{\Lambda}\cong\End_{\mathfrak{R}}(\bigoplus_{i=0}^nN_i^{\oplus a_i})$,
which is only morita equivalent to $\AB$, we need to describe the compatibility between mutation and morita equivalence.   For the positive integers $a_i$ from \eqref{a_i decomp}, we set $Z:=\bigoplus_{i=0}^n N_i^{\oplus a_i}$ so that $\widehat{\Lambda}=\End_{\mathfrak{R}}(Z)$. For a choice of $J\subseteq\{0,1,\hdots,n\}$, consider the summand $Z_J=\bigoplus_{j\in J}N_j^{\oplus a_j}$ of $Z$ and set $Z_{J^c}:=\bigoplus_{i\notin J}N_i^{\oplus a_i}$.
In an identical way to the above, there is a mutation functor 
\[
\Upphi^\prime_{J}:=\RHom_{\widehat{\Lambda}}(\Hom_{\mathfrak{R}}(Z,\nuJ Z),-)\colon\Db(\mod\widehat{\Lambda})\xrightarrow{\sim} \Db(\mod\nuJ\widehat{\Lambda}).
\]
where $\nuJ\widehat{\Lambda}:=\End_{\mathfrak{R}}\Big(\big(\bigoplus_{j\in J}{(\Ker b_j)^*}^{\oplus a_j}\big)\oplus\big(\bigoplus_{i\notin J}N_i^{\oplus a_i}\big)\Big)$. 
The following is elementary.

\begin{lemma}\label{mutation Morita diagram}
The following diagram commutes.
\[
\begin{array}{c}
\begin{tikzpicture}
\node (A1) at (0,0) {$\Db(\mod\AB)$};
\node (A2) at (0,-1.5) {$\Db(\mod\nuJ\AB)$};
\node (B1) at (4,0) {$\Db(\mod\widehat{\Lambda})$};
\node (B2) at (4,-1.5) {$\Db(\mod\nuJ\widehat{\Lambda})$};
\draw[->] (A1) -- node[above] {$\scriptstyle morita$} (B1);
\draw[->] (A2) -- node[above] {$\scriptstyle morita$} (B2);
\draw[->] (A1) -- node[left] {$\scriptstyle \Upphi_{J}$} (A2);
\draw[->] (B1) -- node[right] {$\scriptstyle \Upphi'_{J}$} (B2);
\end{tikzpicture}
\end{array}
\]
\end{lemma}
\begin{proof}
This was stated in \cite[5.8]{DW1}, but since we are working more generally we give the proof here. For simplicity, we drop all $J$ from the notation.  As in \ref{TandSnotation} we denote the top morita functor by $\mathbb{F}$, and we also denote the bottom by $\mathbb{G}$.   Since $P:=\Hom_{\mathfrak{R}}(N,Z)$ is a progenerator, it gives a morita context $(\AB,\End_{\AB}(P),P_{\AB},\Hom_{\AB}(P,\AB))$
which by reflexive equivalence is $(\AB,\widehat{\Lambda},P_{\AB},{}_{\AB}Q)$ where $Q=\Hom_{\mathfrak{R}}(Z,N)$.  Standard morita theory gives an equivalence of categories, and natural isomorphisms
\begin{eqnarray}
\begin{array}{c}
\begin{tikzpicture}[xscale=1]
\node (d1) at (3,0) {$\mod \AB$};
\node (e1) at (8,0) {${}_{}\mod\widehat{\Lambda}$.};
\draw[->,transform canvas={yshift=+0.4ex}] (d1) to  node[above] {$\scriptstyle \mathbb{F}:=\Hom_{\AB}(P,-)\cong-\otimes_{\AB}Q $} (e1);
\draw[<-,transform canvas={yshift=-0.4ex}] (d1) to node [below]  {$\scriptstyle \Hom_{\widehat{\Lambda}}(Q,-)\cong-\otimes_{\widehat{\Lambda}}P $} (e1);
\end{tikzpicture}
\end{array}\label{temp ME diagram}
\end{eqnarray}
There is a similar left version, namely
\begin{eqnarray}
\begin{array}{c}
\begin{tikzpicture}[xscale=1]
\node (d1) at (3,0) {$\mod\AB^{\op}$};
\node (e1) at (8,0) {$\mod\widehat{\Lambda}^{\op}$.};
\draw[->,transform canvas={yshift=+0.4ex}] (d1) to  node[above] {$\scriptstyle \Hom_{\AB^{\op}}(Q,-)\cong P\otimes_{\AB}- $} (e1);
\draw[<-,transform canvas={yshift=-0.4ex}] (d1) to node [below]  {$\scriptstyle \Hom_{\widehat{\Lambda}^{\op}}(P,-) \cong Q\otimes_{\widehat{\Lambda}}- $} (e1);
\end{tikzpicture}
\end{array}\label{temp ME diagram 2}
\end{eqnarray}
Now on one hand 
\begin{align*}
\Upphi^\prime\circ \mathbb{F}&\cong\RHom_{\AB}(\Hom_{\mathfrak{R}}(Z,\upnu Z)\otimes_{\widehat{\Lambda}}P,-)\tag{by adjunction}\\
&\cong \RHom_{\AB}(\Hom_{\widehat{\Lambda}}(Q,\Hom_{\mathfrak{R}}(Z,\upnu Z)),-) \tag{by \eqref{temp ME diagram}}\\
&\cong \RHom_{\AB}(\Hom_{\mathfrak{R}}(N,\upnu Z),-). \tag{by reflexive equivalence}
\end{align*}
On the other hand $\mathbb{G}=\Hom_{\upnu\AB}(P^\prime,-)$ for $P^\prime:=\Hom_{\mathfrak{R}}(\upnu N,\upnu Z)$, with inverse given by $Q^\prime=\Hom_{\mathfrak{R}}(\upnu Z,\upnu N)$.  Thus
\begin{align*}
\mathbb{G}\circ\Upphi&\cong\RHom_{\AB}(P^\prime\otimes_{\upnu\AB}\Hom_{\mathfrak{R}}(N,\upnu N),-)\tag{by adjunction}\\
&\cong \RHom_{\AB}(\Hom_{(\upnu\AB)^{\op}}(Q^\prime,\Hom_{\mathfrak{R}}(N,\upnu N)),-) \tag{by \eqref{temp ME diagram 2}${}^{\prime}$ }\\
&\cong \RHom_{\AB}(\Hom_{\upnu\AB}(\Hom_{\mathfrak{R}}(N, \upnu N)^*,(Q^\prime)^*),-) \tag{$(-)^*$ duality on first term}\\
&\cong \RHom_{\AB}(\Hom_{\upnu\AB}(\Hom_{\mathfrak{R}}(\upnu N,  N),\Hom_{\mathfrak{R}}(\upnu N,\upnu Z)),-) \\
&\cong \RHom_{\AB}(\Hom_{\mathfrak{R}}(N,\upnu Z),-) \tag{by reflexive equivalence}
\end{align*}
and so $\Upphi^\prime\circ \mathbb{F}\cong\mathbb{G}\circ\Upphi$, as required.
\end{proof}

\begin{cor}\label{noncomm flop flop on universal}
In the complete local flops setup of \ref{flopscompletelocal}, and with notation as in \ref{TandSnotation}, for any $J\subseteq \{0,1,\hdots,n\}$ such that $\nuJ\nuJ N\cong N$, the following statements hold.
\begin{enumerate}
\item $\Upphi'_{J}\Upphi'_{J}(\mathbb{F}\AB_{J})\cong \mathbb{F}\AB_{J}[-2]$.
\item $\Upphi'_{J}\Upphi'_{J}(\mathbb{F}S)\cong \mathbb{F}S[-2]$ for all simple $\AB_J$-modules $S$. 
\end{enumerate}
\end{cor}
\begin{proof}
Since minimal approximations sum, it follows that $\nuJ\nuJ Z\cong Z$.  Thus applying \ref{mutation Morita diagram} twice gives a commutative diagram
\[
\begin{array}{c}
\begin{tikzpicture}
\node (A1) at (0,0) {$\Db(\mod\AB)$};
\node (A3) at (0,-1.5) {$\Db(\mod\AB)$};
\node (B1) at (4,0) {$\Db(\mod\widehat{\Lambda})$};
\node (B3) at (4,-1.5) {$\Db(\mod\widehat{\Lambda}).$};
\draw[->] (A1) -- node[above] {$\scriptstyle \mathbb{F}$} (B1);
\draw[->] (A3) -- node[above] {$\scriptstyle \mathbb{F}$} (B3);
\draw[->] (A1) -- node[left] {$\scriptstyle \Upphi_{J}\Upphi_{J}$} (A3);
\draw[->] (B1) -- node[right] {$\scriptstyle \Upphi'_{J}\Upphi'_{J}$} (B3);
\end{tikzpicture}
\end{array}
\]
Hence the result follows from \ref{track Lambda J}\eqref{track Lambda J 2} and \ref{track Lambda J}\eqref{track Lambda J 3}.
\end{proof}

\subsection{Application 1: The $J$-Twists}\label{mut of CM section} In this subsection we exclude $0$ and consider the special case $\indxset\subseteq\{1,\hdots,n\}$ of \S\ref{mut for flops section}.  The situation $n=1$ was considered in \cite{DW1}.  

\begin{prop}\label{mu=nu for hyper}
In the complete local flops setup of \ref{flopscompletelocal}, for any $\indxset\subseteq \{1,\hdots,n\}$, 
\begin{enumerate}
\item\label{mu=nu for hyper 1} $\upnu_\indxset\upnu_\indxset N\cong N$.
\item\label{mu=nu for hyper 2} The simple $A_\indxset$-modules are precisely $S_\indx$ for $\indx\in \indxset$.
\end{enumerate}
\end{prop}
\begin{proof}
Contract the curves in $\indxset$ to obtain $\mathfrak{U}\to\mathfrak{U}_{\con}\to\Spec\mathfrak{R}$.  Since $\mathfrak{U}$ has only Gorenstein terminal singularities, and this is a flopping contraction, $\mathfrak{U}_{\con}$ has only Gorenstein terminal singularities.  Locally, it follows that $\mathfrak{U}_{\con}$ has only hypersurface singularities, so part \eqref{mu=nu for hyper 1} follows from \cite[2.25]{HomMMP}. Part \eqref{mu=nu for hyper 2} is clear.
\end{proof}

The following is then the multi-curve analogue of \cite[5.6, 5.7]{DW1}.

\begin{prop}\label{pd for N and ext} 
In the complete local flops setup of \ref{flopscompletelocal}, for any $\indxset\subseteq \{1,\hdots,n\}$, 
\begin{enumerate}
\item\label{pd for N and ext 1} The minimal projective resolution of $\AB_J$ as an $\AB$-module has the form
\[
0\to P\to Q_1\to Q_0\to P\to \AB_J\to 0
\]
where $P:=\Hom_\mathfrak{R}(N,N_J)$, and $Q_i\in\add Q$ for $Q:=\Hom_\mathfrak{R}(N,N_{J^c})$.
\item\label{pd for N and ext 2} $\pd_{\widehat{\Lambda}}\widehat{\Lambda}_{J}=3$ and $\pd_{\widehat{\Lambda}} \widehat{I}_J=2$.
\item\label{pd for N and ext 3}  We have
\[
\Ext_{\Lambda}^t(\mathbb{F}\AB_J,T_j) 
\cong\Ext_{\widehat{\Lambda}}^t(\mathbb{F}\AB_J,\widehat{T}_j) 
\cong\Ext_{\AB}^t(\AB_J,S_j) 
\cong\left\{ \begin{array}{cl} \mathbb{C}&\mbox{if }t=0,3\\
0&\mbox{else,}\\  \end{array} \right.
\]
for all $j\in J$, and further $\AB_J$ is a self-injective algebra.
\end{enumerate}
\end{prop}
\begin{proof}
(1) This is \cite[A.7(3)]{HomMMP}.\\
(2) Since projective dimension is preserved across morita equivalence, using \eqref{pd for N and ext 1} it follows that $\pd_{\widehat{\Lambda}}\mathbb{F}\AB_J=3$.  But $\add_{\widehat{\Lambda}}\mathbb{F}\AB_J=\add_{\widehat{\Lambda}}\widehat{\Lambda}_J$, so $\pd_{\widehat{\Lambda}}\widehat{\Lambda}_{J}=\pd_{\widehat{\Lambda}}\mathbb{F}\AB_J=3$.  The statement for $\widehat{I}_J$ is then obvious from the completion of \eqref{Icon ses 1J}.\\
(3) The first two isomorphisms are consequences of the fact that the Ext groups are supported only at $\m$, and the third isomorphism is a consequence of \eqref{pd for N and ext 1}.  Since $\mod\AB_J$ is extension-closed in $\mod\AB$ we have
\begin{eqnarray}
\Ext^1_{\AB_J}(S_j,\AB_J)=\Ext^1_{\AB}(S_j,\AB_J)\cong D\Ext^2_{\AB}(\AB_J,S_j)\label{selfinj sequence}
\end{eqnarray}
where the last isomorphism holds since $\AB$ is 3-sCY \cite[2.22(2)]{IW4}, $\pd_{\AB}\AB_J<\infty$ and $S_j$ has finite length.   Thus \eqref{selfinj sequence} shows that $\Ext^1_{\AB_J}(S_j,\AB_J)=0$ for all $j\in J$.  Since $\AB_J$ is finite dimensional, every finitely generated module is filtered by simples, so it follows that $\AB_J$ is self-injective.
\end{proof}

\subsection{Application 2: The Fibre Twist}\label{mu of refl section} This subsection considers the special case $\indxset=\{0\}$ of \S\ref{mut for flops section}, in which case $\AB_\indxset$ is the fibre algebra $\AB_{\fib}$ of \ref{def basic algebra}. Since $\indxset=\{0\}$, this involves mutating the summand $\mathfrak{R}$, which results in reflexive modules that are not Cohen--Macaulay.  Consequently, there is no easy reason for the assumption in \ref{track Lambda J} to be satisfied, and so this subsection is technically much harder than the previous \S\ref{mut of CM section}.  As a result, this subsection requires additional assumptions.

In the Zariski local setup in \ref{flopslocal}, of which the complete local flops setup \ref{flopscompletelocal} is the formal fibre, $f\colon U\to\Spec R$ is a Zariski local crepant contraction, where $U$ has only Gorenstein terminal singularities, contracting precisely one connected chain $C$ of curves with $C^{\redu}=\bigcup_{\indx=1}^n C_j$ with $C_j \cong \mathbb{P}^1$ to a point $\m$ of $\Spec R$. We have that $\Lambda:=\End_R(R\oplus L)$ is derived equivalent to $U$ and recall from \S\ref{localnotation} that we write $\Lambda_\m\cong\End_{R_\m}(R_\m^{\oplus a}\oplus K)$.  Since $f$ is crepant $\Lambda_\m\in\CM R_\m$, so necessarily $K\in\CM R_\m$.

In this subsection, we make the additional assumption that $U$ (not $\mathfrak{U}$) is $\mathds{Q}$-factorial with only Gorenstein terminal singularities, so that $\Lambda_\m$ is an MMA by \ref{MMAsLocalize}.

\begin{remark}
If we assume in addition that $U$ is complete locally $\mathds{Q}$-factorial, equivalently $\mathfrak{U}$ is  $\mathds{Q}$-factorial (which happens for example if $U$ is smooth), then $\widehat{\Lambda}$ and $\AB$ are MMAs and so all the results in this subsection follow immediately from \cite[\S6.3]{IW4}.  However, later we will be working algebraically and it is well-known that the property of being $\mathds{Q}$-factorial does not pass to the completion.  Thus, since we are working with algebraic assumptions, this subsection requires some delicate global--local arguments.
\end{remark}

\begin{thm}\label{pd thm for fibre}
Consider the formal fibre setup \ref{flopscompletelocal} where in the Zariski local flops setting \ref{flopslocal}, $U$ is in addition $\mathds{Q}$-factorial.  Then
\begin{enumerate}
\item\label{pd thm for fibre 1} $\pd_{\widehat{\Lambda}}\widehat{\Lambda}_{\fib}=3$, $\pd_{\widehat{\Lambda}}\widehat{I}_{\fib}=2$ and $\pd_{\AB}\CAR=3$.
\item\label{pd thm for fibre 2} The minimal projective resolution of the $\AB$-module $\CAR$ has the form
\[
0\to Q\to P_1\to P_0\to Q\to \CAR\to 0
\]
where $Q:=\Hom_\mathfrak{R}(N,\mathfrak{R})$, and $P_i\in\add P$ for $P:=\Hom_\mathfrak{R}(N,\bigoplus_{\indx=1}^nN_\indx)$.
\end{enumerate}
\end{thm}
\begin{proof}
(1) Since $\Lambda_\m$ is an MMA of $R_\m$ by \ref{MMAsLocalize}, and by definition $\Lambda_{\fib}=\Lambda_\m/[K]$, it follows that $\pd_{\Lambda_\m}\Lambda_{\fib}<\infty$ by \cite[4.16]{IW5}.  By \ref{Acon and Afib under FBC}, $\Lambda_{\fib}\otimes_{R_\m}\mathfrak{R}\cong\widehat{\Lambda}_{\fib}$. Since completion preserves finite projective dimension, it follows that $\pd_{\widehat{\Lambda}}\widehat{\Lambda}_{\fib}\leq 3$. Since $\widehat{\Lambda}_{\fib}$ is finite dimensional, a projective dimension strictly less than three would contradict the depth lemma, so $\pd_{\widehat{\Lambda}}\widehat{\Lambda}_{\fib}=3$. The statement  $\pd_{\widehat{\Lambda}}\widehat{I}_{\fib}=2$ follows by completing \eqref{Ifib ses 1}.  Finally, $\add_{\widehat{\Lambda}}\mathbb{F}\AB_{\fib}=\add_{\widehat{\Lambda}}\widehat{\Lambda}_{\fib}$, so $\pd_{\widehat{\Lambda}}\mathbb{F}\AB_{\fib}=\pd_{\widehat{\Lambda}}\widehat{\Lambda}_{\fib}=3$.  Since projective dimension is preserved across morita equivalence, $\pd_{\AB}\AB_{\fib}=3$ follows.\\
(2) This is very similar to the arguments in \cite[4.3, 5.6]{IR} and \cite[6.23]{IW4}, but our assumptions here are weaker, so we include the proof.  Since by \eqref{pd thm for fibre 1} $\CAR$ has finite projective dimension as an $\AB$-module, it follows that $\id_{\CAR}\CAR\leq 1$ by e.g.\ \cite[6.19(4)]{IW4}, in particular the injective dimension is finite.  This being the case, since $\CAR$ is local we deduce that
\[
\depth_{\mathfrak{R}}\CAR=\dim_{\mathfrak{R}}\CAR=\id_{\CAR}\CAR
\]
by Ramras \cite[2.15]{Ramras}.  Since $\CAR$ is finite dimensional, this number is zero and in particular $\CAR$ is a Cohen--Macaulay $\mathfrak{R}$-module of dimension zero.

Let $e$ be the idempotent in $\AB$ corresponding to the summand $\mathfrak{R}$, so that $Q=e \AB$, $P=(1-e)\AB$ and $\CAR=\AB/\AB(1-e)\AB$. By \eqref{pd thm for fibre 1} we have a minimal projective resolution 
\begin{eqnarray}\label{proj of CAR}
0\to P_3\to P_2\to P_1\stackrel{f}{\to} e\AB \to\CAR\to 0
\end{eqnarray}
with $f$ a minimal right ($\add{}(1 - e)\AB$)-approximation since it is a projective cover of $e\AB(1 - e)\AB $.  In particular, $P_1\in\add P$.  Now $\CAR$ is finite dimensional and $\AB$ is perfect, so since $\AB$ is $3$-sCY we have
\[
\Ext^t_{\AB}(\CAR, \AB)\cong D\Ext^{3-t}_{\AB}(\AB,\CAR)
\]
which is zero for $t\neq 3$.  Hence applying $(-)^\vee:=\Hom_{\AB}(-,\AB)$ to \eqref{proj of CAR} we obtain an exact sequence
\begin{eqnarray}\label{eek}
0\to (e\AB)^\vee\to (P_1)^\vee\to (P_2)^\vee\to (P_3)^\vee\to \Ext^3_{\AB}(\CAR,\AB)\to 0
\end{eqnarray}
which is the minimal projective resolution of $\Ext^3_{\AB}(\CAR,\AB)$ as an $\AB^{\op}$-module.  But we have
\[
\Ext^3_{\AB}(\CAR,\AB)\cong\Ext^3_{\mathfrak{R}}(\CAR,\mathfrak{R})
\] 
by \cite[3.4(5)]{IR}, and this is a projective $\CAR^{\op}$-module by \cite[1.1(3)]{GN}.  Since $\CAR$ is a local ring, it is a free $\CAR^{\op}$-module.  Further, since $\CAR$ is a Cohen--Macaulay $\mathfrak{R}$-module of dimension zero, we have 
\[
\Ext_{\mathfrak{R}}^3(\Ext^3_{\mathfrak{R}}(\CAR,\mathfrak{R}),\mathfrak{R})\cong \CAR. 
\]
and so the rank has to be one, forcing $\Ext^3_{\mathfrak{R}}(\CAR,\mathfrak{R})\cong \CAR$ as $\CAR^{\op}$-modules.
Hence \eqref{eek} is the minimal projective resolution 
\[
0\to \Hom_{\AB}(e\AB,\AB)\to \Hom_{\AB}(P_1,\AB)\to \Hom_{\AB}(P_2,\AB)\to \Hom_{\AB}(P_3,\AB)\to \CAR\to 0
\]
of $\CAR$ as an $\AB^{\op}$-module.  Thus we have $\Hom_{\AB}(P_3,\AB)\cong \AB e$ and so $P_3\cong e\AB =Q$.  Similarly $\Hom_{\AB}(P_2,\AB)\in\add\AB(1-e)$ forces $P_2\in\add{}(1-e)\AB=\add P$.
\end{proof}

The following corollary ensures that the assumption in \ref{track Lambda J} is satisfied. 

\begin{cor}\label{mu=nu for R}
Consider the formal fibre setup \ref{flopscompletelocal} where in the Zariski local flops setting \ref{flopslocal}, $U$ is in addition $\mathds{Q}$-factorial. Then for $I=\{0\}$, $\upnu_I\upnu_IN\cong N$. 
\end{cor}
\begin{proof}
Set $M:=\bigoplus_{i=1}^nN_i$ so that $N=N_0\oplus M=\mathfrak{R}\oplus M$. By \ref{elementary mut for flops},
\[
T_I\cong \Hom_\mathfrak{R}(N,(\Ker b_0)^*\oplus M)
\]
is a tilting $\AB$-module with $\pd_{\AB} T_I=1$.  We claim that
\[
T:=\Hom_\mathfrak{R}(N,(\Ker a_0)\oplus M)
\]
is also a tilting $\AB$-module with $\pd_{\AB} T_I=1$.  Then, since the projective dimension of both being one implies that both are not isomorphic to the ring $\AB$, and further as $\AB$-modules $T_I$ and $T$ share all summands except possibly one, by a Riedtmann--Schofield type theorem \cite[4.2]{IR} they must coincide, i.e.\ $\Hom_\mathfrak{R}(N,(\Ker b_0)^*\oplus M)\cong \Hom_\mathfrak{R}(N,(\Ker a_0)\oplus M)$.  By reflexive equivalence and Krull--Schmidt it then follows that $(\Ker b_0)^*\cong \Ker a_0$, proving the statement.

Now $T=\Hom_\mathfrak{R}(N,M)\oplus \Hom_\mathfrak{R}(N,\Ker a_0 )$, and the first summand is projective.  Further the sequence \eqref{begin lambdacon} becomes
\[
0\to \Hom_\mathfrak{R}(N,\Ker a_0)\xrightarrow{c\cdot} \Hom_\mathfrak{R}(N,V_0)\xrightarrow{a\cdot}\Hom_\mathfrak{R}(N,\mathfrak{R})\to\AB_{\fib}\to 0 
\]
which is the beginning of the minimal projective resolution of $\AB_{\fib}$.  Since $\pd_{\AB}\AB_{\fib}=3$ by \ref{pd thm for fibre} it follows that $\pd_{\AB}\Hom_\mathfrak{R}(N,\Ker a_0)=1$ and so $\pd_{\AB}T=1$.

Now by reflexive equivalence $\End_{\AB}(T)\cong\End_{\mathfrak{R}}(M\oplus (\Ker a_0))$, and this is a CM $\mathfrak{R}$-module by \cite[6.10]{IW4}.  Thus it follows that $\depth_{\mathfrak{R}}\Ext^1_{\AB}(T,T)>0$ by the depth lemma (see e.g.\ \cite[2.7]{IW4}).  But on the other hand for any prime $\p\in\Spec\mathfrak{R}$ of height two, $T_\p\in\CM \AB_\p$ since reflexive modules are CM for two-dimensional rings.  Since $
\pd_{\AB_\p}T_\p\leq 1$, Auslander--Buchsbaum then implies  that $T_\p$ is projective, and so $\Ext^1_{\AB}(T,T)_\p=\Ext^1_{\AB_\p}(T_\p,T_\p)=0$, implying that the Ext group has finite length.  Combining, it follows that $\Ext^1_{\AB}(T,T)=0$, and since $\pd_{\AB}T=1$ we deduce that $\Ext^i_{\AB}(T,T)=0$ for all $i>0$.  Hence $T$ is a partial tilting $\AB$-module with exactly $n$ non-isomorphic summands.  By the Bongartz completion and Krull--Schmidt, it follows that $T$ is a tilting $\AB$-module with projective dimension one.
\end{proof}

\subsection{Summary of Complete Local Twists}\label{cltwists} The following corollary summarises the main results in the previous two subsections.

\begin{cor}\label{mut mut twist complete}
With the Zariski local setup in \ref{flopslocal}, and its formal fibre setup in \ref{flopscompletelocal}, then with notation as in \eqref{Icon ses 1J} and \eqref{Ifib ses 1},
\begin{enumerate}
\item\label{mut mut twist complete 1} For any $\indxset\subseteq\{1,\hdots,n\}$, 
\[
\Upphi^\prime_{\indxset}\circ\Upphi^\prime_{\indxset}\cong\RHom_{\widehat{\Lambda}}(\widehat{I}_\indxset,-)
\]
is an autoequivalence of $\Db(\mod\widehat{\Lambda})$, and $\widehat{I}_\indxset$ is a tilting $\widehat{\Lambda}$-module of projective dimension two.  This autoequivalence sends
\begin{enumerate}
\item[(a)] $\mathbb{F}\AB_\indxset\mapsto\mathbb{F}\AB_\indxset[-2]$.
\item[(b)] $\mathbb{F}S_\indx\mapsto \mathbb{F}S_\indx[-2]$ for all $\indx\in \indxset$.
\end{enumerate}
\item\label{mut mut twist complete 2} If further $U$ is $\mathds{Q}$-factorial then
\[
\Upphi^\prime_{0}\circ\Upphi^\prime_{0}\cong\RHom_{\widehat{\Lambda}}(\widehat{I}_{\fib},-)
\]  
is an autoequivalence of $\Db(\mod\widehat{\Lambda})$, and $\widehat{I}_{\fib}$ is a tilting $\widehat{\Lambda}$-module of projective dimension two.  This autoequivalence sends
\begin{enumerate}
\item[(a)] $\mathbb{F}\AB_{\fib}\mapsto\mathbb{F}\AB_{\fib}[-2]$.
\item[(b)] $\mathbb{F}S_0\mapsto \mathbb{F}S_0[-2]$.
\end{enumerate}
\end{enumerate}
\end{cor}
\begin{proof}
(1) The two-sided ideal $[Z_{\indxset^c}]$ of $\widehat{\Lambda}=\End_{\mathfrak{R}}(Z)$ is $\widehat{I}_\indxset$, so the first statement  $\Upphi^\prime_{\indxset}\Upphi^\prime_{\indxset}\cong\RHom_{\widehat{\Lambda}}(\widehat{I}_\indxset,-)$ is identical to \ref{track Lambda J}\eqref{track Lambda J 1}.  Since mutation is an equivalence, this means that $\widehat{I}_\indxset$ must be a tilting module, and it has projective dimension two by \ref{pd for N and ext}\eqref{pd for N and ext 2}. Statements (a) and (b) are special cases of \ref{noncomm flop flop on universal}.\\
(2) The two-sided ideal $[\widehat{K}]$ of $\End_{\mathfrak{R}}(Z)$ is $\widehat{I}_{\fib}$, so the proof is identical to \eqref{mut mut twist complete 1}, except now \ref{pd thm for fibre}\eqref{pd thm for fibre 1} is used to establish that the projective dimension of $\widehat{I}_{\fib}$ is two.
\end{proof}

\section{Global Twists}\label{ZLT and GT section}

In this section we define the $J$-twist associated to a subset $J$ of the exceptional reduced curves in the fibres of a global flopping contraction, as well as the fibre twist associated to the whole scheme-theoretic fibre. The key result is \ref{twists autoequivalence B}, establishing that these are autoequivalences, and describing their action on $\Db(\coh X)$ by certain functorial triangles. We study relations between these autoequivalences, and relate the $J$-twist with the flop functor, in~\S\ref{section auto group}.

\subsection{Zariski Local Twists}\label{ZLT section} We first revert to the Zariski local setup in \ref{flopslocal}, namely $f\colon U\to\Spec R$ is a Zariski local crepant contraction, where $U$ has only Gorenstein terminal singularities, contracting precisely one connected chain $C$ of $n$ curves with $C^{\redu}=\bigcup_{\indx=1}^n C_\indx$ such that each $C_\indx\cong\mathbb{P}^1$.  We have that $U$ is derived equivalent to $\Lambda:=\End_R(R\oplus L)$, with the basic complete local version $\AB=\End_{\mathfrak{R}}(\mathfrak{R}\oplus N)$.

The aim of this subsection is to lift the complete local equivalences established in~\ref{mut mut twist complete} to the Zariski open set $U$. To do this we use the sequences of bimodules \eqref{Icon ses 1} and \eqref{Ifib ses 1}.

\begin{cor}\label{Zar local alg twist}
With the Zariski local flops setup in \ref{flopslocal},
\begin{enumerate}
\item\label{Zar local alg twist 1} For any $J\subseteq\{1,\hdots,n\}$, $I_J$ is a tilting $\Lambda$-module of projective dimension two.
\item\label{Zar local alg twist 2} If $U$ is $\mathds{Q}$-factorial, then $I_{\fib}$ is a tilting $\Lambda$-module of projective dimension two.
\end{enumerate}
In either case, if we let $I$ denote either $I_J$ or $I_{\fib}$, viewing $I$ as a right $\Lambda$-module we have $\Lambda\cong \End_\Lambda(I)$, and under this isomorphism the bimodule ${}_{\End_\Lambda(I)}I_\Lambda$ coincides with the natural bimodule structure ${}_{\Lambda}{I}_\Lambda$.
\end{cor}
\begin{proof}
(1) The statement is local, and since $\Lambda_J$ is supported only at $\m$, it is clear that $(I_J)_\n$ is free for all $\n\in\Max R$ with $\n\neq\m$.  Thus it suffices to check that $(I_J)_\m$ is a tilting $\Lambda_\m$-module of projective dimension two.  Since $R_\m$ is local, and a module being zero can be detected on the completion, the statement is equivalent to $\widehat{I}_J=I_J\otimes_R\mathfrak{R}$ being a tilting $\widehat{\Lambda}$-module of projective dimension two.  But this is just \ref{mut mut twist complete}\eqref{mut mut twist complete 1}.\\
(2) By the same logic as in \eqref{Zar local alg twist 1}, the statement is equivalent to $\widehat{I}_{\fib}$ being a tilting $\widehat{\Lambda}$-module of projective dimension two.  Again, since $U$ is now $\mathds{Q}$-factorial, this is just \ref{mut mut twist complete}\eqref{mut mut twist complete 2}.

The final statements follow immediately by repeating the argument in \cite[6.1]{DW1}.
\end{proof}

The following is the extension of \ref{mut mut twist complete} to the Zariski local setting.

\begin{prop}\label{track Lambda J Prop2}
With the Zariski local setup in \ref{flopslocal}, 
\begin{enumerate}
\item For all $J\subseteq\{1,\hdots,n\}$, 
\begin{eqnarray}
\RHom_{\Lambda}(I_J,-)\colon\Db(\mod\Lambda)\xrightarrow{\sim}\Db(\mod\Lambda),\label{Zar alg 1}
\end{eqnarray}
and further this autoequivalence sends
\begin{enumerate}
\item $\mathbb{F}\AB_J\mapsto\mathbb{F}\AB_J[-2]$,
\item $T_j\mapsto T_j[-2]$ for all $j\in J$.
\end{enumerate}
\item If further $U$ is $\mathds{Q}$-factorial,
\begin{eqnarray}
\RHom_{\Lambda}(I_{\fib},-)\colon\Db(\mod\Lambda)\xrightarrow{\sim}\Db(\mod\Lambda),\label{Zar alg 2}
\end{eqnarray}
and further this autoequivalence sends
\begin{enumerate}
\item $\mathbb{F}\AB_{\fib}\mapsto\mathbb{F}\AB_{\fib}[-2]$,
\item $T_0\mapsto T_0[-2]$.
\end{enumerate}
\end{enumerate}
\end{prop}
\begin{proof}
(1) By \ref{mut mut twist complete}\eqref{mut mut twist complete 1}, $\RHom_{\widehat{\Lambda}}(\widehat{I}_J,\mathbb{F}\AB_J)\cong\mathbb{F}\AB_J[-2]$. Since $\mathbb{F}\AB_J$ is supported only at the point $\m$, $\RHom_{\Lambda}(I_J,\mathbb{F}\AB_J)\cong\mathbb{F}\AB_J[-2]$ is a formal consequence, as in \cite[6.3]{DW1}.  Since $\mathbb{F}S_j=\widehat{T}_j$ by \ref{TandSnotation}, the second statement is also a formal consequence of \ref{mut mut twist complete}\eqref{mut mut twist complete 1}.\\
(2) Similarly $\RHom_{\widehat{\Lambda}}(\widehat{I}_{\fib},\mathbb{F}\AB_{\fib})\cong\mathbb{F}\AB_{\fib}[-2]$ by \ref{mut mut twist complete}\eqref{mut mut twist complete 2}, and again $\mathbb{F}\AB_{\fib}$ is supported only at the point $\m$.
\end{proof}

By passing across the equivalence $\Db(\mod\Lambda)\cong\Db(\coh U)$, \eqref{Zar alg 1} and \eqref{Zar alg 2} give autoequivalences on $U$.  To describe these, consider the projections $p_1,p_2\colon U\times U\to U$, 
and set $\cV^{\vee}\boxtimes \cV=p_1^*\cV^{\vee}\otimes_{\cO_{U\times U}}^{\bf L}p_2^*\cV$.  Then there is an induced derived equivalence 
\begin{eqnarray}
\begin{array}{c}
\begin{tikzpicture}[xscale=1]
\node (d1) at (3,0) {$\Db(\coh U\times U)$};
\node (e1) at (9,0) {$\Db(\mod \Lambda\otimes_{\mathbb{C}}\Lambda^{\op})$};
\draw[->,transform canvas={yshift=+0.4ex}] (d1) to  node[above] {$\scriptstyle \RHom_{U\times U}(\cV^{\vee}\boxtimes \cV,-)$} (e1);
\draw[<-,transform canvas={yshift=-0.4ex}] (d1) to node [below]  {$\scriptstyle -\otimes^{\bf L}_{\Lambda^{\rm e}}(\cV^{\vee}\boxtimes \cV)$} (e1);
\end{tikzpicture}
\end{array}\label{bimodule adj geom alg}
\end{eqnarray}
as in \cite{BH}, where we denote the enveloping algebra by $\Lambda^{\rm e}:=\Lambda\otimes_{\mathbb{C}}\Lambda^{\op}$. Applying the lower functor in \eqref{bimodule adj geom alg} to \eqref{Icon ses 1J} and \eqref{Ifib ses 1}, we obtain exact triangles of Fourier--Mukai kernels on $U\times U$ denoted
\begin{gather}
\cW_{J}\to \cO_{\diag,U}\xrightarrow{\upphi_{J}}\cQ_{J}\to\phantom{.}\label{Zariski local kernel triangles1a inverse}\\
\cW_{\fib}\to \cO_{\diag,U}\xrightarrow{\upphi_{\fib}}\cQ_{\fib}\to.\label{Zariski local kernel triangles1b inverse}
\end{gather}
Dualizing \eqref{Icon ses 1J} to obtain an exact triangle
\begin{eqnarray*}
\RHom_\Lambda(\Lambda_J,\Lambda) \to\Lambda\to \RHom_\Lambda(I_J,\Lambda)\to\label{Icon ses 1 dual}
\end{eqnarray*}
then applying the lower functor in \eqref{bimodule adj geom alg}, and similarly for \eqref{Ifib ses 1}, we obtain exact triangles of Fourier--Mukai kernels on $U\times U$ denoted
\begin{gather}
\cQ'_{J}\xrightarrow{\upphi'_{J}} \cO_{\diag,U}\to\cW'_{J}\to\phantom{.}\label{Zariski local kernel triangles1a}\\
\cQ'_{\fib}\xrightarrow{\upphi'_{\fib}} \cO_{\diag,U}\to\cW'_{\fib}\to.\label{Zariski local kernel triangles1b}
\end{gather}
Then, using the obvious adjunctions given by restriction and extension of scalars from the ring homomorphisms $\Lambda\to\Lambda_J$ and $\Lambda\to\Lambda_{\fib}$, passing through the above derived equivalence and the morita equivalence $\mathbb{F}$ from \ref{TandSnotation}, exactly as in \cite[6.10, 6.11, 6.16]{DW1}, \eqref{Zariski local kernel triangles1a} and \eqref{Zariski local kernel triangles1b} yield functorial triangles
\begin{gather*}
\RHom_U(\cE_{J},-)\otimes^{\bf L}_{\AB_J}\cE_{J}\to \Id\to\FM{\cW'_{J}}\to\phantom{,}\\
\RHom_U(\cE_{\fib},-)\otimes^{\bf L}_{\CAR}\cE_{\fib}\to \Id\to\FM{\cW'_{\fib}}\to,
\end{gather*}
where $\FM{\cW'_{J}}$ is the autoequivalence on $U$ corresponding to \eqref{Zar alg 1}, and  $\FM{\cW'_{\fib}}$ is the autoequivalence on $U$ corresponding to \eqref{Zar alg 2}.

Translating \ref{track Lambda J Prop2} through \eqref{notation summary} immediately gives the following.

\begin{cor}\label{local twists shifting}
With the Zariski local setup in \ref{flopslocal}, 
\begin{enumerate}
\item $\FM{\cW'_{J}}\colon\Db(\coh U)\to\Db(\coh U)$ is an autoequivalence, sending 
\begin{enumerate}
\item $\cE_{J}\mapsto \cE_{J}[-2]$,
\item $E_j\mapsto E_j[-2]$ for all $j\in J$.
\end{enumerate}
\item If further $U$ is $\mathds{Q}$-factorial then $\FM{\cW'_{\fib}}\colon\Db(\coh U)\to\Db(\coh U)$ is an autoequivalence, sending
\begin{enumerate}
\item $\cE_{\fib}\mapsto\cE_{\fib}[-2]$,
\item $E_0\mapsto E_0[-2]$.
\end{enumerate}
\end{enumerate}
\end{cor}

\subsection{Definition of Global Twists}\label{GT section}
We now define twist functors for  the global quasi-projective flops  setup $f\colon X\to X_{\con}$ of \ref{flopsglobal}.   As in \cite[\S7]{DW1}, it is technically easier to construct their inverse functors, and this proceeds by gluing in the Zariski local construction of \S\ref{ZLT section}.

Recall that for each point $p\in\noniso f$ we chose an affine open neighbourhood containing $p$ but none of the other points in $\noniso f$, and considered the open set in $X$ given by the inverse image of this affine open, under $f$, which we denote by $U_p$, with inclusion $i\colon U_p \hookrightarrow X$. 


Given the setup above, we have the following definition. We denote the projections from $X\times X$ to its factors by $\FMproja_1, \FMproja_2$. 

\begin{defin}\label{defin twists at a point}
With the global quasi-projective flops setup of \ref{flopsglobal}, and notation in \eqref{openopen}, for any $J\subseteq\{1,\hdots,n_p\}$,
\begin{enumerate}
\item\label{defin twists at a point 1} The inverse $J$-twist functor $\Tconp^*$ on $X$ is defined to be
\[ 
\Tconp^*:=\RDerived \FMproja_{2*}(\cC_{J}\otimes^{\bf L}_{X\times X}\FMproja_1^*(-))
\]
where $\cC_{J}:=\Cone\upphi_{J,X}[-1]$, with $\upphi_{J,X}$ defined to be the composition
\begin{equation}\label{twist cone morphism}
\cO_{\diag,X}
\xrightarrow{\upeta_\diag}
\RDerived(i \times i)_* \cO_{\diag,U_p} \xrightarrow{\RDerived(i \times i)_*\upphi_{J}} \RDerived(i \times i)_* \cQ_{J}.
 \end{equation}
The notation here comes from \eqref{Zariski local kernel triangles1a inverse}, and the natural morphism $\upeta_\diag$ is described explicitly in \cite[7.3]{DW1}.  
\item\label{defin twists at a point 2} If $U_p$ is $\mathds{Q}$-factorial, the inverse fibre twist $\Tfibp^*$ on $X$ is defined in an identical way, using instead  $\upphi_{\fib}$ from \eqref{Zariski local kernel triangles1b inverse}.
\end{enumerate} 
\end{defin}

We may then make the following general definition of (inverse) $J$-twist and fibre twists for a flopping contraction. Recall $n_p$ denotes the number of curves above a point $p\in\noniso f$.

\begin{defin}\label{twists definition} 
With the global quasi-projective flops setup of \ref{flopsglobal}, we enumerate the points of the non-isomorphism locus $\noniso f$ as $\{p_1,\hdots,p_t\}$, and let $J$ denote a collection $(J_1,\hdots,J_t)$ of subsets $J_i\subseteq\{1,\hdots,n_{p_i}\}$. We then define the \emph{inverse $J$-twist} to be
\[
\Tcon^* := \Tconpiinv{1}\circ\cdots\circ\Tconpiinv{t}.
\]
Similarly, when $X$ is $\mathds{Q}$-factorial we set 
\[
\Tfib^* := \Tfibpiinv{1}\circ\cdots\circ\Tfibpiinv{t}.
\]
and call it the \emph{inverse fibre twist}.
\end{defin}

\begin{remark}\label{commute remark}
The order of the compositions in the definitions does not matter, since the individual functors all commute. This may be shown, for instance, by tracking skyscraper sheaves as in \ref{braid 2 curve global disjoint}, or using \ref{sky off curves}\eqref{sky off curves 1} later.
\end{remark}
We will prove in \ref{twists autoequivalence B} that these functors are always equivalences, regardless of whether $\bigcup_{j\in J}C_j$ flops algebraically.  Being able to vary $J$ without worrying about whether the curves flop is the key to obtaining group actions on derived categories without strong assumptions on the flopping contraction.

Since $X$ is quasi-projective, it embeds as an open subscheme of a projective variety~$\projenv$, so that we have 
\begin{eqnarray}
U_p\stackrel{i}{\hookrightarrow}X\stackrel{k}{\hookrightarrow} \projenv.\label{openopen}
\end{eqnarray}
We remark that $\projenv$ need not be Gorenstein, as we make no assumptions on the complement of $X$. However, the projectivity of $\projenv$ will be used to prove that the inverse $J$-twist is an equivalence on $\projenv$ in \S\ref{section proj twist is equiv}, before deducing the analogous result on $X$ in \S\ref{section qproj twist is equiv}. The construction of the inverse $J$-twist on $\projenv$ proceeds by gluing, in the same manner as \ref{defin twists at a point}.

\begin{defin}\label{defin twists at a point aux}
With the setup as above,
\begin{enumerate}
\item\label{defin twists at a point aux 1} The inverse $J$-twist functor $\Tconp^*$ on $\projenv$ is defined, by slight abuse of notation, to be
\[ 
\Tconp^*:=\RDerived \FMprojb_{2*}(\cD_{J}\otimes^{\bf L}_{\projenv\times \projenv}\FMprojb_1^*(-))
\]
where $\cD_{J}:=\Cone\upphi_{J,\projenv}[-1]$, with $\upphi_{J,\projenv}$  defined to be the composition
\[
\cO_{\diag,\projenv}
\xrightarrow{\upeta_\diag}
\RDerived(ki \times ki)_* \cO_{\diag,U_p} \xrightarrow{\RDerived(ki \times ki)_*\upphi_J} \RDerived(ki \times ki)_* \cQ_{J}.
 \]
 \item\label{defin twists at a point aux 2} The $J$-twist functor $\Tconp$ on $\projenv$ is defined
\[
\Tconp:=\RDerived \FMprojb_{1*}(\cD_{J}^R\otimes^{\bf L}_{\projenv\times \projenv}\FMprojb_2^*(-))
\] 
where $\cD_{J}^R :=\RsHom_{\projenv\times \projenv}(\cD_{J},\FMprojb_2^!\cO_\projenv)$. 
\item\label{defin twists at a point aux 3} If $U_p$ is $\mathds{Q}$-factorial, the inverse fibre twist $\Tfibp^*$ and fibre twist $\Tfibp$ on $\projenv$ are defined similarly. 
\end{enumerate} 
\end{defin}

\subsection{First Properties}
\label{section first props}

We establish basic properties of the functors defined in \ref{defin twists at a point} and \ref{defin twists at a point aux}, after the following proposition, which studies the pushforwards of $\cE_J$ to $X$ and $\projenv$. Recall that $\Perf(Z)$ denotes the category of perfect complexes on a scheme $Z$.

\begin{prop}\label{universal and canonical}
With the global quasi-projective setup of \ref{flopsglobal}, and notation in \eqref{openopen},
\begin{enumerate}
\item\label{universal and canonical 1} 
For a fixed $p\in\noniso f$\!, and for any $J\subseteq\{1,\hdots,n_p\}$, 
\begin{enumerate}[label=(\alph*), ref=\alph*]
\item\label{universal and canonical 1 a} $\RDerived {i}_*\cE_{J} = {i}_*\cE_{J}\in\Perf(X)$,
\item\label{universal and canonical 1 b} $\RDerived {(ki)}_{\!*}\cE_{J} = {(ki)}_{\!*}\cE_{J}\in\Perf(\projenv)$.
\end{enumerate}
By abuse of notation we denote these simply by $\cE_J$.  Furthermore,
\begin{enumerate}[resume*]
\setcounter{enumii}{2}
\item\label{universal and canonical 1 c} $\cE_J \otimes^{\bf L}_X\omega_X\cong \cE_J$,
\item\label{universal and canonical 1 d} $\cE_J \otimes_\projenv^{\bf L} \omega_\projenv\cong \cE_J$. 
\end{enumerate}
\item\label{universal and canonical 2}
If further $U_p$ is $\mathds{Q}$-factorial, then the corresponding statements hold for $\cE_{\fib}$, and we abuse notation similarly.
\end{enumerate}
\end{prop}
\begin{proof}
(1) By \ref{pd for N and ext} $\pd_{\widehat{\Lambda}}\mathbb{F}\AB_J=3$, so since $\mathbb{F}\AB_J$ is supported only at $\m$, it follows that $\mathbb{F}\AB_J$ is a perfect $\Lambda$-module. Since by \ref{TandSnotation} $\cE_{J}$ corresponds to $\mathbb{F}\AB_J$ across the derived equivalence, it follows that $\cE_{J}$ is perfect on $U_p$.  This in turn implies that ${i}_*\cE_{J}\in\Perf(X)$ and ${(ki)}_*\cE_J\in\Perf(\projenv)$, since we can check perfectness locally.   Since $\cE_{J}$ has a finite filtration by objects $\cO_{C_j}(-1)$ for $j\in J$, it follows that $\RDerived {i}_*\cE_{J}={i}_*\cE_{J}$ and $\RDerived {(ki)}_*\cE_{J}={(ki)}_*\cE_{J}$.  Lastly, $\mathbb{F}\AB_J$ is supported only at $\m$, so just as in \cite[6.6]{DW1} the Serre functor on $\Lambda$ acts trivially on $\mathbb{F}\AB_J$.  Across the equivalence, by uniqueness of Serre functor this means that the Serre functor on $U_p$ acts trivially on $\cE_{J}$.  This then implies \eqref{universal and canonical 1 c}, since
\[
\cE_J\otimes^{\bf L}_X\omega_X\cong\RDerived i_*\cE_{J}\otimes^{\bf L}_X\omega_X\cong\RDerived i_*(\cE_{J}\otimes^{\bf L}_{U_p}i^*\omega_X)\cong \RDerived i_*(\cE_{J}\otimes^{\bf L}_{U_p}\omega_{U_p^{\phantom{p}\!}})\cong \RDerived i_*\cE_{J}\cong \cE_J. 
\]
The proof of \eqref{universal and canonical 1 d} is identical.\\
(2) The proof is identical to \eqref{universal and canonical 1}, using \ref{pd thm for fibre} instead to give $\pd_{\widehat{\Lambda}}\AB_{\fib}=3$.
 \end{proof}

The next result describes $\Tconp^*$ and $\Tfibp^*$ on both $\Dm(\coh X)$ and $\Dm(\coh \projenv)$ as `twists' by fitting them into certain functorial triangles, and shows that they preserve the bounded derived categories.  

\begin{prop}\label{twist triangle and bounded}
With the setup as above, let $\mathcal{X}$ denote either $X$ or $\projenv$.
\begin{enumerate}
\item\label{twist triangle and bounded 1} For each $J\subseteq\{1,\hdots,n_p\}$, there exist adjunctions
\[
\begin{array}{c}
\begin{tikzpicture}[looseness=1,bend angle=15]
\node (a1) at (0,0) {$\D(\Mod \AB_J)$};
\node (a2) at (5,0) {$\D(\Qcoh \mathcal{X})$};
\draw[->] (a1) -- node[gap] {$\scriptstyle G_J$} (a2);
\draw[->,bend right] (a2) to node[above] {$\scriptstyle G_J^{\LA}$} (a1);
\draw[->,bend left] (a2) to node[below] {$\scriptstyle G_J^{\RA}$} (a1);
\end{tikzpicture}
\end{array}
\]
where $G_J:=-\otimes^{\bf L}_{\AB_J}\cE_J$. Furthermore
\begin{enumerate}[label=(\alph*), ref=\alph*]
\item\label{twist triangle and bounded 1 a} For all $x\in\D(\Qcoh \mathcal{X})$,
\begin{align*}
&\Tconp^*(x)\to x\to G^{\phantom{L}}_J\circ G_J^{\LA}(x)\to
\end{align*}
is a triangle in $\D(\Qcoh \cX)$. 
\item\label{twist triangle and bounded 1 b} For all $x\in \Dm(\coh \cX)$, $G_J^{\LA}(x)\cong \RHom_{\AB_J}(\RHom_{\cX}(-,\cE_J),\AB_J)$.
\item\label{twist triangle and bounded 1 c} $\Tconp^*$ preserves the bounded derived category $\Db(\coh \mathcal{X})$.
\end{enumerate}
\item\label{twist triangle and bounded 3} If further $U_p$ is $\mathds{Q}$-factorial then there is a version of \eqref{twist triangle and bounded 1} for $\Tfibp^*$.
\end{enumerate}
\end{prop}
\begin{proof}
(1) This follows as in \cite[(7.D), (7.E)]{DW1}, using \ref{universal and canonical} above.\\
(a) This is shown as in \cite[7.4(1)]{DW1}.\\
(b) Exactly as in the proof of \cite[(7.E)]{DW1}, $G_J^{\RA}\cong\RHom_{\cX}(\cE_J,-)$.  Further, setting 
\begin{eqnarray}
\alpha:=\left\{ \begin{array}{ll} i &\mbox{if }\cX=X\\  
ki &\mbox{if }\cX=\projenv, \end{array}\right.\label{alpha hack}
\end{eqnarray}
then \[G_J^{\LA}:=\mathbb{F}^{-1}\circ (-\otimes_\Lambda^{\bf L}\Lambda_J)\circ \RHom_{U_p}(\cV,-)\circ \alpha^*\] where $\mathbb{F}$ is the morita equivalence between $\mod\Lambda_J$ and $\mod \AB_J$ induced from \eqref{temp ME diagram}.  

Thus the functor $G_J^{\LA}$ takes $\Dm(\coh\cX)$ to $\Dm(\mod \AB_J)$, and also takes $\Db(\coh\cX)$ to $\Db(\mod\AB_J)$, since each individual functor does: $\mathbb{F}^{-1}$ and $\alpha^*$ since they are already exact, $\RHom_{U_p}(\cV,-)$ since it is a tilting equivalence \cite[3.3]{TodaUehara}, and $(-\otimes_\Lambda^{\bf L}\Lambda_J)$ since $\Lambda_J$ has finite projective dimension.

Since $\AB_J$ is a self-injective algebra, we can now use duality as in  \cite[1.3]{YZ}. Denoting restriction of scalars from $\Lambda_J$ to $\Lambda$ by $\rest$, we see
\begin{align*}
\RHom_{\AB_J}(G_J^{\LA}(-),\AB_J)&\cong
\RHom_{\cX}(-,\RDerived \alpha_*\circ(-\otimes^{\bf L}_{U_p}\cV)\circ \rest\circ\, \mathbb{F}(\AB_J))\\
&\cong 
\RHom_{\cX}(-,\RDerived \alpha_*\cE_J)\tag{by \eqref{notation summary}}\\
&\cong
\RHom_{\cX}(-,\cE_J).
\end{align*}
But now being self-injective, $\AB_J$ is clearly a dualizing complex for $\D(\Mod\AB_J)$, and so since $G_J^{\LA}(x)\in\Dm(\mod\AB_J)$ for all $x\in\Dm(\coh \cX)$, applying the dualizing functor $\RHom_{\AB_J}(-,\AB_J)$ gives the result.\\
(c) As remarked above, $G_J^{\LA}$ takes $\Db(\coh \cX)$ to $\Db(\mod\AB_J)$.  The functor $G_J$ takes the simple $\AB_J$-modules to the coherent sheaves $E_i$, hence since $\AB_J$ is a finite dimensional algebra, it follows that $G_J$ takes $\Db(\mod \AB_J)$ to $\Db(\coh\cX)$.  We conclude that the composition $G^{\phantom{L}}_J\circ G_J^{\LA}$ preserves $\Db(\coh \cX)$, and so since clearly the identity functor enjoys this property, by two-out-of-three, so does $\Tconp^*$. \\
(2) is identical to the above, using all the corresponding properties of $\cE_\fib$ and $\AB_\fib$. 
\end{proof}

\subsection{Projective Twist and Adjoints}
The first difficulty in proving that our twist functors are equivalences, exactly as in \cite[7.7]{DW1}, is to produce adjoints. For this, we use recent work of Rizzardo \cite{Alice}.  As notation, for a projective variety $Z$, denote the two projections from $Z\times Z$ by $\uppi_1$ and $\uppi_2$ respectively.

\begin{thm}[{\cite{Alice}}]\label{Alice thm}
Suppose that $Z$ is a projective variety, and that $\cP\in\D(\Qcoh Z\times Z)$ is such that both
\[
F=\RDerived \uppi_{2*}(\cP\otimes^{\bf L}_{Z\times Z}\uppi_1^*(-))\quad\mbox{and}\quad \RDerived \uppi_{1*}(\cP\otimes^{\bf L}_{Z\times Z}\uppi_2^*(-))
\]
preserve $\Db(\coh Z)$.  Then:
\begin{enumerate}
\item\label{Alice thm 1} The functor $F\colon \D(\Qcoh Z)\to \D(\Qcoh Z)$ admits both a left and right adjoint.
\item\label{Alice thm 2} The left and right adjoints of $F$ preserve $\Dm(\coh Z)$. 
\item\label{Alice thm 3} The right adjoint of $F$ furthermore preserves $\Db(\coh Z)$. 
\end{enumerate}
\end{thm}
\begin{proof}
Using the assumptions that the two functors preserve $\Db(\coh Z)$, by  \cite[5.3]{Alice} $\cP$~is both $\uppi_1$- and $\uppi_2$-perfect (for the definition of this, see \cite[2.4]{Alice}).  In particular $\cP\in\Db(\coh Z\times Z)$.  Thus \eqref{Alice thm 1} follows immediately from \cite[3.1, 4.1]{Alice}, with adjoints given by
\[
F^{\LA}:=\RDerived \uppi_{1*}(\cP^L\otimes^{\bf L}_{Z\times Z}\uppi_2^*(-))\quad\mbox{ and }\quad F^{\RA}:=\RDerived \uppi_{1*}(\cP^R\otimes^{\bf L}_{Z\times Z}\uppi_2^*(-))
\]
where $\cP^L :=\RsHom_{Z\times Z}(\cP,\uppi_1^!\cO_Z)$ and $\cP^R :=\RsHom_{Z\times Z}(\cP,\uppi_2^!\cO_Z)$.\\
(2) Since $\uppi_1,\uppi_2$ are projective and flat, it is well known and easy to see that Fourier--Mukai functors preserve $\Dm(\coh Z)$ provided that their kernels belong to $\Db(\coh Z\times Z)$.  Thus it suffices to show that $\cP,\cP^L,\cP^R\in \Db(\coh Z\times Z)$. We know that $\cP\in\Db(\coh Z\times Z)$ by the above.  

Next, recall by \cite[2.3.9]{AIL} that for any morphism of schemes $g\colon X_1\to X_2$, if $\cR$ is a perfect complex on $X_2$ then $\RHom_{X_1}(-,g^!\cR)$ takes $g$-perfect complexes to $g$-perfect complexes.  Thus $\cP^R$ is $\uppi_2$-perfect and $\cP^L$ is $\uppi_1$-perfect.  By definition, in particular this means that $\cP^R,\cP^L\in\Db(\coh Z\times Z)$.\\
(3) We know that $F$ preserves $\Db(\coh Z)$ by assumption. Since $\cP^R$ is $\uppi_2$-perfect, and the functor $F^{\RA}$ is defined by first pulling up using $\uppi_2^*$, it follows again from \cite[5.3]{Alice} that $F^{\RA}$ preserves $\Db(\coh Z)$.
\end{proof}

It is unclear in general whether both adjoints restrict to $\Db(\coh \projenv)$, and so the next series of results are mainly concerned with properties of $ \Dm(\coh \projenv)$ instead of the more usual $ \Db(\coh \projenv)$.

\begin{cor}\label{adjoint on minus}
With the global quasi-projective flops setup of \ref{flopsglobal}, and notation in \eqref{openopen}, 
\begin{enumerate}
\item\label{adjoint on minus 0} $\Tconp^*\colon \D(\Qcoh \projenv)\to \D(\Qcoh \projenv)$ has both left and right adjoints. 
\item\label{adjoint on minus 1} $\Tconp^*\colon \Dm(\coh \projenv)\to \Dm(\coh \projenv)$ has both left and right adjoints. 
\item\label{adjoint on minus 2} $\Tconp^*\colon \Db(\coh \projenv)\to \Db(\coh \projenv)$ has a right adjoint.
\end{enumerate}
In all cases, the right adjoint is $\Tconp$.  If further $U_p$ is $\mathds{Q}$-factorial then there are similar versions of \eqref{adjoint on minus 0}, \eqref{adjoint on minus 1} and \eqref{adjoint on minus 2} for $\Tfibp^*$.
\end{cor}
\begin{proof}
With the notation of \ref{defin twists at a point aux}\eqref{defin twists at a point aux 1}, write $\cQ:=\RDerived (ki\times ki)_*\cQ_J$, so there is a 
triangle 
\begin{eqnarray}
\cD_J\to \cO_{\diag,\projenv}\to \cQ\to\label{kernel display}
\end{eqnarray}
of FM kernels in $\projenv\times \projenv$.  The proof of \cite[7.6(1)]{DW1} shows that both  
\[
\RDerived \uppi_{2 *}(\cQ\otimes^{\bf L}_{\cO_{\projenv\times \projenv}}\uppi_1^{*}(-))\quad\mbox{ and }\quad 
\RDerived \uppi_{1 *}(\cQ\otimes^{\bf L}_{\cO_{\projenv\times \projenv}}\uppi_2^{*}(-))
\]
preserve $\Db(\coh \projenv)$.  Since the identity functor obviously preserves $\Db(\coh \projenv)$, by two-out-of-three it follows that both $\RDerived \uppi_{2 *}(\cD_J\otimes^{\bf L}_{\cO_{\projenv\times \projenv}}\uppi_1^{*}(-))$ and $\RDerived \uppi_{1 *}(\cD_J\otimes^{\bf L}_{\cO_{\projenv\times \projenv}}\uppi_2^{*}(-))$ preserve $\Db(\coh \projenv)$.  Statements \eqref{adjoint on minus 0}, \eqref{adjoint on minus 1} and \eqref{adjoint on minus 2} then follow immediately from \ref{Alice thm}, and further as in the proof of \ref{Alice thm} all the right adjoints are given by the kernel $\cD_J^R$, which is precisely what defines $\Tconp$.  The $\Tfibp$ version is identical.
\end{proof}

Using the above, we fit the $J$-twist into a functorial triangle, as we did for the inverse $J$-twist in \ref{twist triangle and bounded}\eqref{twist triangle and bounded 1}\eqref{twist triangle and bounded 1 a}. This will be used in the proof that the inverse $J$-twist is an equivalence in \ref{thm inverse Jtwist is equiv on X}.

\begin{cor}\label{Y funct triang}
With the global quasi-projective flops setup of \ref{flopsglobal}, and notation in \eqref{openopen}, for all $y\in\D(\Qcoh \projenv)$, there is a functorial triangle
\[
\RHom_{\projenv}(\cE_J,y)\otimes^{\bf L}_{\AB_J}\cE_J \to y\to \Tconp( y)\to.
\]
If further $U_p$ is $\mathds{Q}$-factorial, there is a similar triangle for $\Tfibp$.
\end{cor}
\begin{proof}
By the above, all kernels in \eqref{kernel display} have right adjoints given by applying $(-)^R:=\RsHom_{\projenv\times \projenv}(-,\uppi_2^!\cO_\projenv)$.  We know that $\cD^R_J$ gives $\Tconp$, and clearly the right adjoint of the identity functor is the identity.  Further by \ref{twist triangle and bounded} the FM functor given by $\cQ$ is $G^{\phantom{L}}_J\circ G_J^{\LA}$, and it has right adjoint $G^{\phantom{L}}_J\circ G_J^{\RA}$.  Since $G_J:=-\otimes^{\bf L}_{\AB_J}\cE_J$, it is clear that $G^{\phantom{L}}_J\circ G_J^{\RA}\cong \RHom_\projenv(\cE_J,-)\otimes^{\bf L}_{\AB_J}\cE_J$.
\end{proof}

\subsection{The Projective Twist is an Equivalence}\label{section proj twist is equiv}

 A formal consequence of the twist definition~\ref{defin twists at a point} is the following intertwinement lemma.

\begin{prop}\label{local-global intertwinement} 
The following diagram is naturally commutative.
\[
\begin{array}{c}
\opt{10pt}{\begin{tikzpicture}}
\opt{12pt}{\begin{tikzpicture}[scale=1.2]}
\node (a1) at (0,0) {$\D(\Qcoh U_p)$};
\node (b1) at (0,-1.5) {$\D(\Qcoh U_p)$};
\node (a2) at (2.5,0) {$\D(\Qcoh X)$};
\node (b2) at (2.5,-1.5) {$\D(\Qcoh X)$};
\node (a3) at (5,0) {$\D(\Qcoh \projenv)$};
\node (b3) at (5,-1.5) {$\D(\Qcoh \projenv)$};
\draw[->] (a1) to node[above] {$\scriptstyle \Ri_*$} (a2);
\draw[->] (b1) to node[above] {$\scriptstyle \Ri_*$} (b2);
\draw[->] (a2) to node[above] {$\scriptstyle \RDerived k_*$} (a3);
\draw[->] (b2) to node[above] {$\scriptstyle \RDerived k_*$} (b3);
\draw[->] (a1) to node[left] {$\scriptstyle \FM{\cW_{J}}$} (b1);
\draw[->] (a2) to node[right] {$\scriptstyle \Tconp^*$} (b2);
\draw[->] (a3) to node[right] {$\scriptstyle \Tconp^*$} (b3);
\end{tikzpicture}
\end{array}
\]
If $U_p$ is in addition $\mathds{Q}$-factorial, there is a similar diagram for $\Tfibp^*$.
\end{prop}
\begin{proof}
This is a standard application of Fourier--Mukai techniques, following \cite[7.8]{DW1} line for line.
\end{proof}

Recall, for an object $\cE\in\Dm(\coh \projenv)$,
\begin{align*}
\cE^\perp&:=\{ a\in\Dm(\coh \projenv)\mid \RHom_{\projenv}(\cE,a)=0 \},\\
{}^{\perp}\cE&:=\{ a\in\Dm(\coh \projenv)\mid \RHom_{\projenv}(a,\cE)=0  \}.
\end{align*}
and $\Omega\subseteq\Dm(\coh \projenv)$ is called a \emph{spanning class} if, for any $a\in\Dm(\coh \projenv)$,
\begin{enumerate}
\item $\RHom_\projenv(a,c)=0$ for all $c\in\Omega$ implies that $a=0$.
\item $\RHom_\projenv(c,a)=0$ for all $c\in\Omega$ implies that $a=0$.
\end{enumerate}

The following lemma will be used in \ref{Y not Goren} for the construction of a spanning class for $\Dm(\coh \projenv)$.

\begin{lemma}\label{Y not Goren A}
Suppose that $Z$ is a projective variety, and let $\cP\in\Perf(Z)$ with $\cP\otimes^{\bf L}_Z\omega_Z\cong \cP$.  Then for all $a\in\Dm(\coh Z)$
\[
\RHom_Z(a,\cP)=0\iff \RHom_Z(\cP,a)=0.
\]
\end{lemma}
\begin{proof}
If $h\colon Z\to\Spec\mathbb{C}$ denotes the structure morphism,  the statement follows since
\begin{align*}
\RHom_Z(a,\cP)=0&\iff \RHom_Z(a,\cP\otimes^{\bf L}_Z\omega_Z)=0\tag{by assumption}\\
&\iff \RHom_Z(\RsHom_Z(\cP,a),\omega_Z)=0\tag{$\cP$ is perfect}\\
&\iff \RHom_Z(\RsHom_Z(\cP,a),h^!\mathbb{C})=0\tag{definition of $\omega_Z$}\\
&\iff \RHom_\mathbb{C}(\RHom_Z(\cP,a),\mathbb{C})=0\tag{Grothendieck duality}\\
&\iff \RHom_Z(\cP,a)=0,
\end{align*}
where in the last step we have used \cite[1.3]{YZ}, namely that $\mathbb{C}$ is clearly a dualizing complex for $\D(\Mod\mathbb{C})$, and $\RHom_Z(\cP,a)$ has finite dimensional cohomology groups. 
\end{proof}

\begin{cor}\label{Y not Goren}
With the global quasi-projective flops setup of \ref{flopsglobal}, and notation in \eqref{openopen}, 
\begin{enumerate}
\item\label{Y not Goren 1} $\cE_J^{\perp}={}^{\perp}\cE_J^{\phantom{\perp}}\!$ in $\Dm(\coh \projenv)$.
\item\label{Y not Goren 2}  $\Omega:=\cE_{\!J}^{\phantom{\perp}}\!\!\cup \cE_J^{\perp}$ is a spanning class of $\Dm(\coh \projenv)$.
\end{enumerate}
If further $U_p$ is $\mathds{Q}$-factorial, then $\cE_\fib^{\perp}={}^{\perp}\cE_\fib^{\phantom{\perp}}\!\!$ and $\cE_\fib^{\phantom{\perp}}\!\!\cup \cE_\fib^{\perp}$ is a spanning class of $\Dm(\coh \projenv)$.
\end{cor}
\begin{proof}
By \ref{universal and canonical}\eqref{universal and canonical 1}, $\cE_J$ satisfies the two hypotheses of \ref{Y not Goren A}, so  part \eqref{Y not Goren 1} follows.  Part \eqref{Y not Goren 2} is a formal consequence of \eqref{Y not Goren 1}, see e.g.\ \cite[7.9]{DW1}.  The remaining statements follow in an identical way, since by \ref{universal and canonical}\eqref{universal and canonical 2} $\cE_\fib$ satisfies the two hypotheses of \ref{Y not Goren A}.
\end{proof}

The proof that $\Tconp^*$ and $\Tfibp^*$ are autoequivalences on $\projenv$ will use the spanning classes in \ref{Y not Goren}, and the following lemma describes the action of the functors $\Tconp^*$ and $\Tfibp^*$ on these classes.

\begin{lemma}\label{spanning class and twist functor props Y}
With the setup as above, on $\Dm(\coh \projenv)$,
\begin{enumerate}
\item\label{spanning class and twist functor props Y 1}$\Tconp^*$ sends $\cE_J\mapsto\cE_J[2]$, and is functorially isomorphic to the identity on $\cE_J^{\perp}$. 
\item\label{spanning class and twist functor props Y 2}If further $U_p$ is $\mathds{Q}$-factorial, then the same is true for $\Tfibp^*$, replacing the subscript $J$ by the subscript $0$.
\end{enumerate}
\end{lemma}
\begin{proof}
We explain only \eqref{spanning class and twist functor props Y 1}, as \eqref{spanning class and twist functor props Y 2} is similar. The last statement is a consequence of the functorial triangle in \ref{twist triangle and bounded}, using in addition $\cE_J^\perp={}^{\perp}\cE_J^{\phantom{\perp}}\!$ by \ref{Y not Goren}\eqref{Y not Goren 1}.  The other statements follow by combining \ref{local twists shifting} with \ref{local-global intertwinement}, after inverting the local twists appearing there.
\end{proof}

The following lemma gives appropriate sufficient conditions for a fully faithful functor to be an equivalence. The right adjoint version is used, as in \cite{DW1}, to show the equivalence property on $\projenv$ in~\ref{twists autoequivalence}; the left adjoint version is used on $X$ in~\ref{thm inverse Jtwist is equiv on X}.

\begin{lemma}\label{equiv trick}
Let $\cC$ be a triangulated category, and $F\colon \cC\to\cC$ an exact fully faithful functor with right adjoint $F^{\RA}$ (respectively left adjoint $F^{\LA}$).  Suppose that there exists an object $c\in\cC$ such that $F(c)\cong c[i]$ for some $i$, and further $F(x)\cong x$ for all $x\in c^\perp$ (respectively $x\in{}^\perp c$).  Then $F$ is an equivalence. 
\end{lemma}
\begin{proof}The right adjoint version is \cite[7.11]{DW1}. The argument for the left adjoint version is identical: we give it here for the convenience of the reader.

First note that, by \cite[1.24]{HuybrechtsFM}, $F$ fully faithful gives
$
F^{\LA} \circ F \overset{\sim}{\rightarrow} \Id,
$
and thence we deduce that for all $a\in {}^\perp c$, $F^{\LA}(c) \cong c$. Now by \cite[1.51]{HuybrechtsFM} $F$ is an equivalence if $F^{\LA}(x)=0$ implies $x=0$. Supposing then that $F^{\LA}(x)=0$, it suffices to show that $x \in {}^\perp c$. But
\[
\Hom_\cC(x,c[j]) \cong
\Hom_\cC(x,F(c)[j-i]) \cong
\Hom_\cC(F^{\LA}(x),c[j-i]) = 0,
\]
and so the claim is proved. 
\end{proof}

Combining the above gives the main result of this subsection.

\begin{thm}\label{twists autoequivalence} 
With the global quasi-projective flops setup of \ref{flopsglobal}, and notation in \eqref{openopen}, then for any $J\subseteq\{1,\hdots,n_p\}$,  the inverse $J$-twist
\[
\Tconp^*\colon\Db(\coh \projenv)\to\Db(\coh \projenv)
\]
is an equivalence. If furthermore $U_p$ is $\mathds{Q}$-factorial, the same is true for $\Tfibp^*$.
\end{thm}
\begin{proof}
We first establish that 
\[
F:=\Tconp^*\colon \Dm(\coh \projenv)\to \Dm(\coh \projenv)
\]
is fully faithful.  Since by \ref{adjoint on minus} it has both left and right adjoints, and further by \ref{Y not Goren}\eqref{Y not Goren 2} $\Omega:=\cE_{\!J}^{\phantom{\perp}}\!\!\cup \cE_J^{\perp}\subseteq\Dm(\coh \projenv)$ is a spanning class, by \cite[1.49]{HuybrechtsFM} we just need to show that 
\begin{eqnarray}
F\colon\Hom_{\D(\projenv)}(a,b[i])\to\Hom_{\D(\projenv)}(Fa,Fb[i])\label{bijective map}
\end{eqnarray}
is a bijection for all $a,b\in\Omega$ and all $i\in\mathbb{Z}$.  Exactly as in \cite[7.12]{DW1} this splits into four cases: the first case $a=b=\cE_J$ follows by the commutativity of \ref{local-global intertwinement}, and in the second case $a=\cE_J$, $b\in \cE_J^\perp$ it is obvious from \ref{spanning class and twist functor props Y} that both sides of \eqref{bijective map} are zero.  Similarly, in the third case $a\in\cE_J^\perp$, $b= \cE_J$ both sides of \eqref{bijective map} are zero by \ref{spanning class and twist functor props Y} and \ref{Y not Goren}\eqref{Y not Goren 1}.   The last case, namely $a,b\in\cE_J^{\perp}$ follows since $F$ is functorially isomorphic to the identity on $\cE_J^{\perp}$ by~\ref{spanning class and twist functor props Y}.

It follows that $F$ is fully faithful, and so in particular its restriction to $\Db(\coh \projenv)$ is fully faithful.  But by \ref{twist triangle and bounded}\eqref{twist triangle and bounded 1}\eqref{twist triangle and bounded 1 c} $F$ preserves $\Db(\coh \projenv)$, so
\[
F:=\Tconp^*\colon \Db(\coh \projenv)\to \Db(\coh \projenv)
\]
is fully faithful.  But also by \ref{adjoint on minus}\eqref{adjoint on minus 2} this functor has a right adjoint, and so it is an equivalence by combining \ref{equiv trick} with \ref{spanning class and twist functor props Y}.

The proof for $\Tfibp^*$ is identical, using the $\Tfibp^*$ version of all the results referenced above.
\end{proof}

\begin{remark}
Since the assumptions on the singularities in \ref{twists autoequivalence} are local, the above gives more evidence for the conjecture \cite[B.1]{HomMMP}, namely there is a flop equivalence if and only if $\cE$ is a perfect complex.  The main point of the conjecture is that we should expect flop equivalences not by some global restriction on singularities, but instead by only assuming that $\cE$ is perfect, which is a local condition.
\end{remark}

\subsection{The Quasi-Projective Twist is an Equivalence}\label{section qproj twist is equiv}
We next restrict the equivalences in the previous subsection to $X$.  Since $X$ is only quasi-projective, obtaining adjoints is tricky.  Consequently we do not use a spanning class argument as in \ref{twists autoequivalence} (which requires both left and right adjoints) to prove the functors are equivalences on $X$, instead we appeal to the theory of compactly generated triangulated categories.   The following is standard.
\begin{cor}\label{QcohY compact}
With notation as in \ref{twists autoequivalence}, 
\begin{enumerate}
\item\label{QcohY compact 1} $\Tconp^*\colon\Perf(\projenv)\to\Perf(\projenv)$ is an equivalence.
\item\label{QcohY compact 2} $\Tconp^*\colon\D(\Qcoh \projenv)\to\D(\Qcoh \projenv)$ is an equivalence.
\end{enumerate}
If furthermore $U_p$ is $\mathds{Q}$-factorial then the same is true for $\Tfibp^*$.
\end{cor}
\begin{proof}
(1) This follows immediately since perfect complexes can be intrinsically characterised inside $\Db(\coh \projenv)$ as the homologically finite complexes \cite[1.11]{Orlov}.\\
(2) We know that $\Tconp^*$ is an exact functor between compactly generated triangulated categories with infinite coproducts, which preserves coproducts (since it is a Fourier--Mukai functor) and restricts to an equivalence on the compact objects by \eqref{QcohY compact 1}.  The result is standard application of compactly generated triangulated categories (see e.g.\ \cite[3.3]{Schwede}).
\end{proof}

The following lemma, which parallels \ref{spanning class and twist functor props Y}, is used in the proof that the $J$-twist is an equivalence on $X$, and later in \ref{twist vs flopflop} and \ref{conj NC and fibre}.

\begin{lemma}\label{spanning class and twist functor props X}
With the setup as above, on $\Db(\coh X)$,
\begin{enumerate}
\item\label{spanning class and twist functor props X 1}$\Tconp^*$ sends $\cE_J\mapsto\cE_J[2]$, $E_j\mapsto E_j[2]$ for all $j\in J$, and is functorially isomorphic to the identity on ${}^{\perp}\cE_J$. 
\item\label{spanning class and twist functor props X 2}If further $U_p$ is $\mathds{Q}$-factorial, then the same is true for $\Tfibp^*$, replacing the subscripts $J$ and $j$ by the subscript $0$.
\end{enumerate}
\end{lemma}
\begin{proof}
We show \eqref{spanning class and twist functor props X 1}, as \eqref{spanning class and twist functor props X 2} is similar. The last statement follows from the functorial triangle in \ref{twist triangle and bounded}.  The other statements follow, as in \ref{spanning class and twist functor props Y}, by combining \ref{local twists shifting} with \ref{local-global intertwinement}.
\end{proof}

The following is the main result of this subsection.
\begin{thm}\label{thm inverse Jtwist is equiv on X}
With the global quasi-projective flops setup $X\to X_{\con}$ of \ref{flopsglobal}, 
\[
\Tconp^*\colon\Db(\coh X)\to\Db(\coh X)
\]
is an equivalence.  If furthermore $U_p$ is $\mathds{Q}$-factorial then the same is true for $\Tfibp^*$.
\end{thm}
\begin{proof}
Set $F:=\Tconp^*$ on $X$, and set $\projenvtwist:=\Tconp^*$ on $\projenv$.   First, $F$ preserves $\Db(\coh X)$ by \ref{twist triangle and bounded}\eqref{twist triangle and bounded 1}\eqref{twist triangle and bounded 1 c}. Then by \ref{local-global intertwinement}, for all $a,b\in\Db(\coh X)$ the diagram
\[
\begin{array}{c}
\opt{10pt}{\begin{tikzpicture}}
\opt{12pt}{\begin{tikzpicture}[scale=1.2]}
\node (a1) at (0,0) {$\Hom_{\Db(\coh X)}(a,b)$};
\node (b1) at (0,-1.5) {$\Hom_{\Db(\coh X)}(Fa,Fb)$};
\node (a2) at (5.5,0) {$\Hom_{\D(\projenv)}(\RDerived k_* a,\RDerived k_* b)$};
\node (b2) at (5.5,-1.5) {$\Hom_{\D(\projenv)}(\projenvtwist\RDerived k_* a,\projenvtwist\RDerived k_* b)$};
\draw[->] (a1) to node[above] {$\scriptstyle \RDerived k_*$} node[below] {$\scriptstyle \sim$} (a2);
\draw[->] (b1) to node[above] {$\scriptstyle \RDerived k_*$}  node[below] {$\scriptstyle \sim$} (b2);
\draw[->] (a1) to node[left] {$\scriptstyle F$} (b1);
\draw[->] (a2) to node[right] {$\scriptstyle \projenvtwist$} (b2);
\end{tikzpicture}
\end{array}
\]
commutes, where the upper and lower morphisms are isomorphisms since $k$ is an open immersion, and the right-hand morphism is an isomorphism by \ref{QcohY compact}.  It follows that the left-hand morphism is an isomorphism, and so $F$ is fully faithful.

To show that $F$ is an equivalence, by \ref{equiv trick} and \ref{spanning class and twist functor props X} it suffices to show that $F$ possesses a left adjoint that preserves $\Db(\coh X)$.  But since $\RDerived k_*$ is fully faithful, $k^!\RDerived k_*\cong \Id$, so by~\ref{local-global intertwinement}  
\[
F\cong k^!\circ \projenvtwist\circ\RDerived k_*.
\]
Now $\projenvtwist$ is an equivalence by \ref{twists autoequivalence}, so its left adjoint coincides with its right adjoint, which is $\Tconp$ by \ref{adjoint on minus}.  Thus $F$ has a left adjoint on $\D(\Qcoh X)$, given by
\begin{eqnarray}
F^{\LA}\cong k^*\circ \Tconp \circ\RDerived k_*.\label{temp formula for RA}
\end{eqnarray}
We claim that $F^{\LA}$ preserves $\Db(\coh X)$, so let $x\in\Db(\coh X)$.  Applying the functorial triangle in \ref{Y funct triang} to $\RDerived k_*x$, then applying $k^*$,  results in the triangle
\begin{eqnarray}
k^*(\RHom_{\projenv}(\cE_J,\RDerived k_*x)\otimes^{\bf L}_{\AB_J}\cE_J)\to x\to F^{\LA}(x)\to,\label{key for DB}
\end{eqnarray}
where we have used $k^*\RDerived k_*\cong \Id$ since $\RDerived k_*$ is fully faithful, and also \eqref{temp formula for RA}.  
But
\[
\RHom_{\projenv}(\cE_J,\RDerived k_*x)\otimes^{\bf L}_{\AB_J}\cE_J
\cong 
\RDerived k_*(\RHom_{X}(\cE_J,x)\otimes^{\bf L}_{\AB_J}\cE_J),
\]
so again using $k^*\RDerived k_*\cong \Id$ we see that the left-hand term in \eqref{key for DB} is isomorphic to $\RHom_X(\cE_J,x)\otimes^{\bf L}_{\AB_J}\cE_J$, which clearly belongs to $\Db(\coh X)$.  By two-out-of-three applied to \eqref{key for DB}, we deduce that $F^{\LA}$ preserves $\Db(\coh X)$, and so we are done.
\end{proof}

Since $\Tconp^*\colon\Db(\coh X)\to\Db(\coh X)$ is an equivalence, we define $\Tconp$ to be its right adjoint, and likewise for $\Tcon$.  Similarly if $U_p$ is $\mathds{Q}$-factorial, $\Tfibp^*$ is an equivalence on $X$, so we define $\Tfibp$ to be its right adjoint.  The following is a completely standard application of adjoint functors of equivalences, together with \eqref{key for DB}.

\begin{cor}\label{twists autoequivalence B}
With the global quasi-projective flops setup $X\to X_{\con}$ of \ref{flopsglobal}, 
\begin{enumerate}
\item\label{twists autoequivalence B 1} For any $J\subseteq\{1,\hdots,n_p\}$, the $J$-twist
\[
\Tconp\colon\Db(\coh X)\to\Db(\coh X)
\]
is an equivalence. If $U_p$ is $\mathds{Q}$-factorial then the same is true for $\Tfibp$.
\item\label{twists autoequivalence B 2} When the twists are defined, there are functorial triangles
\begin{align*}
&\RHom_X(\cE_J,x)\otimes^{\bf L}_{\AB_J}\cE_J \to x\to \Tconp(x)\to\\
&\RHom_X(\cE_{\fib},x)\otimes^{\bf L}_{\CAR}\cE_{\fib}\to x\to \Tfibp(x)\to.
\end{align*}
\item\label{twists autoequivalence B 3} For a general $J$ as in \ref{twists definition}, $\Tcon$ is an equivalence, and if $X$ is $\mathds{Q}$-factorial then so is $\Tfib$.
\end{enumerate}
\end{cor}

The following result is used in the next subsection to compare autoequivalences.

\begin{cor}\label{cor Jtwist and pushdown} The $J$-twist $\Tconp$ commutes with the pushdown $\RDerived f_*$, i.e.
\[\RDerived f_* \circ \Tconp \cong \RDerived f_*.\]
\end{cor}
\begin{proof}This is shown by the method of \cite[7.15]{DW1}. The key point is that $\cE_J$ is filtered by the $E_j$, and so $\RDerived f_* \cE_J = 0$.
\end{proof}

\section{The Autoequivalence Group}
\label{section auto group}

In this final section we complete our study of the $J$-twist and fibre twist autoequivalences defined in \S\ref{ZLT and GT section}. We show first that the $J$-twist gives an intrinsic characterisation of the inverse of the flop--flop functor, when the latter exists. This result is used to produce, in \ref{global main result action}, a group action on the derived category in the general case where individual curves are not necessarily floppable. Finally in \S\ref{Section6.2} we compare the two twists, and characterise circumstances in which they are conjugate.

\subsection{$J$-Twists and Group Actions}\label{Section6.1}

The input to this subsection is a flopping contraction in the global quasi-projective flops setup $f\colon X\to X_{\con}$ of \ref{flopsglobal}. We will sometimes restrict the singularities of $X$ to be Gorenstein terminal, but we will only do this if we need to ensure that the flop functor is an equivalence \cite{Chen}.

\begin{prop}\label{twist vs flopflop}
With the global quasi-projective flops setup of \ref{flopsglobal}, suppose that $X$ has only Gorenstein terminal singularities and a subset $J$ of the exceptional curves flops algebraically (e.g.\ when $J$ equals all exceptional curves).  Then 
 $\Tcon \circ (\flop_J\circ\flop_J) \functcong \Id$.
\end{prop}
\begin{proof} 
Since the curves in $J$ flop algebraically, there exists some factorisation of $f$ into 
\[
X\xrightarrow{g} X_J\to X_{\con}
\]
where $g$ is a flopping contraction. First, note that $\{ E_j\mid j\in J\}$ generates
\[
\cC_g:=\{ c\in\Db(\coh X)\mid \RDerived g_*c=0\},
\]
since $\cC_g$ splits as a direct sum over the points $p$ of the non-isomorphism locus of $g$, so we may follow the argument of \cite[5.3]{KIWY}.

We next argue that $\TwistVsFlop := \Tcon \circ (\flop_J\circ \flop_J)$  preserves $\Per (X,X_J)$, since the remainder of the proof then follows exactly as in \cite[7.18]{DW1}.  For this,  note that 
\begin{eqnarray}
\Tcon (E_j) = E_j[-2]\label{J action on simples}
\end{eqnarray}
for all $j\in J$, since
\[
\begin{array}{ll}
\Tconp (E_j) = E_j[-2] &\mbox{ if }C_j\in g^{-1}(p)\\
\Tconp (E_j) = E_j &\mbox{ if }C_j\notin g^{-1}(p)
\end{array}
\]
where the top line holds by \ref{spanning class and twist functor props X}, and the bottom by \ref{twists autoequivalence B}\eqref{twists autoequivalence B 2} since if $C_j\notin g^{-1}(p)$, then $E_j\in ({i_p}_*\cE_{J_p})^\perp$. On the other hand
\begin{eqnarray}
(\flop_J\circ\flop_J)(E_j) = E_j[2]\label{FF action on simples}
\end{eqnarray}
for all $j\in J$, using the description in \cite[\S3(i)]{TodaResPub} of the action of the flop on the sheaves $E_j=\cO_{C_j}(-1)$, noting a correction given in \cite[Appendix~B]{TodaGV} to the sign of the shift in \cite{Bridgeland}.   Combining \eqref{J action on simples} and \eqref{FF action on simples}, we find that $\TwistVsFlop$ fixes the set $\{ E_j\mid j\in J\}$ and thus preserves $\cC_g$.

Next we note that $\TwistVsFlop$ commutes with the pushdown $\RDerived g_*$. This follows because $\Tconp$ commutes with the pushdown by \ref{cor Jtwist and pushdown}, and the flop--flop commutes with the pushdown exactly as in \cite[7.16(1)]{DW1}.  It then follows by the argument of \cite[7.17]{DW1} that $\TwistVsFlop$ preserves $\Per (X,X_J)$. 

Finally, by combining as in \cite[7.16(2)]{DW1}, we have that $\TwistVsFlop (\cO_X) \cong \cO_X$. Using this, and the fact that $\TwistVsFlop$ commutes with the pushdown $\RDerived g_*$ as noted above, the proof now follows just as in \cite[7.18]{DW1}.
\end{proof}

We next give a group action for any algebraic flopping contraction in the global setup of \ref{flopsglobal}.

\begin{prop}\label{sky off curves}
Under the Zariski local setup \ref{flopslocal}, or the global setup of \ref{flopsglobal}, 
\begin{enumerate}
\item\label{sky off curves 1} $\Tcon(\cO_x)\cong\cO_x$ for all $x\notin \bigcup_{j\in J} C_j$. 
\item\label{sky off curves 2} Choose $p\in \noniso f$ and a collection $J_1,\hdots,J_m$ of subsets of $\{1,\hdots,n_{p}\}$.  If the twists $\Tconpip{1},\hdots,\Tconpip{m}$ satisfy some relation on the formal fibre above $p$, then they satisfy the same relation Zariski locally.
\end{enumerate}
\end{prop}
\begin{proof}
We give the proof in the Zariski local setup  \ref{flopslocal}, since the proof in the global setup is obtained by simply replacing $U$ by $X$ throughout.\\
(1) Since $\cE_{J}$ is filtered by $\cO_{C_j}(-1)$ with $j\in J$, it follows that $\RHom_U(\cE_{J},\cO_x)=0$ for all $x\notin \bigcup_{j\in J}C_j$, since $x\notin\Supp\cE_{J}$ \cite[5.3]{BM}. Thus, the triangle
\[
\RHom_U(\cE_J,\cO_x)\otimes^{\bf L}_{\AB_J}\cE_J\to\cO_x\to\Tcon(\cO_x)\to
\]
implies that $\Tcon(\cO_x)\cong\cO_x$.\\
(2) Suppose that the relation on the formal fibre above $p$ can be written as
\begin{eqnarray}
(\Tconpip{i_1})^{a_1}\hdots(\Tconpip{i_\ell})^{a_\ell} \functcong\Id\label{word relation}
\end{eqnarray}
for some $a_1,\hdots,a_\ell\in\mathbb{Z}$ and some $i_1,\hdots,i_\ell\in\{1,\hdots,m\}$.  We again track skyscrapers.   By~\eqref{sky off curves 1}, certainly the skyscrapers not supported on $\bigcup_{i=1}^{n_p}C_i$ are fixed under the left-hand side of \eqref{word relation}, so we need only track the skyscrapers on $\bigcup_{i=1}^{n_p}C_i$.  As in the latter stages of the proof of \ref{braid 2 curve local}, we can do this by passing to the formal fibre, and by assumption  we know that the relation \eqref{word relation} holds there.  Hence on the formal fibre these skyscrapers are fixed under the left-hand side of \eqref{word relation}.  Hence overall every skyscraper $\cO_x$ gets sent to some skyscraper under the left-hand side of \eqref{word relation}, so since the twists commute with pushdown, as before it follows that the relation holds Zariski locally.  
\end{proof}

Combining \ref{commute remark} with \ref{sky off curves}\eqref{sky off curves 2}  shows that globally we can still view the $J$-twists as a subgroup of $\fundgp(\mathds{G}_{\cH})$. Since the $J$-twists are equivalences by \ref{twists autoequivalence}, this then immediately gives the following, which is our main result.

\begin{cor}\label{global main result action}
With the global quasi-projective flops setup $X\to X_{\con}$ of \ref{flopsglobal},  the subgroup $K$ of \,$\fundgp(\mathds{G}_{\cH})$ generated by the $J$-twists, as $J$ ranges over all subsets of curves, acts on $\Db(\coh X)$.
\end{cor}

We remark that the subgroup $K$ can equal $\fundgp(\mathds{G}_{\cH})$, as the following example illustrates.  As stated in the introduction, it is unclear in what level of generality this holds, and indeed this seems to be an interesting problem, both geometrically and group-theoretically.
\begin{example}\label{whole pure braid group}
Consider an algebraic flopping contraction $X\to X_{\con}$ of two intersecting curves contracting to a $cA_n$ singularity.  In this situation the chamber structure and subgroup $K$ are illustrated as follows:
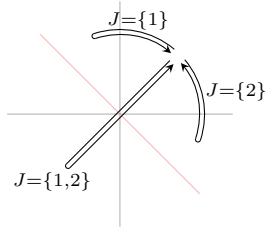
\[
\begin{array}{c}
\begin{tikzpicture}[scale=0.75,>=stealth, bend right, bend angle=25,looseness=1.5]
\coordinate (A1) at (135:2cm);
\coordinate (A2) at (-45:2cm);
\coordinate (B1) at (153.435:2cm);
\coordinate (B2) at (-26.565:2cm);
\coordinate (C1) at (161.565:2cm);
\coordinate (C2) at (-18.435:2cm);
\draw[red!30] (A1) -- (A2);
\draw[black!30] (-2,0)--(2,0);
\draw[black!30] (0,-2)--(0,2);
\draw ([shift=(50:1.5cm)]0,0) arc (50:108:1.5cm);
\draw[->] ([shift=(108:1.41cm)]0,0) arc (108:50:1.41cm);
\draw (108:1.5cm) arc (108:288:0.045cm);
\node at (79:1.7cm) {$\scriptstyle J=\{1\}$};
\draw ([shift=(40:1.5cm)]0,0) arc (40:-18:1.5cm);
\draw[->] ([shift=(-18:1.41cm)]0,0) arc (-18:40:1.41cm);
\draw (-18:1.5cm) arc (-18:-198:0.045cm);
\node at (11:2.1cm) {$\scriptstyle J=\{2\}$};
\draw  (47:1.3cm) -- (223:1.3cm);
\draw[->] (227:1.3cm) -- (43:1.3cm);
\node at (225:1.7cm) {$\scriptstyle J=\{1,2\}$};
\draw ([shift=(315:0.045cm)]225:1.3cm) arc (315:135:0.045cm);
\end{tikzpicture}
\end{array}
\]
Thus inside $\fundgp(\mathds{G}_{\cH})$, which is the pure braid group on three strands, $ K=\langle a^2, b^2, (aba)^2\rangle$ where $a$ and $b$ are the standard braids in the classical presentation of the braid group.  Now $f:=a^2ba^2b=ababab=(aba)^2\in K,$ and hence also $a^{-2}fb^{-2} = ba^2b^{-1} \in K$. But it is well-known that the pure braid group is generated by $a^2$, $b^2$, and $ba^2b^{-1}$ (see e.g.\ \cite[p5]{Brendle}), and hence $K=\fundgp(\mathds{G}_{\cH})$ in this case.
\end{example}

Similarly, it is possible to show that $K=\fundgp(\mathds{G}_{\cH})$  for the $A_3$ hyperplane arrangement.

\subsection{Conjugacy of Twists}\label{Section6.2}
It follows immediately from \ref{twist vs flopflop} that $\Tcon$ lies in the image of the homomorphism in \ref{flops global intro}. This cannot hold for $\Tfib$, as it does not commute with the pushdown along the contraction $\Rf_*$. Nevertheless we might ask whether $\Tfib$ lies in the subgroup generated by this image and twists by line bundles, and furthermore whether the two twists are ever conjugate. In the simplest case, for a smooth rational curve of type~$A$, we find that the two twists may indeed be conjugate by a certain line bundle. When such a curve is not of type~$A$, however, we show that the two twists are never related in this way. As stated in the introduction, the functor $\Tfib$ is not expected to be related to the flop--flop functor in general, instead it is expected to be the affine element in some (pure) braid action. Because of the intricacy of the combinatorics, we will return to this more general affine action in the future.

\begin{thm}\label{conj NC and fibre}
In the global quasi-projective flops setup of \ref{flopsglobal}, for a contraction of a single irreducible curve to a point $p$, where in addition $U$ is $\mathds{Q}$-factorial, there exists a functorial isomorphism
\[
\Tfib(x \otimes \conjlb) \cong \Tcon(x) \otimes \conjlb
\]
for some line bundle $\conjlb$ on $X$, if and only if the following conditions hold.
\begin{enumerate}
\item\label{conj NC and fibre 1} The point $p$ is cDV of Type $A$.
\item\label{conj NC and fibre 2} There exists a line bundle $\conjlb$ on $X$ such that $\deg(\conjlb|_{f^{-1}(p)})=-1$.
\end{enumerate}
\end{thm}
\begin{proof}By assumption $J=\{1\}$, corresponding to the single irreducible curve in the fibre $C = f^{-1}(p)$.

($\Leftarrow$) We require a natural isomorphism of functors
\begin{equation}\label{conjugation natural iso}
\Tfib \functcong (- \otimes \conjlb)\circ \Tcon \circ (- \otimes \conjlb^{-1}) =: \Tconconj.
\end{equation}
Observe that, using \ref{twists autoequivalence B}\eqref{twists autoequivalence B 2}, there are functorial triangles for all $x\in\D(\Qcoh X)$,
\begin{equation*}
\def\arraystretch{1.2}
\begin{array}{c@{\,\,}c@{\,}c@{\,\,}c@{\,\,}c@{\,\,}c@{\,\,}c@{\,\,}c@{}l}
\RHom_X(\cE_{\fib},x) &\otimes^{\bf L}_{\AB_{\fib}}& \cE_{\fib} & \to & x & \to & \Tfib(x) & \to &\\
\RHom_X(\cE_J\otimes\conjlb,x) &\otimes^{\bf L}_{\AB_J} & (\cE_J\otimes\conjlb) & \to & x & \to & \Tconconj(x) & \to &.
\end{array}
\end{equation*}
By the arguments of \cite[\S 3]{DW1}, see also \cite{DW2}, the objects $\cE_{\fib}$ and $\cE_J$ are universal families, corresponding to the representing objects $\AB_{\fib}$ and $\AB_J$ respectively, for the noncommutative deformations of
\[
E_0=\omega_{C}[1]
\quad\text{and}\quad
E_1=\cO_{C^{\redu}}(-1).
\]
Under our assumption \eqref{conj NC and fibre 1}, we have that $C=C^{\redu}\cong\mathbb{P}^1$, and so $\omega_C \cong \cO_{\mathbb{P}^1}(-2)$, and thence using~\eqref{conj NC and fibre 2}
\[
E_0 \cong E_1 \otimes \conjlb[1].
\] The functor~$(- \otimes \conjlb)[1]$ is an equivalence, and so we obtain an isomorphism of representing objects $\AB_{\fib}\cong \AB_J$, and an isomorphism of universal families
\begin{equation}
\cE_{\fib} \cong \cE_J \otimes \conjlb[1],
\end{equation}
respecting the module structures over the algebras $\AB_{\fib}$ and $\AB_J$. We thus obtain isomorphisms between the objects in the functorial triangles above, and the natural isomorphism~\eqref{conjugation natural iso} follows.

($\Rightarrow$) Suppose as in the statement that the twists are conjugate by some line bundle~$\conjlb$. It will be more convenient to work with perverse sheaves in the category $\PerOne$, and so we conjugate by the dualizing functor $\dualizing$ to find
\[
(\Tcondual) \circ ( - \otimes \conjlb) \functcong ( - \otimes \conjlb) \circ ( \Tfibdual).
 \]
Recall that $\dualizing$ exchanges $E_0 = \omega_C[1]$ and $\cO_C$ \cite[3.5.8]{VdB1d}, so that by \ref{spanning class and twist functor props X} we have
\[
(\Tfibdual) (\cO_C) \cong \cO_C [2].
\]
It follows immediately that
\[
(\Tcondual) (\cO_C \otimes \conjlb)  \cong \cO_C \otimes \conjlb[2].
\]

The functor $\Tcondual$ commutes with the pushdown $\Rf_*$, by the argument of~\ref{cor Jtwist and pushdown}. The key point there was that $\Rf_* \cE_J = 0$, whereas here we use $\Rf_* (\dualizing\,\cE_J) = 0$, which is obtained from relative Serre duality, as stated for instance in \cite[6.17]{DW1}. We then deduce that $\Rf_*(\cO_{C} \otimes \conjlb )=0$ by boundedness. As in \S\ref{completelocalnotation} we may work complete locally around $p$, on a formal fibre~$\mathfrak{U}$. It follows by base change that 
\begin{equation}\label{necessary condition for intertwinement}
\Rf_*(\cO_{C} \otimes \conjlb|_{\mathfrak{U}} )=0.
\end{equation}
We now analyse this pushdown via the tilting equivalence on $\mathfrak{U}$, and use its vanishing to control the degree of $\cF$ on the contracted curve $C$.

From \S\ref{completelocalnotation}, we have that $\cO_{\mathfrak{U}}\oplus\cN_1$ is a tilting bundle on $\mathfrak{U}$, so its dual $\cW:=\cO_{\mathfrak{U}}\oplus\cM_1$ is also a tilting bundle, and so gives an equivalence 
\[
\begin{array}{c}
\begin{tikzpicture}
\node (a1) at (0,0) {$\Db(\coh\mathfrak{U})$};
\node (a2) at (5,0) {$ \Db(\mod\AB^{\op}).$};
\draw[->,transform canvas={yshift=+1ex}] (a1) -- node[above] {$\scriptstyle\Upphi=\RHom_{\mathfrak{U}}(\cW,-)$} node [below] {$\scriptstyle\sim$} (a2);
\draw[->,transform canvas={yshift=-1ex}] (a2) -- node[below] {$\scriptstyle\Uppsi$} (a1);
\end{tikzpicture}
\end{array}
\]
We will place the natural surjection $\cO_{C} \to \cO_{C^{\redu}}$ into a distinguished triangle, by passing through the equivalence~$\Upphi$.

Recall that there is an exact sequence
\[
0\to\cO_{\mathfrak{U}}^{\oplus(r-1)}\to\cM_1\to\cL\to 0
\]
where $\cL$ is a line bundle having degree $1$ on $C^{\redu}$, and $r \geq 1$, with $r=1$ occurring precisely when $p$ is cDV of Type $A$. We see immediately from the standard calculation of the cohomology of line bundles on $C^{\redu} \cong \mathbb{P}^1$ that $\RHom_{\mathfrak{U}}(\cL,\cO_{C^{\redu}})=0$, and we thence obtain that
\begin{align*}
\RHom_{\mathfrak{U}}(\cO_{\mathfrak{U}},\cO_{C^{\redu}}) &= \K 
\\
\RHom_{\mathfrak{U}}(\cM_1,\cO_{C^{\redu}}) &= \K^{r-1}. 
\end{align*}
In particular, $\Upphi(\cO_{C^{\redu}})$ is a module in degree zero, which when viewed as a quiver representation has $\K$ at the vertex $0$, and $\K^{r-1}$ at the vertex $1$.

\newcommand\remainder{\cG}
\newcommand\kernelSheaf{\mathcal{H}}

Since $\Upphi$ is an equivalence, $\Upphi(\cO_{C} \to \cO_{C^{\redu}})$ is a non-zero morphism.  But $\Upphi(\cO_{C})=S_0$, which is just $\K$ at the vertex $0$, and so this is simply a non-zero morphism between representations.  Since $\Upphi(\cO_{C})$ is one-dimensional, necessarily this is injective, so we set $\remainder$ to be the cokernel and so obtain an exact sequence
\[
0\to \Upphi(\cO_{C}) \to \Upphi(\cO_{C^{\redu}})\to \remainder\to 0.
\]
By the above, $\remainder$ is a module with a filtration consisting of $r-1$ copies of $S_1$. Applying $\Uppsi$ we get a triangle
\begin{eqnarray}\label{equation.kernel_triangle}
\kernelSheaf := \Uppsi(\remainder)[-1] \to \cO_{C} \to \cO_{C^{\redu}} \to \label{triangle for conj result}
\end{eqnarray}
where $\kernelSheaf$ is a sheaf with a filtration consisting of $r-1$ copies of $\Uppsi(S_1)[-1] \cong \cO_{C^{\redu}}(-1)$. 

 Tensoring \eqref{equation.kernel_triangle} by~$\conjlb|_{\mathfrak{U}}$, and writing $d=\deg(\conjlb|_{C^{\redu}})$ for convenience, we have a short exact sequence
\[
0\to \kernelSheaf(d) \to \cO_{C} \otimes \conjlb|_{\mathfrak{U}} \to \cO_{C^{\redu}}(d)\to 0
\]
where as above $\kernelSheaf$ is filtered by objects $\cO_{C^{\redu}}(-1)$. Applying $\Rf_*$, and using the cohomology vanishing~\eqref{necessary condition for intertwinement}, we find
\[
\RDerived{f_*} (\cO_{C^{\redu}}(d)) \cong \RDerived{f_*} (\kernelSheaf(d))[1].
\]
We now make further use of the cohomology of line bundles on $C^{\redu} \cong \mathbb{P}^1$ to deduce our result. The fibres of $f$ are at most one-dimensional, so $\RDerivedi{>1}{f_*} (\kernelSheaf(d)) = 0$, and therefore $\RDerivedi{1}{f_*} (\cO_{C^{\redu}}(d)) = 0$ which implies $d\geq-1$. Suppose, for a contradiction, that furthermore $d\geq 0$. Then, using the filtration of $\kernelSheaf$, $\RDerivedi{>0}{f_*} (\kernelSheaf(d)) = 0$ whereas $f_* (\cO_{C^{\redu}}(d)) \neq 0$. It follows that $d=-1$, yielding \eqref{conj NC and fibre 2}. We then have $\Rf_*(\kernelSheaf(-1)) = 0$, forcing $\kernelSheaf=0$, so that $r=1$, giving \eqref{conj NC and fibre 1}.
\end{proof}


\begin{thebibliography}{KIWY}


\bibitem[AL1]{AL1}
R.~Anno and T.~Logvinenko, \emph{Orthogonally spherical objects and spherical fibrations}, {\sf arXiv:1011.0707}.

\bibitem[AL2]{AL2}
R.~Anno and T.~Logvinenko, \emph{Spherical DG functors},  {\sf arXiv:1309.5035}.

\bibitem[A]{Arvola}
W.~A.~Arvola, \emph{The fundamental group of the complement of an arrangement of complex hyperplanes}, Topology \textbf{31} (1992), no.~4, 757--765.

\bibitem[AIL]{AIL}
L.~Avramov, S.~B.~Iyengar and J.~Lipman, \emph{Reflexivity and rigidity for complexes, II: Schemes}, Algebra Number Theory \textbf{5} (2011), no.~3, 379--429.

\bibitem[BM]{BayerMacri}
A.~Bayer and E.~Macr\`i, \emph{MMP for moduli of sheaves on K3s via wall-crossing: nef and movable cones, Lagrangian fibrations}, Invent.\ Math.\ \textbf{198} (2014), no.~3, 505--590.

\bibitem[BB]{Brendle}
J.~Birman and T.~Brendle, \emph{Braids: a survey}, Handbook of knot theory, 19--103, Elsevier B.\ V., Amsterdam, 2005. 

\bibitem[B02]{Bridgeland}
T.~Bridgeland, \emph{Flops and derived categories}, Invent. Math. \textbf{147} (2002), no.~3, 613--632.

\bibitem[B09]{BridgelandKleinian}
T.~Bridgeland, \emph{Stability conditions and Kleinian singularities}, 
Int.\ Math.\ Res.\ Not.\ 2009, no.~21, 4142--4157. 

\bibitem[BM]{BM}
T.~Bridgeland and A.~Maciocia, \emph{Fourier--Mukai transforms for K3 and elliptic fibrations}, J.~Algebraic Geom.\ \textbf{11} (2002), no.~4, 629--657.

\bibitem[BO]{BO}
A.~Bondal and D.~Orlov, \emph{Derived categories of coherent sheaves}, Proceedings of the International Congress of Mathematicians, Vol.\ II (Beijing, 2002), 47--56, Higher Ed. Press, Beijing, 2002. 

\bibitem[BH]{BH}
R.-O.~Buchweitz and L.~Hille, \emph{Hochschild (co-)homology of schemes with tilting object}. Trans.\ Amer.\ Math.\ Soc.\ \textbf{365} (2013), no.~6, 2823--2844.

\bibitem[C14]{Calabrese}
J.~Calabrese, \emph{Donaldson--Thomas Invariants and Flops}, to appear J.\ Reine\ Angew.\ Math.\ (Crelle), DOI: 10.1515/crelle-2014-0017.

\bibitem[C02]{Chen}
J.-C.~Chen, \emph{Flops and equivalences of derived categories for threefolds with only terminal Gorenstein singularities}, J. Differential Geom. \textbf{61} (2002), no. 2, 227--261.

\bibitem[CM]{CM}
R.~Cordovil and M.~L.~Moreira, \emph{A homotopy theorem on oriented matroids}, Graph theory and combinatorics (Marseille-Luminy, 1990). Discrete Math.\ \textbf{111} (1993), no.~1--3, 131--136. 

\bibitem[CI]{CI}
A.~Craw and A.~Ishii, \emph{Flops of $G$-Hilb and equivalences of derived categories by variation of GIT quotient}. Duke Math.\ J.\ \textbf{124} (2004), no.\ 2, 259--307. 



\bibitem[D10]{DaoNCCR}
H.~Dao, \emph{Remarks on non-commutative crepant resolutions of complete intersections},  Adv.\ Math.\ \textbf{224} (2010), no.\ 3, 1021--1030.

\bibitem[D72]{Deligne}
P. Deligne, \emph{Les immeubles des groupes de tresses g\'en\'eralis\'es}, Invent.\ Math.\ \textbf{17} (1972), 273--302.

\bibitem[D13]{Donovan}
W.~Donovan, \emph{Grassmannian twists on the derived category via spherical functors}, Proc.\ Lond.\ Math.\ Soc.\ (5) \textbf{107} (2013), 1053--1090.

\bibitem[DS1]{DonSeg1}
W.~Donovan and E.~Segal, \emph{Window shifts, flop equivalences and Grassmannian twists}, Compositio Mathematica (6) \textbf{150} (2014), 942--978.

\bibitem[DS2]{DonSeg2}
W.~Donovan and E.~Segal, \emph{Mixed braid group actions from deformations of surface singularities}, Comm.\ Math.\ Phys. (1) \textbf{335} (2014), 497--543.

\bibitem[DW1]{DW1}
W.~Donovan and M.~Wemyss, \emph{Noncommutative deformations and flops}, \textsf{arXiv:1309.0698}, to appear Duke Math.\ J.

\bibitem[DW2]{DW2}
W.~Donovan and M.~Wemyss, \emph{Contractions and deformations}, in preparation.

\bibitem[GN]{GN}
S.~Goto and K.~Nishida, \emph{Finite modules of finite injective dimension over a Noetherian algebra}, J.\ Lond.\ Math.\ Soc.\ (2) \textbf{63} (2001), 319--335.


\bibitem[HLS]{HLS}
D.~Halpern-Leistner and I.~Shipman, \emph{Autoequivalences of derived categories via geometric invariant theory}, {\sf arXiv:1303.5531}.

\bibitem[H]{HuybrechtsFM}
D.~Huybrechts, \emph{Fourier-Mukai transforms in algebraic geometry}. Oxford Mathematical Monographs. The Clarendon Press, Oxford University Press, Oxford, 2006. viii+307 pp.


\bibitem[KIWY]{KIWY}
M.~Kalck, O.~Iyama, M.~Wemyss and D.~Yang, \emph{Frobenius categories, Gorenstein algebras and rational surface singularities},  Compos.\ Math.\ \textbf{151} (2015), no.~3, 502--534.

\bibitem[IR]{IR}
O.~Iyama and I.~Reiten, \emph{Fomin--Zelevinsky mutation and tilting modules over Calabi-Yau algebras}, Amer.\ J.\ Math.\ \textbf{130} (2008), no.\ 4, 1087--1149.


\bibitem[IW1]{IW4}
O.~Iyama and M.~Wemyss, \emph{Maximal modifications and Auslander--Reiten duality for non-isolated singularities}, Invent.\ Math.\ \textbf{197} (2014), no.~3, 521--586.


\bibitem[IW2]{IW5}
O.~Iyama and M.~Wemyss, \emph{Singular derived categories of $\mathds{Q}$-factorial terminalizations and maximal modification algebras},  Adv.\ Math.\ \textbf{261} (2014), 85--121.

\bibitem[K14]{Joe}
J.~Karmazyn, \emph{Quiver GIT for varieties with tilting bundles}, {\sf arXiv:1407.5005}.

\bibitem[K91]{Katz}
S.~Katz, \emph{Small resolutions of Gorenstein threefold singularities}, Algebraic geometry: Sundance 1988, 61--70, Contemp.\ Math., \textbf{116}, Amer.\ Math.\ Soc., Providence, RI, 1991.

\bibitem[Ko1]{KollarFlops}
J.~Koll\'ar, \emph{Flops}. Nagoya Math.\ J.\ \textbf{113} (1989), 15--36.


\bibitem[LW]{LW}
G.~Leuschke and R.~Wiegand, \emph{Cohen--Macaulay representations}, Mathematical Surveys and Monographs, \textbf{181}. American Mathematical Society, Providence, RI, 2012. xviii+367 pp.


\bibitem[N]{Nagao}
K.~Nagao, \emph{Derived categories of small toric Calabi-Yau 3-folds and curve counting invariants}, Q.\ J.\ Math. \textbf{63} (2012), no.~4, 965--1007.

\bibitem[O]{Orlov}
D.~Orlov, \emph{Triangulated categories of singularities and equivalences between Landau-Ginzburg
Models}, Sb.\ Math.\ \textbf{197} (2006), no.~11--12, 1827--1840.

\bibitem[P93]{Paris}
L.~Paris, \emph{Universal cover of Salvetti's complex and topology of simplicial arrangements of hyperplanes}, Trans.\ Amer.\ Math.\ Soc.\ \textbf{340} (1993), no.~1, 149--178.

\bibitem[P00]{Paris2}
L.~Paris, \emph{On the fundamental group of the complement of a complex hyperplane arrangement}. Arrangements---Tokyo 1998, 257--272, Adv.\ Stud.\ Pure Math., \textbf{27}, Kinokuniya, Tokyo, 2000. 

\bibitem[P83]{Pinkham}
H.~Pinkham, \emph{Factorization of birational maps in dimension 3}, Singularities (P. Orlik, ed.), Proc. Symp. Pure Math., vol. 40, Part 2, 343--371, American Mathematical Society, Providence, 1983.

\bibitem[R69]{Ramras}
M.~Ramras, \emph{Maximal orders over regular local rings of dimension two}, Trans. Amer. Math. Soc. \textbf{142} (1969), 457--479.

\bibitem[R82]{Randell}
R.~Randell, \emph{The fundamental group of the complement of a union of complex hyperplanes}, Invent.\ Math.\ \textbf{69} (1982), no.~1, 103--108.

\bibitem[R83]{Pagoda}
M.~Reid, \emph{Minimal models of canonical 3-folds}, Algebraic varieties and analytic varieties (Tokyo, 1981), 131--180, Adv. Stud. Pure Math., 1, North-Holland, Amsterdam, 1983.

\bibitem[R15]{Alice}
A.~Rizzardo, \emph{Adjoints to a Fourier--Mukai transform}, {\sf arXiv:1505.01035}.

\bibitem[S87]{Salvetti}
M.~Salvetti, \emph{Topology of the complement of real hyperplanes in $\mathbb{C}^N$}, Invent.\ Math.\ \textbf{88} (1987), no.~3, 603--618.

\bibitem[S01]{Schwede}
S.~Schwede, \emph{The stable homotopy category has a unique model at the prime 2}, Adv.\ Math.\ \textbf{164} (2001), no.~1, 24--40.

\bibitem[ST]{ST}
P.~Seidel and R.~P.~Thomas, \emph{Braid group actions on derived categories of sheaves}, Duke Math.\ J.\ \textbf{108} (2001), 37--108.


\bibitem[S03]{Szendroi}
B.~Szendr\H{o}i, \emph{Enhanced gauge symmetry and braid group actions}, Commun.\ Math.\ Phys.\ \textbf{238} (2003), 35--51.

\bibitem[T07]{Toda}
Y.~Toda, \emph{On a certain generalization of spherical twists},
Bulletin de la Soci\'et\'e Math\'ematique de France \textbf{135}, fascicule 1 (2007), 119--134.

\bibitem[T08]{TodaResPub}
Y.~Toda, \emph{Stability conditions and crepant small resolutions}, Trans.\ Amer.\ Math.\ Soc.\ \textbf{360} (2008), no.~11, 6149--6178.

\bibitem[T13]{TodaExtremal}
Y.~Toda, \emph{Stability conditions and extremal contractions}, Math.\ Ann.\ \textbf{357} (2013), no.~2, 631--685. 

\bibitem[T14]{TodaGV}
Y.~Toda, \emph{Non-commutative width and Gopakumar--Vafa invariants}, {\sf arXiv:1411.1505}.

\bibitem[TU]{TodaUehara}
Y.~Toda and H.~Uehara, \emph{Tilting generators via ample line bundles}, Adv.\ Math.\ \textbf{223} (2010), no.~1, 1--29.


\bibitem[V]{VdB1d}
M.~Van den Bergh, \emph{Three-dimensional flops and noncommutative rings}, 
Duke Math.\ J.\ \textbf{122} (2004), no.~3, 423--455. 


\bibitem[W]{HomMMP}
M.~Wemyss, \emph{Aspects of the homological minimal model program}, {\sf arXiv:1411.7189}.

\bibitem[YZ]{YZ}
A.~Yekutieli and J.~J.~Zhang, \emph{Rings with Auslander dualizing complexes}, J.\ Alg.\ \textbf{213} (1999), no.~1, 1--51.
\end{thebibliography}
\end{document}